\RequirePackage{etex}
\documentclass[10pt]{amsart}
\usepackage{tabularx}
\usepackage{fancyhdr}
\usepackage{float}
\usepackage[margin=1.35 in]{geometry}
\usepackage[english]{babel}
\usepackage[utf8]{inputenc}
\usepackage{subcaption}
\usepackage{amsmath}
\usepackage{amssymb}
\usepackage{amsfonts}
\usepackage{amsthm}
\usepackage{mathdots}
\usepackage{mathrsfs}
\usepackage[all]{xy}
\usepackage[pdftex]{graphicx}
\usepackage{color}
\usepackage{cite}
\usepackage{url}
\usepackage{indent first}
\usepackage[labelfont=bf,labelsep=period,justification=raggedright]{caption}
\usepackage[english]{babel}
\usepackage[utf8]{inputenc}
\usepackage[colorinlistoftodos]{todonotes}
\usepackage{tkz-fct}
\usepackage{tikz}
\usepackage{tikz-3dplot}
\usetikzlibrary{patterns}
\usetikzlibrary{calc}
\usetikzlibrary{knots}
\usetikzlibrary{spath3, intersections, hobby}
\usetikzlibrary{arrows, automata}
\usetikzlibrary{arrows.meta}
\usetikzlibrary{shapes.multipart}
\usepackage{pgfplots}
\usepackage{multicol}
\PassOptionsToPackage{dvipsnames,svgnames}{xcolor}
\usepackage{textcomp}
\usepackage{bbm}
\usepackage{array}
\newcolumntype{C}{>{$}c<{$}}
\usepackage{forest}
\usepackage{tikzsymbols}
\usetikzlibrary{3d}
\tdplotsetmaincoords{80}{110}
\usepackage{worldflags}
\usepackage{youngtab}
\usepackage{tikzducks}
\usepackage{hyperref}

\usepackage{framed}

\makeatletter
\DeclareFontFamily{OMX}{MnSymbolE}{}
\DeclareSymbolFont{MnLargeSymbols}{OMX}{MnSymbolE}{m}{n}
\SetSymbolFont{MnLargeSymbols}{bold}{OMX}{MnSymbolE}{b}{n}
\DeclareFontShape{OMX}{MnSymbolE}{m}{n}{
    <-6>  MnSymbolE5
   <6-7>  MnSymbolE6
   <7-8>  MnSymbolE7
   <8-9>  MnSymbolE8
   <9-10> MnSymbolE9
  <10-12> MnSymbolE10
  <12->   MnSymbolE12
}{}
\DeclareFontShape{OMX}{MnSymbolE}{b}{n}{
    <-6>  MnSymbolE-Bold5
   <6-7>  MnSymbolE-Bold6
   <7-8>  MnSymbolE-Bold7
   <8-9>  MnSymbolE-Bold8
   <9-10> MnSymbolE-Bold9
  <10-12> MnSymbolE-Bold10
  <12->   MnSymbolE-Bold12
}{}

\let\llangle\@undefined
\let\rrangle\@undefined
\DeclareMathDelimiter{\llangle}{\mathopen}%
                     {MnLargeSymbols}{'164}{MnLargeSymbols}{'164}
\DeclareMathDelimiter{\rrangle}{\mathclose}%
                     {MnLargeSymbols}{'171}{MnLargeSymbols}{'171}
\makeatother

\makeatletter
\tikzoption{canvas is xy plane at z}[]{%
  \def\tikz@plane@origin{\pgfpointxyz{0}{0}{#1}}%
  \def\tikz@plane@x{\pgfpointxyz{1}{0}{#1}}%
  \def\tikz@plane@y{\pgfpointxyz{0}{1}{#1}}%
  \tikz@canvas@is@plane
}
\makeatother

\tikzset{surface/.style={draw=black, left color=orange,right color=orange,middle
color=orange!60!#1, fill opacity=1},surface/.default=white}

\pgfplotsset{compat=1.17}

\newcommand{\eps}{\varepsilon}

\newcommand{\ZZ}{\mathbb{Z}}

\def\multiset#1#2{\ensuremath{\left(\kern-.3em\left(\genfrac{}{}{0pt}{}{#1}{#2}\right)\kern-.3em\right)}}

\makeatletter

\newif\ifpgfcirclecrosssplitcustomfill
\tikzset{%
  circle cross split part fill/.code=\def\pgf@lib@sh@ccs@list@fill{#1}\pgfcirclecrosssplitcustomfilltrue,%
  circle cross split uses custom fill/.is if=pgfcirclecrosssplitcustomfill}
\pgfdeclareshape{circle cross split}{%
  \nodeparts{text,two,three,four}%
  \savedanchor\centerpoint{%
    \pgfmathsetlength\pgf@xa{\pgfkeysvalueof{/pgf/inner xsep}}%
    \pgfmathsetlength\pgf@ya{\pgfkeysvalueof{/pgf/inner ysep}}%
    \pgf@x\wd\pgfnodeparttextbox
    \pgf@yb\dp\pgfnodeparttextbox
    \pgf@yc\dp\pgfnodeparttwobox
    \ifdim\pgf@yb>\pgf@yc
      \pgf@yc\pgf@yb
    \fi
    \advance\pgf@y-\pgf@yc
    \advance\pgf@x\pgf@xa
    \advance\pgf@y-\pgf@ya
    \advance\pgf@x.5\pgflinewidth
    \advance\pgf@y-.5\pgflinewidth
  }%
  \savedanchor\twoanchor{%
    \pgfmathsetlength\pgf@xa{\pgfkeysvalueof{/pgf/inner xsep}}%
    \pgfmathsetlength\pgf@ya{\pgfkeysvalueof{/pgf/inner ysep}}%
    \advance\pgf@x.5\pgflinewidth
    \advance\pgf@x\pgf@xa
    \advance\pgf@y.5\pgflinewidth
    \advance\pgf@y\pgf@ya
    \pgf@yb\dp\pgfnodeparttextbox
    \pgf@yc\dp\pgfnodeparttwobox
    \ifdim\pgf@yb>\pgf@yc
      \pgf@yc\pgf@yb
    \fi
    \advance\pgf@y\pgf@yc
  }%
  \savedanchor\threeanchor{%
    \pgfmathsetlength\pgf@ya{\pgfkeysvalueof{/pgf/inner ysep}}%
    \pgf@x\wd\pgfnodeparttextbox
    \pgf@yb\dp\pgfnodeparttextbox
    \pgf@yc\dp\pgfnodeparttwobox
    \ifdim\pgf@yb>\pgf@yc
      \pgf@yc\pgf@yb
    \fi
    \advance\pgf@y-\pgf@yc
    \advance\pgf@y-2\pgf@ya
    \advance\pgf@y-\pgflinewidth
    \pgf@yb\ht\pgfnodepartthreebox
    \pgf@yc\ht\pgfnodepartfourbox
    \ifdim\pgf@yb>\pgf@yc
      \pgf@yc\pgf@yb
    \fi
    \advance\pgf@y-\pgf@yc
    \advance\pgf@x-\wd\pgfnodepartthreebox
  }%
  \savedanchor\fouranchor{%
    \pgfmathsetlength\pgf@xa{\pgfkeysvalueof{/pgf/inner xsep}}%
    \advance\pgf@x\wd\pgfnodepartthreebox
    \advance\pgf@x2\pgf@xa
    \advance\pgf@x\pgflinewidth
  }%
  \saveddimen\radius{%
    \pgf@y\ht\pgfnodeparttextbox
    \pgf@yb\ht\pgfnodeparttwobox
    \ifdim\pgf@yb>\pgf@y
      \pgf@y\pgf@yb
    \fi
    \pgf@yc\dp\pgfnodeparttextbox
    \pgf@yb\dp\pgfnodeparttwobox
    \ifdim\pgf@yc>\pgf@yb
      \advance\pgf@y\pgf@yc
    \else
      \advance\pgf@y\pgf@yb
    \fi
    \pgf@yb\ht\pgfnodepartthreebox
    \ifdim\pgf@yb<\ht\pgfnodepartfourbox
      \pgf@yb\ht\pgfnodepartfourbox
    \fi
    \pgf@yc\dp\pgfnodepartthreebox
    \ifdim\pgf@yc<\dp\pgfnodepartfourbox
      \advance\pgf@yb\dp\pgfnodepartfourbox
    \else
      \advance\pgf@yb\pgf@yc
    \fi
    \ifdim\pgf@yc>\pgf@y
      \pgf@y\pgf@yc
    \fi
    \pgfmathsetlength\pgf@ya{\pgfkeysvalueof{/pgf/inner ysep}}%
    \advance\pgf@y2\pgf@ya
    \pgf@x\wd\pgfnodeparttextbox
    \pgf@xa\wd\pgfnodepartthreebox
    \pgf@xb\wd\pgfnodeparttwobox
    \pgf@xc\wd\pgfnodepartfourbox
    \ifdim\pgf@xa>\pgf@x
      \pgf@x\pgf@xa
    \fi
    \ifdim\pgf@xb>\pgf@x
      \pgf@x\pgf@xb
    \fi
    \ifdim\pgf@xc>\pgf@x
      \pgf@x\pgf@xc
    \fi
    \pgfmathsetlength\pgf@xa{\pgfkeysvalueof{/pgf/inner xsep}}%
    \advance\pgf@x2\pgf@xa
    \ifdim\pgf@y>\pgf@x
      \pgf@x\pgf@y
    \fi
    \advance\pgf@x.5\pgflinewidth
    \pgfmathsetlength{\pgf@xb}{\pgfkeysvalueof{/pgf/minimum width}}%
    \pgfmathsetlength{\pgf@yb}{\pgfkeysvalueof{/pgf/minimum height}}%
    \ifdim\pgf@x<.5\pgf@xb
        \pgf@x=.5\pgf@xb
    \fi
    \ifdim\pgf@x<.5\pgf@yb
        \pgf@x=.5\pgf@yb
    \fi
    \pgfmathsetlength{\pgf@xb}{\pgfkeysvalueof{/pgf/outer xsep}}%
    \pgfmathsetlength{\pgf@yb}{\pgfkeysvalueof{/pgf/outer ysep}}%
    \ifdim\pgf@xb<\pgf@yb
      \advance\pgf@x\pgf@yb
    \else
      \advance\pgf@x\pgf@xb
    \fi
  }%
  \inheritanchorborder[from=circle]%
  \inheritanchor[from=circle]{north}%
  \inheritanchor[from=circle]{north west}%
  \inheritanchor[from=circle]{north east}%
  \inheritanchor[from=circle]{center}%
  \inheritanchor[from=circle]{west}%
  \inheritanchor[from=circle]{east}%
  \inheritanchor[from=circle]{mid}%
  \inheritanchor[from=circle]{mid west}%
  \inheritanchor[from=circle]{mid east}%
  \inheritanchor[from=circle]{base}%
  \inheritanchor[from=circle]{base west}%
  \inheritanchor[from=circle]{base east}%
  \inheritanchor[from=circle]{south}%
  \inheritanchor[from=circle]{south west}%
  \inheritanchor[from=circle]{south east}%
  \anchor{two}{\twoanchor}%
  \anchor{three}{\threeanchor}%
  \anchor{four}{\fouranchor}%
  \inheritbackgroundpath[from=circle]
  \beforebackgroundpath{%
    \pgfutil@tempdima=\radius
    \pgfmathsetlength{\pgf@xb}{\pgfkeysvalueof{/pgf/outer xsep}}%
    \pgfmathsetlength{\pgf@yb}{\pgfkeysvalueof{/pgf/outer ysep}}%
    \ifdim\pgf@xb<\pgf@yb
      \advance\pgfutil@tempdima by-\pgf@yb
    \else
      \advance\pgfutil@tempdima by-\pgf@xb
    \fi
    \advance\pgfutil@tempdima by-.5\pgflinewidth%
    \pgfsetshortenstart{0pt}%
    \pgfsetshortenend{0pt}%
    \pgfsetarrows{-}%
    \pgfpathmoveto{\pgfpointadd{\centerpoint}{\pgfqpoint{-\pgfutil@tempdima}{0pt}}}%
    \pgfpathlineto{\pgfpointadd{\centerpoint}{\pgfqpoint{\pgfutil@tempdima}{0pt}}}%
    \pgfpathmoveto{\pgfpointadd{\centerpoint}{\pgfqpoint{0pt}{-\pgfutil@tempdima}}}%
    \pgfpathlineto{\pgfpointadd{\centerpoint}{\pgfqpoint{0pt}{\pgfutil@tempdima}}}%
    \pgfusepathqstroke
  }%
  \behindbackgroundpath{%
    \pgfutil@tempdima=\radius
    \pgfmathsetlength{\pgf@xb}{\pgfkeysvalueof{/pgf/outer xsep}}%
    \pgfmathsetlength{\pgf@yb}{\pgfkeysvalueof{/pgf/outer ysep}}%
    \ifdim\pgf@xb<\pgf@yb
      \advance\pgfutil@tempdima by-\pgf@yb
    \else
      \advance\pgfutil@tempdima by-\pgf@xb
    \fi
    \advance\pgfutil@tempdima by-.5\pgflinewidth%
    \ifpgfcirclecrosssplitcustomfill%
      \pgf@lib@sh@rs@process@list{\pgf@lib@sh@ccs@list@fill}{4}%
      {%
        \pgfmathloop
           \ifnum\pgfmathcounter>4%
           \else%
             \pgf@lib@sh@getalpha\pgf@lib@sh@rs@number{\pgfmathcounter}%
              \edef\pgf@tempa{\csname pgf@lib@sh@rs@\pgf@lib@sh@rs@number @item\endcsname}%
              \ifx\pgf@tempa\pgf@lib@sh@rs@nonetext\else
                \pgfsetfillcolor{\pgf@tempa}%
                \pgf@lib@sh@ccs@angles{\pgfmathcounter}%
                \pgfpathmoveto{\centerpoint}%
                \pgfpathlineto{\pgfpointadd{\centerpoint}{\pgfqpointpolar{\pgf@lib@sh@ccs@angle}{\pgfutil@tempdima}}}%
                \pgfpatharc{\pgf@lib@sh@ccs@angle}{\pgf@lib@sh@ccs@angle@}{\pgfutil@tempdima}%
                \pgfpathclose
                \pgfusepathqfill
              \fi
        \repeatpgfmathloop
      }%
    \fi%
  }%
}
\def\pgf@lib@sh@ccs@angles#1{%
  \ifcase#1\or\def\pgf@lib@sh@ccs@angle{90}%
           \or\def\pgf@lib@sh@ccs@angle{0}%
           \or\def\pgf@lib@sh@ccs@angle{180}%
           \else\def\pgf@lib@sh@ccs@angle{270}%
  \fi
  \edef\pgf@lib@sh@ccs@angle@{\number\numexpr\pgf@lib@sh@ccs@angle+90\relax}%
}
\makeatother

\theoremstyle{plain}
\newtheorem{thm}{Theorem}
\newtheorem{lemma}[thm]{Lemma}
\newtheorem{cor}[thm]{Corollary}

\newtheorem{prop}[thm]{Proposition}

\theoremstyle{definition}
\newtheorem{defn}{Definition}

\theoremstyle{remark}
\newtheorem{rem}{Remark}

\numberwithin{equation}{section}
\numberwithin{thm}{section}
\usepackage{cleveref}

\title{On Finite Fields and Higher Reciprocity}
\author{Matias Carl Relyea}
\date{Summer 2024}

\begin{document}

\begin{abstract}
    Cubic and biquadratic reciprocity have long since been referred to as ``the forgotten reciprocity laws", largely since they provide special conditions that are widely considered to be unnecessary in the study of number theory. However, this paper aims to approach reciprocity with ample detail to motivate its existence. In this exposition of finite fields and higher reciprocity, we will begin by introducing concepts in abstract algebra and elementary number theory. This will motivate our approach toward understanding the structure and then existence of finite fields, especially with a focus on understanding the multiplicative group $\mathbb{F}^{*}$. While surveying finite fields we will provide another proof of quadratic reciprocity. We will proceed to investigate properties of the general multiplicative character, covering the concept of a general Gauss sum as well as basic notions of the Jacobi sum. From there we will begin laying the foundations for the cubic reciprocity law, beginning with a classification of the primes and units of the Eisenstein integers, denoted $\mathbb{Z}[\omega]$, and further investigations into the residue class ring $\mathbb{Z}[\omega]/\pi\mathbb{Z}[\omega]$ for $\pi$ prime, which is predominantly the world in which cubic reciprocity lies. 

    \indent We will then use multiplicative characters to define the cubic residue character and state cubic reciprocity in its entirety. Following this, we provide a proof of the cubic reciprocity law as well as its supplementary theorems using cubic Gauss sums. We will finish the section on cubic reciprocity with a brief survey of the cubic residue character of the even prime $2$ and state a significant result due to Gauss that summarizes the conditions for $2$ to be a cubic residue. 

    \indent We conclude with the statement of biquadratic reciprocity and provide a brief discussion on how it relates to cubic reciprocity in both its proof and usage of the analogy between the Eisenstein integers, $\ZZ[\omega]$, and the Gaussian integers, $\ZZ[i]$. 
\end{abstract}

\maketitle

\tableofcontents

\section*{Introduction}
\noindent Reciprocity laws have been studied for over 200 years. As of now, there exist over 300 proofs of quadratic reciprocity, but there exist far fewer proofs of cubic and biquadratic reciprocity. Quadratic reciprocity, the first reciprocity law, was proven by Carl Friedrich Gauss in 1796, and he deemed it the ``Theorema Aureum”, or ``Golden Theorem”. Gauss went on to prove quadratic reciprocity in 6 different ways, with 2 more posthumously. Unlike other techniques, one particular technique used, which certainly wouldn’t have been named what it is named now, was quadratic Gauss sums. The formulation of quadratic Gauss sums would become the first step toward investigating higher reciprocity laws. Gotthold Eisenstein first published his proof of cubic reciprocity in 1844, and it used the techniques of Gauss and Jacobi sums. In 1850, Eisenstein published his paper on generalized higher reciprocity, a result now known as Eisenstein reciprocity. Even though Eisenstein reciprocity eventually became a direct corollary to work completed by Emil Artin on higher reciprocity using class field theory in the earlier 20th century, it became the first formal law for reciprocity of odd primes. Though Eisenstein reciprocity is a generalisation of cubic and biquadratic reciprocity and is highly relevant to modern research regarding reciprocity laws and algebraic number theory, it lies beyond the scope of this paper. 

\indent Quadratic reciprocity asks the question: under what conditions does the congruence $x^{2}\equiv a\pmod{p}$ have solutions? Cubic reciprocity asks a similar question: under what conditions does the congruence $x^{3}\equiv a\pmod{p}$ have solutions? The difference is subtle, but cubic reciprocity demands significantly more mechanics, and this is primarily what we will address in this paper. 

\indent This paper seeks to explore 3 reciprocity laws: a proof of quadratic reciprocity, a more familiar result in the context of finite fields; a proof of cubic reciprocity; and finally the statement of the law of biquadratic reciprocity. The preliminary section introduces necessary technical and conceptual preliminaries. Section 1 defines, states, and proves facts concerning the structure and existence of finite fields. Section 2 focuses primarily on motivating the study of multiplicative characters alongside Gauss and Jacobi sums. Section 3 states and proves cubic reciprocity and provides a sketch for the cubic character of 2. Section 4 concludes the paper with an overview of the different components needed for the proof of biquadratic reciprocity.  

\indent As we see in section 1, finite fields play an important role in higher reciprocity. In this paper, we survey the finite field $\mathbb{F}$ of order $p$ and its multiplicative subgroup $\mathbb{F}^{*}$ of order $p-1$. While the construction and existence of this finite field are important facts of theory covered in section 1, we use them primarily to show that a residue class ring involving the Eisenstein integers $\mathbb{Z}[\omega]$ and an element of $\ZZ[\omega]$, introduced in section 3, are a finite field. This forms an important connection between algebra and cubic reciprocity and uses objects such as associates, norms, etc. that we are already familiar with. Since cubic reciprocity is considered over the finite field $\mathbb{Z}[\omega]/\pi\mathbb{Z}[\omega]$ for $\pi$ a prime in $\mathbb{Z}[\omega]$, it is only logical that biquadratic reciprocity will be considered over the finite field $\mathbb{Z}[i]/\pi\mathbb{Z}[i]$ for $\pi$ a prime in $\mathbb{Z}[i]$, where $\mathbb{Z}[i]$ is the Gaussian integers. We will omit many of the details relating to Gaussian integers as it lies beyond the general scope of this paper, but \cite{ireland1990classical} gives a strict yet enlightening overview of its intricacies. 

\indent Though this paper aims to introduce finite fields and higher reciprocity to a fairly new reader, it also assumes a certain degree of abstract mathematics knowledge. We assume fluency with elementary number theory and most of algebra with the exception of some facts from ring theory that are introduced when needed. \cite{Gallian_2002} is an excellent resource for understanding and gaining useful insight on the group-theoretic and ring-theoretic portions of this paper. We also assume familiarity with the definitions of the algebraic integers and algebraic numbers as well as their algebraic structures; for instance, that the algebraic integers form a ring and that the algebraic numbers form a field. Everything in relation to finite fields is constructed from elementary principles with the exception of some facts and definitions about fields. As stated in the abstract, immediately following finite fields - the backbone of much of what we will do here - we will survey multiplicative characters and their uses, especially in the context of cubic residue characters. This leads naturally into an elegant proof of cubic reciprocity after surveying the Eisenstein integers.

\indent This expository work is a continuation of previous work done in \cite{relmat2022quadrec} related to quadratic reciprocity. For more insightful information about the history of early higher reciprocity, \cite{collison1977origin} is an excellent introduction. As mentioned earlier, \cite{ireland1990classical} is an extensive text not only for higher reciprocity but for algebraic number theory as well, and contains a proof for Eisenstein reciprocity for interested and advanced readers. \cite{Rousseau2012reciprocity} provides even more detailed proofs for cubic reciprocity and its supplements. Many algebraic results that we do not prove may be found in \cite{Gallian_2002}, and even so it is a wonderful text to gain further insight into finite fields. 

\section*{Preliminaries}

\subsection{Algebra}
\noindent Much of the algebra used in this paper is self-contained, in that any algebraic definitions or theorems used are largely either stated or proven before they are needed. Readers who are unfamiliar with more basic algebra, for instance that of groups, may find the opening chapters of \cite{Gallian_2002} illuminating. Many later chapters of this textbook also cover the theory of finite fields in more detail and are interesting in their own right. Knowledge of ring theory is also assumed, but we will state and prove several of the larger results when they are needed. 

\begin{defn}[Homomorphism]
    We define a \textit{homomorphism} to be a structure preserving mapping from one set to another set of the same type. In other words, if $A$ and $B$ are two sets of the same type, then under the mapping  $\phi:A\rightarrow B$, it is true for all $(x,y)\in A$ that 
    \begin{equation*}
        \phi(xy) = \phi(x)\phi(y). 
    \end{equation*}
\end{defn}
\noindent In this paper we will be utilising multiple types of homomorphisms, namely \textit{group homomorphisms} and \textit{ring homomorphisms}, which in practice are functionally identical. Each homomorphism also possesses a \textit{kernel}. 
\begin{defn}[Kernel of a Homomorphism]
    Let $\phi$ be a homomorphism defined with $\phi: A \rightarrow B$, and let $B$ have some identity element $e$. We define the \textit{kernel} of $\phi$ to be 
    \begin{equation*}
        \text{Ker}(\phi) = \{ x\in A | \phi(x)=e \}.
    \end{equation*}
    \noindent In other words, the kernel of $\phi$ is the set of elements in $A$ that map to the identity element of $B$. 
\end{defn}
\noindent Much of the results that we are concerned with utilize ring theory, so we will define some basic objects. Ideals are analogous to normal subgroups in group theory. In the following definitions we let $R$ be a commutative ring. 
\begin{defn}[Associate]
     We say two members of $R$ are \textit{associate} if for $r,s\in R$ there exists some unit $u\in R$ such that $r=us$. In this case we say that $r$ is associate to $s$. 
\end{defn}
\begin{defn}[Ideal]
    We define an \textit{ideal} to be a subring $\overline{R}\subset R$ such that for every $r\in R$ and every $a\in \overline{R}$, both $ar,ra\in \overline{R}$. 
\end{defn}

\begin{defn}[Maximal Ideal]
    Some ideal $\overline{R}$ is a \textit{maximal ideal} if whenever there is some ideal $B$ of $R$ with $\overline{R}\subseteq B \subseteq R$, either $B=\overline{R}$ or $B=R$. In other words, it is the largest ideal of $R$ that is not $R$. An ideal of $R$ is \textit{proper} if it is not the entirety of $R$. 
\end{defn}

\begin{defn}[Prime Ideal]
    Some ideal $\overline{R}$ is a \textit{prime ideal} if for the product $ab\in R$, then either $a$ or $b$ is in $\overline{R}$. 
\end{defn}
\noindent A prime ideal is in fact a ring-theoretic analogue to Euclid's Lemma. 
\begin{defn}[Principal Ideal]
    Some ideal $\overline{R}$ is a \textit{principal ideal} if it can be generated by a singular element. 
\end{defn}

\begin{defn}[Integral Domain]
    We define an \textit{integral domain} to be a nonzero commutative ring. 
\end{defn}
\noindent Integral domains are generalizations of the ring of integers, and the product of any two members yields a nonzero output.
\begin{defn}[PID]
    An integral domain is a \textit{principal ideal domain}, conventionally a \textit{PID}, if every proper ideal of the integral domain is principal. 
\end{defn}

\begin{defn}[Euclidean Domain]
    Some integral domain $R$ is a Euclidean domain if there exists some function $\lambda$, known as the \textit{norm}, that maps the nonzero elements of $R$ to $\mathbb{Z}_{\geq 0}$ such that if there exist some $a,b\in R$ with $b\neq 0$, then there also exist some $c,d\in R$ with the property that $a=cb+d$ and either $d=0$ or $\lambda(d)<\lambda(b)$.
\end{defn}
\noindent Essentially, a Euclidean domain asserts the existence of a division algorithm (more specifically, the Euclidean algorithm) over a ring. 
\begin{rem}
\noindent It can be seen that $k[x]$ for some field $k$ is a Euclidean domain, as we can map polynomials from $k[x]$ to the degrees of polynomials from $\mathbb{Z}_{\geq 0}$ via the norm. 
\end{rem}

\begin{defn}[UFD]
    An integral domain $R$ is a \textit{unique factorization domain}, conventionally \textit{UFD}, if every nonzero element of $x$ in $R$ may be written as a product of some unit $u\in R$ and some finite number of irreducible elements $p_{i}$ as follows:
    \begin{equation*}
        x = up_{1}p_{2}\cdots p_{n}.
    \end{equation*}
    \noindent Furthermore, this representation is unique in that any other $x$ that can represented in the same way must have a bijection between its irreducible elements and $p_{i}$ with $1\leq i\leq n$. 
\end{defn}

\begin{rem} \label{Remark: class inclusions}
    Integral domains, Euclidean domains, PIDs, and UFDs are ultimately what we will use to study finite fields later in section 1. A diagram of their class inclusions may be described as follows.
\begin{figure}[H]
\begin{framed}
\begin{center}
field $\subset$ Euclidean domain $\subset$ PID $\subset$ UFD $\subset$ integral domain
\end{center}
\end{framed}
\caption{Diagram of class inclusions for algebraic structures}
\end{figure}
    \noindent Class inclusions are useful because they allow us to make statements about complex algebraic structures when working with simpler algebraic structures.
    
    \indent We show in Theorem \ref{Every Euclidean domain is a PID} that every Euclidean domain is a PID. It is also possible to show that every PID is a UFD and that every UFD is an integral domain. We will not do this, but in the case of $\mathbb{Z}[\omega]$, the Eisenstein integers, we note that since it is a UFD it is also an integral domain. 
\end{rem} 
\noindent We define a field more rigorously in section 1. 

\begin{thm}[Every Euclidean domain is a PID] \label{Every Euclidean domain is a PID}
    If $R$ is an integral domain and $I\subseteq R$ is an ideal, then there exists some element $a\in R$ such that $I=Ra=aR=\{ ra=ar | r\in R \}$. This is necessarily the requirement to be a PID.
\end{thm}
\begin{proof}
    We proceed by double inclusion. First consider the set of nonnegative integers given by $\{ \lambda(b) | b\in I, b\neq 0 \}$. By well-ordering, every nonnegative set of integers must have a lowest term; we call this term $a\in I,a\neq 0$ with the property that $\lambda(a)\leq \lambda(b)$ for all $b\in I,b\neq 0$. We claim that $I=Ra=aR$, namely that the ideal $I$ is generated by $a$. By the definition of $I$, we know that $Ra=aR\subseteq I$. We want to show that $I\subseteq Ra=aR$. To begin, we know that since $I$ is a subring of $R$, it retains the properties of $R$ and is thus a Euclidean domain. Namely, for any $b\in I$, there exist $c,d\in R$ such that $b=ca+d$ with either $d=0$ or $\lambda(d)<\lambda(a)$. Clearly $d=b-ca\in I$, so it is not possible for $\lambda(d)<\lambda(a)$. Therefore $d=0$, and so $b=ca$. Then $b=ca\in I$. Since we showed this for any $b$, it follows that $I=Ra=aR$, and thus every Euclidean domain is a PID. 
\end{proof}
\noindent In developing finite fields in section 1 we will need the following results.

\begin{thm}[Lagrange's Theorem; Gallian]\label{Lagrange's Theorem}
    Let $G$ be a finite group and let $H$ be a subgroup of $G$. Then the order of $H$ divides the order of $G$. Furthermore, there are exactly $|G|/|H|$ distinct left or right cosets of $H$ in $G$. 
\end{thm}
\begin{proof}
    The proof may be found in Chapter 7 of \cite{Gallian_2002}. 
\end{proof} 

\noindent An important result that is necessary to prove Proposition \ref{an important isomoprhism by the First Isomorphism Theorem} is the First Isomorphism Theorem, alternatively referred to as the Fundamental Theorem of Group Homomorphisms. The theorem can easily be extended to rings for our purposes, but we will only prove the theorem for groups.
\begin{thm}[First Isomorphism Theorem; Gallian] \label{First Isomorphism Theorem}
    Let $\phi$ be a group homomorphism from a group $G$ to $\overline{G}$. Then the mapping from the quotient group $G/\text{Ker}(\phi)$ to the image of $\phi$ given as $\phi(G)$ is an isomorphism, i.e.
    \begin{equation*}
        G/\text{Ker}(\phi) \approx \phi(G). 
    \end{equation*}
\end{thm}
\noindent Before we can prove this theorem, let us first examine what the different components are. We will first examine the quotient group $G/\text{Ker}(\phi)$. This set is defined as $\{ gH | g\in\text{Ker}(\phi) \}$ for $H$ a normal subgroup of $G$, so we can write it as $g\text{Ker}(\phi)$ for all $g\in G$. The set $\phi(G)$ is also known as the image of $G$ or $\text{im}(\phi)$. If we consider these simplifications, then we can write that the mapping is now defined as $g\text{Ker}(\phi) \rightarrow \phi(g)$ for all $g\in G$. Now we proceed with the proof.
\begin{proof}[Proof of Theorem \ref{First Isomorphism Theorem}]
    For the sake of convenience, we will use the function $\psi$ to denote the mapping $g\text{Ker}(\phi) \rightarrow \phi(g)$. An isomorphism requires that the mapping between two groups preserve group operations and is one-to-one (namely, an injective function, so we necessarily need to show that the homomorphism is a function). To begin, we must first show that $\psi$ is well-defined, or that for any $g$, the \textit{LHS} of the mapping remains unique: that is, $g$ is the only such coset representative that generates the coset. Suppose that there exist $x,y$ such that $x\text{Ker}(\phi)=y\text{Ker}(\phi)$. Then, multiplying both sides by $y^{-1}$ we have that $y^{-1}x\in \text{Ker}(\phi)$. Then, by the definition of the kernel, we have $e=\phi(y^{-1}x)=\phi(y^{-1})\phi(x)=(\phi(y))^{-1}\phi(x)$. Multiplying both sides by $\phi(y)$, this means that $\phi(x)=\phi(y)$, so we have shown that $\psi$ is well-defined and $\psi$ is a function. Next we need to show that $\psi$ preserves operations. Notice that 
    \begin{equation*}
        \psi(x\text{Ker}(\phi)y\text{Ker}(\phi))=\psi(xy\text{Ker}(\phi))=\phi(xy)=\phi(x)\phi(y)=\psi(x\text{Ker}(\phi))\psi(y\text{Ker}(\phi)),
    \end{equation*}
    \noindent which indeed shows that the group operation is preserved. Finally, we need to show that $\psi$ is one-to-one. Note that $\psi(g_{1}\text{Ker}(\phi))=\psi(g_{2}\text{Ker}(\phi))\implies \phi(g_{1})=\phi(g_{2})$. Multiplying both sides by $\phi(g_{2})^{-1}$ we have $e=(\phi(g_{2})^{-1}\phi(g_{1})=\phi(g_{2}^{-1})\phi(g_{1})=\phi(g_{2}^{-1}g_{1})$. Therefore, by the definition of the kernel, $g_{2}^{-1}g_{1}\in \text{Ker}(\phi)$. Hence $g_{1}\text{Ker}(\phi)=g_{2}\text{Ker}(\phi)$, proving that $\psi$ is one-to-one. Thus the mapping $\psi$ is an isomorphism. 
\end{proof}

\noindent To conclude this subsection, we include some definitions in elementary field theory that will be useful to the language we use in our investigation of the existence of finite fields. We formally define a field in section 1.

\begin{defn}[Field Extension]
    Let $K$ and $L$ be fields such that $K\subseteq L$ is a subfield of $L$. We define a field extension $K$ of $L$, which we denote as $L/K$, to be a field such that $K$ is a subfield of $L$. In this way, $L$ is referred to as a $K$\textit{-vector space} as it forms a vector space over the scalar field $K$. 
\end{defn}

\noindent We might say that $L$ is a field extension, or simply \textit{extension}, of $K$. A useful concept is the idea of an \textit{intermediate field extension}. If $L$ is an extension of $F$, and $F$ is an extension of $K$, then $F$ is an intermediate field extension. We now have a definition for the degree of a field extension. Let $K$ and $L$ be the same fields.

\begin{defn}[Degree of a Field Extension]
    We define the degree of a field extension, denoted $[K:L]$, to be the dimension of the vector space $L$ over its scalar field $K$. 
\end{defn}

\subsection{Elementary Number Theory} 
\noindent In this paper we assume a general knowledge of elementary number theory, including results such as Bézout's Lemma, Fermat's Little Theorem, quadratic reciprocity, and including other results concerning quadratic residues and nonresidues. Some specific concepts such as primitive roots and units will be introduced as needed. Let $(a/p)$ denote the Legendre symbol. 
\begin{lemma} \label{Legendre symbol theorem 1}
Let $\gcd(a,p)=1$ and $a,b\in \ZZ$ for $p$ prime. Then
    \begin{enumerate}
        \item
            \begin{equation*}
                a\equiv b \pmod{p} \iff \bigg(\frac{a}{p}\bigg)=\bigg(\frac{b}{p}\bigg)
            \end{equation*}
        \item
            \begin{equation*} 
                \bigg(\frac{0}{p}\bigg)=0
            \end{equation*}
        \item
            \begin{equation*} 
                \bigg(\frac{a^{2}}{p}\bigg)=1.
            \end{equation*} 
    \end{enumerate}
\end{lemma}
\begin{proof}
    The proof of this may be found in the preliminary section of \cite{relmat2022quadrec}. Everything follows through the definition of the Legendre symbol. 
\end{proof}
\begin{lemma} \label{Legendre symbol theorem 2}
Let $p$ prime, $\gcd(a,p)=1$, and $a,b\in \ZZ$. Then
\begin{enumerate}
    \item
            \begin{equation*}
                a^{\frac{p-1}{2}}\equiv \bigg(\frac{a}{p}\bigg) \pmod{p}
            \end{equation*}
    \item     
            \begin{equation*}
                \bigg(\frac{a}{p}\bigg)\bigg(\frac{b}{p}\bigg)=\bigg(\frac{ab}{p}\bigg). 
            \end{equation*}
\end{enumerate}
\end{lemma}
\begin{proof}
    The proof of this may be found in the preliminary section of \cite{relmat2022quadrec}.
\end{proof}
\noindent One result that will be of utmost use to us in this paper is the law of quadratic reciprocity. Though it ultimately resides in elementary number theory, it is a stepping stone for higher reciprocity.
\begin{thm}[The Law of Quadratic Reciprocity] \label{the law of quadratic reciprocity}
    Let $p,q\in\mathbb{Z}$ be odd primes. Then
        \begin{equation*}
            \bigg(\frac{p}{q}\bigg)\bigg(\frac{q}{p}\bigg)=(-1)^{\frac{p-1}{2}\cdot\frac{q-1}{2}}.
        \end{equation*}
    \noindent Alternatively, we can express this as 
    \begin{equation*}
\bigg(\frac{p}{q}\bigg)= \left\{
        \begin{array}{ll}
            (\frac{q}{p}) & \text{if $p\equiv 1 \pmod{4}$ or $q\equiv 1 \pmod{4}$} \\
            -(\frac{q}{p}) & \text{if $p\equiv 3 \pmod{4}$ and $q\equiv 3 \pmod{4}$}.
        \end{array}
    \right.
\end{equation*}
\end{thm}
\begin{proof}
    Several proofs of this may be found in \cite{relmat2022quadrec} as well as Chapter 5 of \cite{ireland1990classical}. We will also provide a proof for this result using finite fields in section 1 due to Hausner. 
\end{proof}
\noindent There is also a supplement to quadratic reciprocity concerning whether $-1$ or $2$ is a quadratic residue or nonresidue modulo $p$. In elementary number theory, the case of $(-1/p)$ is known as Euler's Criterion; when considering cubic residues, the case of $-1$ is trivial as $-1^{3}=-1$, implying that it is always a cubic residue. The case of $2$ will provide useful insights when considering whether $2$ is a cubic residue or nonresidue in a similar sense.

\begin{thm}[Supplement to Theorem \ref{the law of quadratic reciprocity}] \label{supplement to quadratic reciprocity}
    Let $p$ be an odd prime. Then
    \begin{enumerate}
        \item \begin{equation*}
            \bigg( \frac{-1}{p}\bigg) = (-1)^{\frac{p-1}{2}},
                \end{equation*}
        \item \begin{equation*}
            \bigg(\frac{2}{p}\bigg) = (-1)^{\frac{p^{2}-1}{8}}.
                \end{equation*}
    \end{enumerate}
\end{thm}
\begin{proof}
    The proof for (1) follows immediately by letting $a=-1$ in Lemma \ref{Legendre symbol theorem 2}. The proof for (2) may be found in sections 2.2 and 4.2 of \cite{relmat2022quadrec}. 
\end{proof}

\noindent As we will investigate later in section 2, Gauss sums generalize the notion of a quadratic Gauss sum to be expressed in terms of a multiplicative character of higher degree. In \cite{relmat2022quadrec}, we developed notions of quadratic Gauss sums in order to prove quadratic reciprocity. As we will see later in section 1.5, another proof of quadratic reciprocity can be given by combining the theory of finite fields and quadratic Gauss sums. As such, we state some elementary properties of the quadratic Gauss sum that will be useful later. 
\begin{defn}[Quadratic Gauss sum]
    \begin{equation*}
        g_{a} = \sum_{t}\bigg( \frac{t}{p} \bigg)\zeta_{p}^{at},
    \end{equation*}
    \noindent where $\zeta_{p}$ is a $p$th root of unity.
\end{defn}
\noindent For the sake of notational convention, we denote the quadratic Gauss sum when $a=1$, or $g_{1}$, as simply $g$. Proofs for the following identities may be found in section 4.3 of \cite{relmat2022quadrec}.
\begin{prop} \label{property 1 of quadratic Gauss sums}
    \begin{equation*}
        g_{a} = \bigg( \frac{a}{p} \bigg)g.
    \end{equation*}
\end{prop}
\begin{prop} \label{property 2 of quadratic Gauss sums}
    \begin{equation*}
        g^{2} = (-1)^{\frac{p-1}{2}}p.
    \end{equation*}
\end{prop}
\noindent A useful tool we will use later is the Kronecker delta, $\delta(x,y)$, which is defined to be $1$ if $x\equiv y\pmod{p}$ and $0$ otherwise. 

\subsection{Möbius Inversion} A useful tool that will be used throughout this paper is the Möbius function and a powerful result known as Möbius inversion. Let $n\in\mathbb{Z}^{+}$. 
\begin{defn}[Möbius function]
\begin{equation*}
        \mu(n) = \left\{
                \begin{array}{lll}
                    0 & \text{if} \ n \ \text{non-squarefree}, \\
                    1 & \text{if} \ n=1, \\
                    (-1)^{k} & \text{if} \ n=p_{1}p_{2}\cdots p_{k} \ \text{for distinct} \ p_{i}.
                \end{array}
                 \right.
\end{equation*}
\end{defn}
\noindent An important property of the Möbius function is that it is multiplicative, or $\mu(mn)=\mu(m)\mu(n)$, the proof of which requires that we show it holds for $m=n=1$ and then show that it also holds for $m$ and $n$ squarefree. There are many other interesting properties of the Möbius function, but we are most interested in Möbius inversion. We first need the following result.
\begin{lemma}
    Consider the summatory function $F(n)=\sum_{d|n}\mu(d)$. Then
    \begin{equation*}
        F(n) = \left\{
                \begin{array}{ll}
                    1 & \text{if} \ n=1, \\
                    0 & \text{if} \ n>1.
                \end{array}
                 \right.
\end{equation*}
\end{lemma}
\begin{proof}
    The result is obvious for $n=1$. To prove it for $n>1$, prove that it holds for $p^{k}$ for some $k>0$, and then use the multiplicativity of $\mu(n)$ to prove the result for any integer $n=p_{1}^{a_{1}}\cdots p_{t}^{a_{t}}>1$ using the information gained from computing $F(p^{k})$. 
\end{proof}

\noindent Recall that an arithmetic function, or number-theoretic function, is a function that maps $\ZZ$ to $\ZZ$. We now introduce Möbius inversion. 

\begin{thm}[Möbius Inversion]\label{Theorem: Mobius Inversion}
    Suppose that $f$ is an arithmetic function and that $F$ is its summatory function. Then
    \begin{equation*}
        f(n) = \sum_{d|n}\mu(d)F\bigg(\frac{n}{d}\bigg).
    \end{equation*}
\end{thm} 
\begin{proof}
    \begin{align*}
        \sum_{d|n}\mu(d)F\bigg(\frac{n}{d}\bigg) & = \sum_{d|n}\mu(d)\sum_{e|\frac{n}{d}}f(e) &&& \text{(By definition of $F$.)} \\ & 
        = \sum_{d|n}\sum_{e|\frac{n}{d}}\mu(d)f(e) &&& \text{(By combining double sums.)} \\ & 
        = \sum_{e|n}\sum_{d|\frac{n}{e}}\mu(d)f(e) &&& \text{(By divisibility in indices.)}\\ &
        = \sum_{e|n}f(e)\sum_{d|\frac{n}{e}}\mu(d) &&& \text{(By rearranging the double sums.)}
    \end{align*}
\noindent Notice that $\sum_{d|\frac{n}{e}}\mu(d)=0$ since $n/e>1$. If we allow $n/e=1$, then $n=e$, so $\sum_{d|\frac{n}{e}}\mu(d)=1$. Then
    \begin{equation*}
        \sum_{e|n}f(e)\sum_{d|\frac{n}{e}}\mu(d) = f(n)(1)=f(n).
    \end{equation*}
\end{proof}
\noindent The importance of Möbius inversion is that it allows us to form an algebraic relationship between arithmetic functions and their summatory functions. This will become evident in the next section and following sections. 

\section{Finite Fields}
\noindent The language of cubic, biquadratic, and higher reciprocity is expressed in the language of finite fields, so naturally we will explore notions of a finite field in regard to both construction and existence as well as classification of its elements. The reason that finite fields are so fundamental to the study of reciprocity will be evident later. 

\indent Finite fields may also be referred to as ``Galois fields" as they were created by Évariste Galois and are used widely in Galois theory and higher reciprocity. In this paper, we refer to such a field as a finite field. 

\indent A field is defined as follows. 
\begin{defn}
    We say that a set $\mathbb{F}$ is a field if it has two operations $+$, or addition (sometimes denoted as $\oplus$), and $*$, or multiplication, and it satisfies the following axioms:
    \begin{enumerate}
        \item $\mathbb{F}$ is an abelian group under $\oplus$ with identity element $0$,
        \item the multiplicative set $\mathbb{F}^{*}=\mathbb{F}/\{ 0 \}$ is an abelian group under $*$ with identity element $1$,
        \item and it satisfies the distributive law that $\forall a,b,c\in\mathbb{F}, (a\oplus b)*c = (a*c)\oplus(b*c)$. 
    \end{enumerate}
\end{defn}
\noindent Naturally, we can define a finite field to be one such field with a finite number of elements, say $q$. Then, since we exclude the additive identity, the multiplicative group $\mathbb{F}^{*}$ has $q-1$ elements. Therefore every element $\alpha\in \mathbb{F}^{*}$ satisfies the relation $\alpha^{q-1}=1$. Similarly, every element $\beta \in \mathbb{F}^{+}$ the additive group satisfies the relation $\beta^{q}=\beta+\cdots +\beta=\beta$. In either case, $\alpha$ behaves like a generator of the multiplicative or additive group, but we only consider the multiplicative group of a finite field in this paper. 

\subsection{The Multiplicative Group of a Finite Field is Cyclic}
\indent We denote the finite field of single-variable polynomials in $x$ as $\mathbb{F}[x]$. 
\begin{prop} \label{factorization of x^q-x}
\noindent Suppose that $\mathbb{F}$ is a finite field of order $q$. Then 
    \begin{equation*}
        x^{q}-x = \prod_{\alpha\in\mathbb{F}}(x-\alpha).
    \end{equation*}
\end{prop}
\begin{proof}
    By the construction of $\mathbb{F}$, notice that every element $\alpha\in\mathbb{F}$ is a root of $x^{q}-x$ by the definition of $\alpha$ as a generator of $\mathbb{F}$. Since the polynomial on the \textit{RHS} runs through all $q$ elements of the additive group $\mathbb{F}$, its maximum degree must be $q$. Therefore the result follows from the factorization of the \textit{LHS}. 
\end{proof}
\noindent From this result we can prove the following about subfields. 
\begin{cor} \label{first corollary to first proposition in section about finite fields}
    Let $\mathbb{K}$ be a field with $\mathbb{F}\subset \mathbb{K}$ a subfield of $\mathbb{K}$. An element $\alpha\in\mathbb{K}$ is also contained within the subfield $\mathbb{F}$ if and only if $\alpha^{q}=\alpha$. 
\end{cor}
\begin{proof}
    By our original construction, any root $\alpha\in\mathbb{F}$ of $x^{q}-x$ must satisfy the relation $\alpha^{q}=\alpha$. By Proposition \ref{factorization of x^q-x}, the roots of some polynomial $x^{q}-x$ are exactly the elements of $\mathbb{F}$ (by its construction from $\alpha\in\mathbb{F}$), so we have proven the forward direction. We now prove the backward direction. If $\alpha^{q}=\alpha$, then $\alpha$ must be a root of $x^{q}-x$ by our original construction. Since the condition for an element to be a root of $x^{q}-x$ is for it to be contained within $\mathbb{F}$, we have $\alpha\in\mathbb{F}$, which proves the result.
\end{proof}
\noindent In order to develop another necessary corollary toward our result about $\mathbb{F}^{*}$, we must establish a result for polynomials in a field. Let $k$ denote an arbitrary field.\footnote{While the notation $R[x]$ to denote the ring of single-variable polynomials for some field $R$ is often convention, we prefer the notation $k[x]$ to distinguish the fact $k$ is a field for which polynomials in $k[x]$ take coefficients.}
\begin{prop} \label{most number of roots for f in k[x]}
    Let $f(x)\in k[x]$. Suppose that $\deg(f(x))=n$. Then $f(x)$ has at most $n$ distinct roots. 
\end{prop}
\begin{proof}
    We prove this by induction on the degree of $f(x)\in k[x]$. If $n=1$, then the assertion is clearly true, as a monic polynomial clearly has both a minimum and maximum equivalent number of roots, hence exactly 1 distinct root. Assume that the assertion is true for polynomials of degree $n-1$. This allows us to extend the assertion to degree $n$ later. 
    
    \indent To begin, if $f(x)$ has no roots in the field $k$, then clearly we are done as $f(x)$ therefore has $0$ roots. However, if $\alpha$ is a root, then by the division algorithm for polynomials, we can write $f(x)=q(x)(x-\alpha)+r$, where $r$ is some constant and $q(x)\in k[x]$. If we let $x=\alpha$ then $f(\alpha)=q(\alpha)(\alpha-\alpha)+r=r$. We assumed that $\alpha$ was a root of $f(x)$, so $r=0$. Therefore $q(x)|f(x)$ and $f(x)=q(x)(x-\alpha)$ and $\deg(q(x))=\deg(f(x))-1=n-1$. Let $\beta\neq\alpha$ be another such root of $f(x)$. Then $f(\beta)=0=(\beta-\alpha)q(\beta)$, which is only possible when $q(\beta)=0$. Therefore we have shown that $q(x)$ has at most $n-1$ distinct roots as we can repeat the process for distinct roots $\beta_{1}\neq\beta_{2}\neq\cdots\neq\beta_{n-1}$. Thus, since $q(x)$ has at most $n-1$ roots, $f(x)$ has at most $n$ roots, and we are done. 
\end{proof}
\noindent We now present one final corollary that relates polynomials to the polynomial $x^{q}-x$ and then prove that the multiplicative group of a finite field is cyclic.
\begin{cor} \label{number of roots for f(x) that divides x^q-x}
    If some polynomial $f(x)|x^{q}-x$, then $f(x)$ has exactly $d$ distinct roots, where $\deg(f(x))=d$. 
\end{cor}
\begin{proof}
    Let the product $f(x)g(x)=x^{q}-x$, where $g(x)=(x^{q}-x)/f(x)$ has the property $\deg(g(x))=q-d$. Assume that $f(x)$ has fewer than $d$ distinct roots. Then by Proposition \ref{most number of roots for f in k[x]}, the product $f(x)g(x)$ would have fewer than $d+(q-d)=q$ roots, noninclusive. However, by the definition of the product, the product has at most $q$ roots, contradicting our assumption. Therefore $f(x)$ has $d$ distinct roots. 
\end{proof}

\noindent The final proof of this section requires that we are familiar with some basic facts about cyclic groups. The following lemma can be proven using the Fundamental Theorem of Cyclic Groups and other results in elementary group theory. 

\begin{lemma}[Gallian] \label{result for cyclic groups to prove Theorem 0.3}
    If $d$ is a positive divisor of $n$, the number of elements of order $d$ in a cyclic group of order $n$ is $\phi(d)$.\footnote{\noindent Recall that the \textit{Euler totient function} $\phi(n)$ returns the number of integers coprime to $n$, i.e. $\phi(n)=|\{ a|\gcd(a,n)=1 \}|$.}
\end{lemma}
\begin{proof}
    The proof may be found in Chapter 4 of \cite{Gallian_2002}. 
\end{proof}

\noindent We now proceed with the final proof.

\begin{thm} \label{multiplicative group of finite field is cyclic}
    The multiplicative group of a finite field is cyclic. 
\end{thm}
\begin{proof}
    In order to prove that the multiplicative group of a finite field is cyclic, we must prove several important properties about its generator. Note first that some multiplicative group of a finite field $\mathbb{F}^{*}$ must have order $q-1$ in this proof. First, if there exists some subgroup with order $d|q-1$, then $x^{d}-1|x^{q-1}-1$ because we are using $\mathbb{F}[x]$ and degrees of polynomials in $\mathbb{F}[x]$ as an analogue for considering orders of the multiplicative group $\mathbb{F}^{*}$ and its subgroups. By Corollary \ref{number of roots for f(x) that divides x^q-x}, we know that since $x^{q-1}-1$ and $x^{q}-x$ are equivalent forms (this is true because Corollary \ref{number of roots for f(x) that divides x^q-x} also makes an assertion about divisibility in $\mathbb{F}^{*}$), we therefore can say that $x^{d}-1$ has $d$ distinct roots. Therefore the subgroup of $\mathbb{F}^{*}$ with elements satisfying the relation $x^{d}-1=0$ or $x^{d}=1$ has order $d$.

    \indent Let $\psi(d)$ be the number of elements in $\mathbb{F}^{*}$ of order $d$. Recall that the order of an element is the number of times an operation must be applied in order to return to itself; in this case, we are defining the arithmetic function $\psi(d)$ to be the order of the subgroup of $\mathbb{F}^{*}$ containing elements with order exactly $d$. Then we can make the assertion that for every $c$ that divides this order $d$, there exists some summatory function that takes in all such $c$ and outputs this exact order $d$. We can express this as 
    \begin{equation} \label{summatory function for psi(d) in multiplicative group of finite field is cyclic proof}
        \sum_{c|d}\psi(c)=d.
    \end{equation}
    \noindent For example, let there exist some set of divisors $c_{1},c_{2},\ldots,c_{n}$ such that $c_{1},c_{2},\ldots,c_{n}|d$. Then in \ref{summatory function for psi(d) in multiplicative group of finite field is cyclic proof} we have
    \begin{equation*}
        \sum_{c|d}\psi(c) = \psi(c_{1})+\psi(c_{2})+\cdots + \psi(c_{n}) = d. 
    \end{equation*}
    \noindent Therefore, by Möbius inversion, we can write
    \begin{equation*}
        \psi(d) = \sum_{c|d}\mu(c)F\bigg(\frac{d}{c}\bigg).
    \end{equation*}
    \noindent Notice that the summatory function $F$ no longer applies here. This is because for each divisor $c_{i}$ of $d$, the arithmetic function $\psi(d)$ simply counts the number of $c_{i}$, adding $1$ each time to its value. Therefore, we can write
    \begin{equation*}
        \sum_{c|d}\mu(c)F\bigg(\frac{d}{c}\bigg) = \sum_{c|d}\mu(c)\frac{d}{c}.
    \end{equation*}
    \noindent Notice too, however, that this is simply the number of divisors coprime to $d$, or $\phi(d)$. We can remove $\mu(c)$ because it takes in divisors of $d$, which will always evaluate to $1$ if there are $k\equiv 0 \pmod{2}$ distinct prime factors, and $0$ if there are $k\equiv 1\pmod{2}$ distinct prime factors. In either case, we have $\psi(d)=\phi(d)$, or the number of elements in $\mathbb{F}^{*}$ of order $d$ is equivalent to the number of divisors coprime to $d$. By Lemma \ref{result for cyclic groups to prove Theorem 0.3}, $\mathbb{F}^{*}$ is cyclic. 
\end{proof}

\subsection{$n$th Power Residues and a Connection to Finite Fields}
\noindent In this section, we focus on proving the following important result. Let $n\in\mathbb{Z}^{+}$ and let $|\mathbb{F}|=q$ so $|\mathbb{F}^{*}|=q-1$. 
\begin{thm} \label{last result about F*}
    Let $\alpha\in\mathbb{F}^{*}$. Then the equation $x^{n}=\alpha$ has solutions if and only if $\alpha^{(q-1)/d}=1$, where $d=\gcd(n,q-1)$. If there indeed are solutions, then there are exactly $d$ solutions. 
\end{thm}
\noindent In order to understand the proof of the prior result in the context of finite fields, we begin by proving a proposition for which the prior result is a generalization. Recall that for $m,n\in\mathbb{Z}^{+}$ and $a\in\mathbb{Z}$ with $\gcd(a,m)=1$, we define $a$ to be an \textit{$n$th power residue modulo $m$} if $x^{n}\equiv 1 \pmod{m}$ is solvable and $a$ is a solution.
\begin{prop} \label{analogue to last result about F*}
    If $m\in\mathbb{Z}^{+}$ has primitive roots and $\gcd(a,m)=1$, then $a$ is an $n$th power residue modulo $m$ if and only if $a^{\phi(m)/d}\equiv 1 \pmod{m}$, where $d=\gcd(n,\phi(m))$. 
\end{prop}
\noindent In order to prove this result, we must recall some facts in elementary number theory. Recall that the set of residue classes modulo $m$ is denoted by $\mathbb{Z}/m\mathbb{Z}$. This is in fact a ring, but we will not prove it here. The set of all representatives for the residue classes of $\ZZ/m\ZZ$ is the complete set of residues modulo $m$. 

\indent Perhaps the most elementary study of congruences can be summarized in linear congruences of the form $ax\equiv b \pmod{m}$. In particular, we are interested in determining the solvability of linear congruences of this form. For some interesting perspective, the number of solutions to a linear congruence is the value of $n$ in an $n$-tuple $(a_{1},\ldots,a_{n})$ such that $f(a_{1},a_{2},\ldots,a_{n})\equiv 0\pmod{m}$ for a linear congruence $f(x_{1},x_{2},\ldots,x_{n})\equiv 0\pmod{m}$ in $n$ variables. Uniqueness of such an $n$-tuple is assumed, so that if there exists some $n$-tuple $(b_1,\ldots,b_n)$ that also satisfies the polynomial equation, it must be the exact same $n$-tuple. 

\indent As we move forward, we will begin by looking at linear congruences in one variable. 
\begin{prop} \label{solutions of linear congruence ax=b (modulo m)}
    Let $d=\gcd(a,m)$. The linear congruence $ax\equiv b \pmod{m}$ has solutions if and only if $d|b$. Moreover, if $d|b$, then there are exactly $d$ solutions. Furthermore, if $x_{0}$ is a solution, then all other solutions can be written in the form $x_{0}+m',\ldots,x_{0}+(d-1)m'$.
\end{prop}
\begin{proof}
    \noindent To preface the proof, we need the following remark. We let $d=\gcd(a,m)$. We also set $a'=a/d$ and $m'=m/d$. Then $\gcd(a',m')=1$ since $a'$ and $m'$ are in their ultimately reduced forms. 
    
    \indent We prove the biconditional first. We begin with the forward direction. Let $x_{0}$ be a solution to the linear congruence. Then it satisfies the relation $ax_{0}-b=my_{0}$ for some $y_{0}\in\mathbb{Z}$. Then by the opening remark in this proof, we have $b=ax_{0}-my_{0}=da'x_{0}-dm'y_{0}=d(a'x_{0}-m'y_{0})$; thus $d|b$. We now prove the backward direction. Suppose that $d|b$. By Bézout's Lemma, there must exist integers $x_{0}'$ and $y_{0}'$ such that $ax_{0}'-my_{0}'=d$. If we let some $c=b/d$ and multiply both sides by $c$, we obtain $a(x_{0}'c)-m(y_{0}'c)=b$. Letting $x_0=cx_{0}'$ and $y_{0}=cy_{0}'$, we can write $ax_{0}-b=m(y_{0})$ or $ax_{0}-my_{0}=b$, which gives the linear congruence a solution. 

    \indent We now prove that there are exactly $d$ solutions to the congruence $ax\equiv b \pmod{m}$. Suppose that both $x_{0}$ and $x_{1}$ are solutions such that $ax_{0}\equiv b\pmod{m}$ and $ax_{1}\equiv b\pmod{m}$. This implies that $a(x_{1}-x_{0})\equiv 0\pmod{m}$. Therefore, for two distinct pairs $a,m$ and $a',m'$ we have $m|a(x_{1}-x_{0})$ and $m'|a'(x_{1}-x_{0})$ respectively. The second statement implies that $m'|x_{1}-x_{0}$ by our opening remark; in other words, for some $k\in\ZZ$ we have $x_{1}=x_{0}+km'$. As we vary the value of $k$ from $0$ to $d-1$, we see that there are incongruent solutions in the form $x_{0},x_{0}+m',\ldots,x_{0}+(d-1)m'$. Suppose that another solution to $ax\equiv b\pmod{m}$ is $x_{1}=x_{0}+km'$. By the division algorithm, there exist $r,s\in\ZZ$ such that $k=rd+s$ where $0\leq s<d$. Substituting, this gives $x_{1}=x_{0}+(rd+s)m'=x_{0}+sm'+rm$. Since $x_{1}=x_{0}+km'$, we must have $r=0$, so $k=s$. Since $s$ runs from $0$ to $d-1$, there are thus exactly $d$ solutions. 
\end{proof}
\noindent This establishes the solvability of linear congruences. The equation $ax\equiv b\pmod{m}$ is equivalent to writing the equivalence relation $[a]x=[b]$ in the ring $\mathbb{Z}/m\mathbb{Z}$. By Proposition \ref{solutions of linear congruence ax=b (modulo m)}, the congruence has solutions if and only if $d|b=1$, which is equivalent to when $\gcd(a,m)=1$\footnote{We say that the residue class $[a]$ is a \textit{unit} if and only if it satisfies $[a]x=1$, or if $ax\equiv 1 \pmod{m}$ has solutions.}. Thus, $[a]$ is a unit if and only if $\gcd(a,m)=1$. A special fact about $\mathbb{Z}/m\mathbb{Z}$ is that there are exactly $\phi(m)$ such units. If we let $m=p$ be prime, then all residue classes in $\mathbb{Z}/p\mathbb{Z}$ are units, and we can prove that, in both the multiplicative and additive cases, $\ZZ /p\ZZ$ is a field. 

\indent With this result, Euler's Theorem can be proven using the elementary fact that a residue class is a unit if, when multiplied with another residue class, yields the $[1]$ residue class modulo $m$. 
\begin{thm}[Euler's Theorem] \label{Euler's Theorem}
    If $\gcd(a,m)=1$, then
    \begin{equation*}
        a^{\phi(m)}\equiv 1 \pmod{m}.
    \end{equation*}
\end{thm}
\noindent An immediate corollary for prime $m$ is Fermat's Little Theorem, which is used widely in elementary number theory. 
\begin{cor}[Fermat's Little Theorem] \label{Fermat's Little Theorem}
    If $p$ prime and $p\nmid a$, then 
    \begin{equation*}
        a^{p-1}\equiv 1 \pmod{p}.
    \end{equation*}
\end{cor}

\noindent We are becoming carried away with ourselves with these results. In a final step before our proof of Proposition \ref{analogue to last result about F*}, we note that much like our proof of Theorem \ref{multiplicative group of finite field is cyclic}, we were studying the existence of an $x$ that acts as a generator for $\mathbb{F}^{*}$. Now we consider the analogue over $U(\mathbb{Z}/n\mathbb{Z})$, the group of units of the integers modulo $n$. It can be shown that if $p$ prime, $U(\mathbb{Z}/p\mathbb{Z})$ is cyclic; the proof is essentially identical to the proof of Theorem \ref{multiplicative group of finite field is cyclic}. We say that $a$ is a \textit{primitive root modulo $p$} if $p-1$ is the smallest integer such that $a^{p-1}\equiv 1 \pmod{p}$\footnote{We also define an integer $a$ to be a \textit{primitive root} modulo $p$ prime if its residue class $[a]$ generates $U(\mathbb{Z}/p\mathbb{Z})$, the group of units of $\ZZ/p\ZZ$.}. Now we proceed to prove Proposition \ref{analogue to last result about F*}.

\begin{proof}[Proof of Proposition \ref{analogue to last result about F*}]
To begin, let $g$ be a primitive root modulo $m$, and let $a=g^{b}$ and $x=a^{y}$. Then the $n$th-degree congruence $x^{n}\equiv 1 \pmod{m}$ is equivalent to $a^{yn}\equiv g^{yn}\equiv g^{b}\pmod{m}$. The second equivalence can be taken due to the fact that we assumed $g$ to be a primitive root modulo $m$. Simplifying, we have $g^{ny-b}\equiv 1\pmod{m}$. Since $g$ is a primitive root, this only occurs if $ny-b$ is some multiple of $\phi(m)$ by Euler's Theorem. Therefore $ny-b\equiv k\phi(m)$ for $k\in\mathbb{Z}$ so $ny\equiv b\pmod{\phi(m)}$. This is a linear congruence and is solvable if and only if $d|b$ by Proposition \ref{solutions of linear congruence ax=b (modulo m)}. To show the forward direction of this proposition, let $d|b$. Then $a^{\phi(m)/d}\equiv g^{(b\phi(m))/d}\equiv g^{l\phi(m)} \pmod{m}$ for some integer constant $l$, or $g^{l\phi(m)} \equiv 1\pmod{m}$. To show the backward direction, let $a^{\phi(m)/d}\equiv 1 \pmod{m}$. Then $g^{(b\phi(m))/d}\equiv 1 \pmod{m}$, which only occurs when $b/d$ is an integer. This only occurs if $d|b$. Therefore $d|b$, and we have proven both directions. 
\end{proof}

\noindent Notice also that since the linear congruence is solvable, $d=\gcd(n,\phi(m))$, and there are exactly $d$ solutions. This deduction is necessary in the proof of Theorem \ref{last result about F*}. We now return to our proof of Theorem \ref{last result about F*}. 

\begin{proof}[Proof of Theorem \ref{last result about F*}]
\noindent The statement of the theorem appears very similar to that of Proposition \ref{analogue to last result about F*}. To begin, let $\gamma$ be a generator of the cyclic group $\mathbb{F}^{*}$. As in Proposition \ref{analogue to last result about F*}, we let $\alpha=\gamma^{a}$, and $x=\gamma^{y}$. Then the equivalence relation $x^{n}=\alpha$ is equivalent to $\gamma^{yn}=\gamma^{a}$. We can reduce this equivalence to a congruence by removing the base of $\gamma$ as follows. Dividing both sides by $\gamma^{b}$, we have $\gamma^{ny-b}=1$. Similar to the proof of Proposition \ref{analogue to last result about F*}, and due to the fact that we defined $\gamma$ to be a generator of $\mathbb{F}^{*}$, we must have that $ny-b$ is an integer multiple of $q-1$, the order of the multiplicative group of the finite field. Therefore $ny-b=k(q-1)$ so $ny\equiv b\pmod{q-1}$. As in Proposition \ref{analogue to last result about F*}, this is a linear congruence, and we can apply Proposition \ref{solutions of linear congruence ax=b (modulo m)} as follows. 

\indent The congruence is solvable if and only if $d|a$. Suppose first that $d|a$. Then $\alpha^{(q-1)/d}\equiv \gamma^{(a(q-1))/d} \equiv \gamma^{r(q-1)}$ for some integer constant $r$. Since $r$ must be an integer, $\gamma^{r(q-1)}=1$. To prove the backward direction, let $\alpha^{(q-1)/d}=1$. Then $\gamma^{(a(q-1))/d}=1$, which is only possible if $d|a$, as $\gamma$ is a generator of $\mathbb{F}^{*}$. Therefore $d|a$ and we have proven both directions. 
\end{proof}

\noindent Notice that in this result, since the linear congruence is solvable, the number of solutions is given by $d=\gcd(n,q-1)$ by Proposition \ref{solutions of linear congruence ax=b (modulo m)}, and we are done. 

\begin{rem}\label{Remark: nth power residues and how many solutions to x^n=alpha}
In relation to $n$th power residues, it is also interesting consider what might happen to the number of solutions to the equation $x^{n}=\alpha$ for $\alpha\in\mathbb{F}^{*}$ with varying values for $d$. If $\gcd(n,q-1)=1$, then there is only 1 unique solution to the equation $x^{n}=\alpha$. Alternatively, if $n|q-1$ instead, then there are exactly $\gcd(n,q-1)=\frac{q-1}{n}$ solutions to $x^{n}=\alpha$, and there are $n$ solutions if $\alpha=\beta^{n}$ for some $\beta\in\mathbb{F}^{*}$. 
\end{rem}

\subsection{Structure of Finite Fields}
\noindent Now that we have surveyed the multiplicative group of a finite field, we might be interested in determining further characteristics of finite fields and their structural properties, especially in regard to their construction. Most notably, in this section we determine the order of a finite field and show how finite fields have a very intuitive relationship with their subfields. These results prepare us in proving the existence of finite fields later. 

\begin{prop} \label{an important isomoprhism by the First Isomorphism Theorem}
    Let $\mathbb{F}$ be a finite field. The integer multiples of the identity form a subfield of $\mathbb{F}$ that is isomorphic to $\mathbb{Z}/p\mathbb{Z}$ for $p$ a prime. 
\end{prop}
\begin{proof}
    As a means of standardizing notation, we use $e$ as the identity of the multiplicative group of $\mathbb{F}$ given as $\mathbb{F}^{*}$ in this proof. Define $\phi$ as a mapping of the integers to the finite field $\mathbb{F}$ that takes every $n\in\mathbb{Z}$ to some $ne$, or the integer multiples of the multiplicative identity of $\mathbb{F}$. This is a ring homomorphism because the original operation is preserved, and we are operating under $\phi$ from the ring of integers to $\mathbb{F}$. It is not difficult to show that $\phi$ is bijective and satisfies $\phi(a+b)=\phi(a)+\phi(b)$ and $\phi(ab)=\phi(a)\phi(b)$. The resulting image, namely the $ne$s, form a finite subring of $\mathbb{F}$. More specifically, since $\mathbb{Z}$ commutes, it is also a nonzero commutative ring, or an integral domain. The kernel of the homomorphism is thus Ker$(\phi)=\{ n\in\mathbb{Z} | \phi(n)=e \}$. In this way, the kernel is a nonzero prime ideal, meaning that either $n$ or $e$ must belong to the finite subring and naturally the integral domain. Thus, by Theorem \ref{First Isomorphism Theorem}, the field $\mathbb{Z}/p\mathbb{Z}$ of the integers modulo $p$ a prime (Note: the set $p\mathbb{Z}$ is exactly the aforementioned Ker$(\phi)$ because it is a prime ideal of $\mathbb{Z}$) must be isomorphic to the image of $\phi$, or $\text{im}(\phi)$, in the mapping from $\mathbb{Z}$ to $\mathbb{F}$. 
\end{proof}

\noindent Now that we have proven a fact about the finite field and its relation to the field $\mathbb{Z}/p\mathbb{Z}$, we further explore properties of $\mathbb{F}$ in determining its order. This brings us to the following result. 

\begin{prop} \label{number of elements in a finite field is a power of a prime}
    The number of elements in a finite field is a power of a prime. Namely, a finite field over a vector space with dimension $n$ has order $p^{n}$. 
\end{prop}

\begin{proof}
    From linear algebra, we know that every field can be expressed as a finite-dimensional vector space over each of its subfields, or in this case, every field is a finite-dimensional vector space over $\mathbb{Z}/p\mathbb{Z}$. We will not prove this here as it lies beyond the scope of this paper, but the result is critical in our proof of this result. 
    
    \indent Let $n$ be the dimension of the vector space and let $\omega_{1},\omega_{2},\ldots,\omega_{n}$ be a basis of $\mathbb{F}$. By the construction of a finite-dimensional vector space, every element in $\mathbb{F}$ has a unique representation as a linear combination of all vectors in the basis and elements of $\ZZ/p\ZZ$, namely as $a_{1}\omega_{1}+a_{2}\omega_{2}+\cdots+a_{n}\omega_{n}$ where $a_{i}\in\mathbb{Z}/p\mathbb{Z}$ for all $1\leq i\leq n$. Since the field $\mathbb{Z}/p\mathbb{Z}$ has prime order $p$, we know that there are exactly $p$ possible inputs for each $a_{i}$. This gives us $p^{n}$ total linear combinations of the basis, so the order of $\mathbb{F}$ is $p^{n}$. 
\end{proof}

\noindent As we continue studying $\mathbb{F}$, we introduce a definition. Let $e$ represent the multiplicative identity of a finite field $\mathbb{F}$. We define the \textit{characteristic} of a finite field $\mathbb{F}$ to be the minimum element $p$ to satisfy $pe=0$, where $0$ is the additive identity. As we have seen before, $p$ must be a prime as it is the only possible integer that satisfies the isomorphism in Proposition \ref{an important isomoprhism by the First Isomorphism Theorem}. An important fact about the characteristic is that when applied to \textit{any} element of the finite field, it yields the additive identity. In other words, if there is some $\alpha\in\mathbb{F}$, then $p\alpha=p(\alpha e)=(pe)\alpha=0\cdot \alpha=0$ by the commutativity of $\mathbb{F}$. This leads us to the following result. 
\begin{prop} \label{binomial theorem analogue for finite field with characteristic p}
    If a finite field $\mathbb{F}$ has characteristic $p$, then for $\alpha,\beta\in\mathbb{F}$ and some $d\in\mathbb{Z}^{+}$, 
    \begin{equation*}
        (\alpha+\beta)^{p^{d}}=\alpha^{p^{d}}+\beta^{p^{d}}.
    \end{equation*}
\end{prop}
\begin{proof}
    We prove this by induction on $d$. In the base case where $d=1$, this is obvious as $(\alpha+\beta)^{p}=\sum_{r=0}^{p}\binom{p}{r}\alpha^{p-r}\beta^{r}$. Expanding, we have 
    \begin{equation*}
        \alpha^{p}+\binom{p}{1}\alpha^{p-1}\beta + \binom{p}{2}\alpha^{p-2}\beta^{2} + \cdots + \binom{p}{p-1}\alpha\beta^{p-1} + \beta^{p}.
    \end{equation*}
    \noindent Since $\mathbb{F}$ has characteristic $p$, each multiple of $p$ must be equivalent to the additive identity. In other words, since $p$ divides each binomial coefficient where $i=1,2,\ldots,p-1$ (this can be proven using simple facts about the binomial coefficient), their multiples equate to the additive identity. Therefore we have $(\alpha+\beta)^{p}=\alpha^{p}+\beta^{p}$. 

    \indent Assume that this relation holds for some $d=k$. Then $(\alpha+\beta)^{p^{k}}=\alpha^{p^{k}}+\beta^{p^{k}}$. We wish to prove that it holds for $d=k+1$. As a result of the inductive hypothesis, we have
    \begin{align*}
        ((\alpha+\beta)^{p^{k}})^{p} & =(\alpha^{p^{k}}+\beta^{p^{k}})^{p} \\ 
        (\alpha+\beta)^{p^{k+1}} & = \sum_{r=0}^{p}\binom{p}{r}(\alpha^{p^{k}})^{p-r}(\beta^{p^{k}})^{p} \\ &
        = \sum_{r=0}^{p}\binom{p}{r}(\alpha^{p^{k}})^{p}(\alpha^{p^{k}})^{-r}(\beta^{p^{k}})^{p} \\ &
        = \alpha^{p^{k+1}} + \binom{p}{1}\alpha^{p^{k+1}}\alpha^{-p^{k}}\beta^{p^{k+1}} + \cdots + \binom{p}{p-1}\alpha^{p^{k+1}}\alpha^{-p^{k}(p-1)}\beta^{p^{k+1}} + \beta^{p^{k+1}} \\ &
        = \alpha^{p^{k+1}} + \beta^{p^{k+1}},
    \end{align*}
    \noindent where intermediate terms in the binomial expansion vanish modulo $p$ since $\mathbb{F}$ has characteristic $p$. 
\end{proof}

\noindent Now that we have information about binomial powers and the order of finite fields, we may be interested in determining not only properties of $\mathbb{F}$ but of subfields of the finite field $\mathbb{F}$ for which $\mathbb{Z}/p\mathbb{Z}$ is the scalar field. For the sake of the following results, we will denote such fields as $\mathbb{E}$. Let $n$ be the order of $\mathbb{F}$ and $d$ be the dimension of $\mathbb{E}$. It is possible to show, using techniques of field extensions and algebraic extensions, that $d|n$. We will instead provide an alternate proof that $d|n$ using finite fields. The underlying concept suggests that there is one and only one intermediate subfield $\mathbb{E}$ corresponding to each divisor $d$ of $n$. In order to prove this we begin with the following elementary results. 
\begin{lemma} \label{divisibility of polynomials if and only if l|m}
    Let $\mathbb{F}$ be a field. Then for $x^{l}-1, x^{m}-1\in \mathbb{F}[x]$ we have that $x^{l}-1 | x^{m}-1$ if and only if $l|m$. 
\end{lemma}
\begin{proof}
    Suppose that $l\nmid m$, or that $m=ql+r$ for some remainder $r\in[0,1)$ and divisor $q\in\mathbb{Z}$. Then, 
    \begin{equation*}
        \frac{x^{m}-1}{x^{l}-1}=\frac{x^{ql}x^{r}-1}{x^{l}-1}=\frac{x^{ql}-1}{x^{l}-1}x^{r}+\frac{x^{r}-1}{x^{l}-1}.
    \end{equation*}
    \noindent By polynomial division, the first term on the \textit{RHS} can be written as the polynomial $x^{r}((x^{l})^{q-1}+(x^{l})^{q-2}+\cdots + x^{l}+1)$. The remaining quotient can be seen to be $0$ if and only if $r=0$, as it evaluates to $0/(x^{l}-1)=0$. Therefore the \textit{RHS} is a polynomial if and only if $r=0$, which occurs if and only if $l|m$, so we are done. 
\end{proof}

\begin{rem} \label{Remark: important remark as a corollary to lemma about divisibility}
    This result is also true for positive integers in place of $x$ and can be summarized as follows: if $a\in\mathbb{Z}^{+}$, then $a^{l}-1 | a^{m}-1$ if and only if $l|m$. We will not prove it here as the proof is identical to that of Lemma \ref{divisibility of polynomials if and only if l|m}. 
\end{rem} 

\noindent Now that we have these results about divisibility, we can prove the relation between $\mathbb{F}$ and its subfields. 

\begin{thm}
    Let $\mathbb{F}$ be a finite field of dimension $n$ over $\mathbb{Z}/p\mathbb{Z}$. Then the subfields of $\mathbb{F}$ have an injection with the divisors $d$ of $n$. 
\end{thm}
\begin{proof}
    Let $\mathbb{E}$ be a field of dimension $d$ over the field $\mathbb{Z}/p\mathbb{Z}$. Furthermore, let the finite field $\mathbb{F}$ have dimension $n$ such that $\mathbb{E}$ is a subfield of $\mathbb{F}$. In this proof we wish to show that $d|n$. 

    \indent To begin, notice that by Proposition \ref{number of elements in a finite field is a power of a prime}, we know that $\mathbb{E}$ must have an order of a power of a prime, namely $p^{d}$, since the dimension of $\mathbb{E}$ is $d$. We can verify this by counting all possible linear combinations with respect to some basis, just as in the proof of Proposition \ref{number of elements in a finite field is a power of a prime}. Since the subfield $\mathbb{E}$ has order $p^{d}$, its multiplicative group $\mathbb{E}^{*}$ therefore must have order $p^{d}-1$. By the definition of the multiplicative group, exactly $p^{d}-1$ elements satisfy the polynomial equation $x^{p^{d}-1}-1=0$ where $x$ is some arbitrary variable. Now, considering the multiplicative group $\mathbb{F}^{*}$, we can see that it similarly has exactly $p^{n}-1$ elements satisfying $x^{p^{n}-1}-1=0$. Therefore $x^{p^{d}-1}-1|x^{p^{n}-1}-1$. By Lemma \ref{divisibility of polynomials if and only if l|m}, we know that this implies $p^{d}-1|p^{n}-1$, and furthermore, by Remark \ref{Remark: important remark as a corollary to lemma about divisibility}, this implies that $d|n$.

    \indent Now that we have shown that $d|n$, we must prove that there is an injection between the subfields of $\mathbb{F}$ and the divisors of $n$. In order for an injection to exist, there must be a correspondence such that if $f(x)=f(y)$ for $x,y\in\mathbb{E}$ an arbitrary subfield of $\mathbb{F}$, then $x=y$ for $f$ a function from $\mathbb{E}$ of $\mathbb{F}$ to the divisors of $n$. We show this relationship in the following. To begin, suppose that $d|n$ again, as we will use this fact. Now let $\mathbb{E}=\{ \alpha\in\mathbb{F}|\alpha^{p^{d}}=\alpha \}$. By constructing the subfield $\mathbb{E}$ in this way, we are constructing all subfields such that all elements $\alpha$ in the finite field $\mathbb{E}$ are also in $\mathbb{F}$, just as we observed in Corollary \ref{first corollary to first proposition in section about finite fields} with the requirement that $\alpha\in\mathbb{F}$ be in $\mathbb{E}$ if and only if $\alpha$ satisfies the equation $\alpha^{p^{d}}=\alpha$. In order to ensure that this construction is valid, we must prove that $\mathbb{E}$ is indeed a field. In other words, the following properties must be true for any $\alpha,\beta\in\mathbb{E}$:
    \begin{center}
    \begin{enumerate}
        \item $(\alpha+\beta)^{p^{d}} = \alpha^{p^{d}}+\beta^{p^{d}} = \alpha + \beta$,
        \item $(\alpha\beta)^{p^{d}} = \alpha^{p^{d}}\beta^{p^{d}} = \alpha\beta$,
        \item $(\alpha^{-1})^{p^{d}} = (\alpha^{p^{d}})^{-1} = \alpha^{-1}$ for $\alpha\neq 0$. 
    \end{enumerate}
    \end{center}
    \noindent Property (1) follows immediately from Proposition \ref{binomial theorem analogue for finite field with characteristic p} since $\mathbb{E}$ also has characteristic $p$, and then by our initial construction of elements in $\mathbb{E}$. Properties (2) and (3) are trivial as they are inherent properties of $\mathbb{F}$ with the exception of the construction of $\mathbb{E}$. 

    \indent We now determine the order of $\mathbb{E}$. By the very construction of $\mathbb{E}$, we know that $\mathbb{E}$ is the set of solutions to the polynomial equation $x^{p^{d}}-x=0$. By Remark \ref{Remark: important remark as a corollary to lemma about divisibility}, since $d|n$, we know that $p^{d}-1|p^{n}-1$, and similarly by Lemma \ref{divisibility of polynomials if and only if l|m}, we must also have $x^{p^{d}-1}-1|x^{p^{n}-1}-1$. Consequently $x^{p^{d}}-x|x^{p^{n}}-x$. Since $\mathbb{E}$ comprises the solutions to $x^{p^{d}}-x=0$, we thus have that $x^{p^{d}}-x$ is exactly the $f(x)$ described by Corollary \ref{number of roots for f(x) that divides x^q-x}. Furthermore, the order of $\mathbb{E}$ is exactly the roots of $x^{p^{d}}-x$, so the order of $\mathbb{E}$ is $p^{d}$. Therefore $\mathbb{E}$ must have dimension $d$ over $\mathbb{Z}/p\mathbb{Z}$ by applying Proposition \ref{number of elements in a finite field is a power of a prime}. 

    \indent To finish the proof, we must show that each divisor $d$ of $n$ corresponds to a unique subfield of $\mathbb{F}$. To do this, let $\mathbb{E}'$ be another subfield of $\mathbb{F}$ of dimension $d$ over $\mathbb{Z}/p\mathbb{Z}$. Then by our previous workings, every element of $\mathbb{E}'$ must satisfy $x^{p^{d}}-x=0$, and these are the elements that do. However, we constructed $\mathbb{E}$ in the same way, so the solutions to $x^{p^{d}}-x=0$ must coincide, i.e. $\mathbb{E}=\mathbb{E}'$. 
\end{proof}

\noindent These results establish some basic properties of finite fields. We are now concerned with whether such a finite field can exist. 

\subsection{The Existence of Finite Fields}
\noindent Now that we have shown the construction and some properties of finite fields, we may ask another question: given some number $p^{n}$ where $n\in\mathbb{Z}^{+}$, does there exist some finite field with exactly $p^{n}$ elements? This is largely what we will investigate here, summarized in the following theorem.  

\begin{thm} \label{existence of a finite field with p^n elements}
    Let $n\geq 1$ be an integer and let $p$ be prime. Then there must exist a finite field with exactly $p^{n}$ elements. 
\end{thm}

\noindent We first need to prove some ring-theoretic facts, and then we will return to proving results for the theorem. We stated that $k[x]$ was a Euclidean domain for some arbitrary field $k$. By Theorem \ref{Every Euclidean domain is a PID}, $k[x]$ is a PID. We say that some polynomial $p(x)\in k[x]$ is \textit{irreducible} if some polynomial $q(x)$ divides $p(x)$, then $q(x)$ must either be a constant multiple of $p(x)$ or is a constant. In this way, irreducible polynomials are analogous to prime numbers in $\mathbb{Z}$. We now need to prove the following results. 

\begin{lemma} \label{lemma for R/A is a field iff A is maximal ideal}
    If $A$ is an ideal of a ring $R$ and contains a unit (1), then $A=R$. 
\end{lemma}
\begin{proof}
    We prove this by double inclusion. Since $A$ is an ideal of $R$, by definition it is true that $A\subseteq R$. We need to show that $R\subseteq A$. Let $a\in A$ and $r\in R$. Since $A$ is an ideal, we know that $ar=ra\in R$ for all $r\in R$. Since $A$ contains $1$, it must be true that $1r=r1\in R$ for all $r\in R$, namely that $r\in A$ for all $r\in R$, which means that $R\subseteq A$. Therefore $R=A$, and we are done. 
\end{proof}

\begin{prop} \label{R/A is a field iff A is maximal ideal}
    Let $R$ be a commutative ring with unity and let $A$ be an ideal of $R$. Then the quotient ring $R/A$ is a field if and only if $A$ is maximal. 
\end{prop}
\begin{proof}
    We begin by proving the forward direction. Let the quotient ring $R/A$ be a field. Let $B$ be an ideal of $R$ such that $A\subseteq B$. Furthermore, consider some $b\in B$ but $b\not\in A$. Then $b+A$ is a nonzero element of the field $R/A$, and by the definition of a field, there must exist another element $c+A$ such that $(b+A)(c+A)=1+A$, or the multiplicative identity of $R/A$. Since $b\in B$, it must be true by the definition of an ideal that any multiple $bc\in B$. We then have $1+A=(b+A)(c+A)=bc+A$. Therefore $1-bc\in A\subseteq B$, so $1-bc+bc=1\in B$. By Lemma \ref{lemma for R/A is a field iff A is maximal ideal}, we have that $A$ is a maximal ideal. 

    \indent We now prove the backward direction. Suppose that $A$ is a maximal ideal and there exists some $b\in R$ but $b\not\in A$. To ensure that the quotient ring $R/A$ is a field, we need only show that $b+A$ has a multiplicative inverse because the other properties of a field are trivial. Consider the set $B=\{ br+a|r\in R, a\in A \}$. We can show that $B$ is an ideal of $R$ as follows. First, to prove the first condition for $B$ to be an ideal, consider distinct $r,a\in R$ such that we have two elements $br_{1}+a_{1}$ and $br_{2}+a_{2}$. Then we can see $br_{1}+a_{1}+br_{2}+a_{2}=b(r_{1}+r_{2})+(a_{1}+a_{2})\in B$ by the construction of $B$. Next, consider some element $r'\in R$. Then $(br+a)r'=brr'+ar'=b(rr')+r'a\in B$ again by the construction of $B$. Therefore $B$ is an ideal of $R$. Since $A$ was assumed to be maximal, in the construction of $B$ as an ideal of $R$ we then have that $B=R$. Then the unit $1$ is an element of $B$. Let $1=bc+a'$ for some $a'\in A$. Then we may write $1+A=bc+a'+A=bc+A=(b+A)(c+A)$, so we have shown that $b+A$ has a multiplicative inverse, namely $c+A$. 
\end{proof}
\noindent This leads us to the following result about the existence of polynomials with roots in a field. Recall that an \textit{irreducible} polynomial in $k[x]$ for some arbitrary field $k$ is equivalent to the notion of a prime in $\mathbb{Z}$. 
\begin{prop}\label{Proposition: for a field K there exists alpha in K such that f(alpha)=0}
    Let $k$ be an arbitrary field, and also let $f(x)\in k[x]$ be an irreducible polynomial. There exists some field $K$ containing $k$ and an element $\alpha\in K$ such that $\alpha$ is a root of $f(x)$, or $f(\alpha)=0$. 
\end{prop}
\begin{proof}
    We showed above that $k[x]$ is a PID. Since $f(x)$ is an irreducible polynomial, the ideal given by $(f(x))$ generated by $f(x)$ is a maximal and principal ideal (while also being a proper ideal), meaning that it is one of the largest ideals of $k[x]$. By Proposition \ref{R/A is a field iff A is maximal ideal}, the quotient ring $k[x]/(f(x))$ is a field since $(f(x))$ is a maximal ideal. Let $K'=k[x]/(f(x))$. This field runs through all $a(x)\in k[x]$ and combines each $a(x)$ with the entire ideal $(f(x))$, thereby generating the group of cosets $\{ a(x)+(f(x)) | a(x)\in k[x] \}$. Let $\phi$ be the homomorphism that maps $k[x]$ to $K'$ by mapping each element of $k[x]$ to its respective coset modulo $(f(x))$. For example, if $a_{1}(x)\in k[x]$ then $\phi$ would take $a_{1}(x)$ to its unique coset in the group of cosets. Now consider the following diagram. 
    \begin{figure}[H]
    \begin{tikzpicture}
        \node (A) at (0,3) {$k[x]$};
        \node (B) at (4,3) {$K'$};
        \node (C) at (0,0) {$k$};
        \node (D) at (4,0) {$\phi(k)$};

        \draw[->] (A) -- (B) node[midway, above] {$\phi$};
        \draw[->] (C) -- (D) node[midway, above] {$\phi$};
        \draw (0,2.5) -- (0,0.5);
        \draw (4,2.5) -- (4,0.5);
    \end{tikzpicture}
    \caption{Illustration of the homomorphism $\phi$ over two analogues}
    \end{figure}
    \noindent In this diagram, we naturally have the mapping from $k[x]$ to $K'$ described before. However, to simplify the proof, we are also considering the mapping via $\phi$ that similarly maps elements from the previously defined field $k$ to its cosets, defined as $\phi(k)=\{ a+(f(x))|a\in k \}$. It is obvious that $k$ is a subfield of $k[x]$, and it is also clear that since for all $a\in k$ it is also true that all $a(x)\in k[x]$ with coefficients $a$, $\phi(k)$ is a subfield of $K'$. 
    
    \indent We claim that $\phi(k)$ is isomorphic to $k$. To prove this, we need to show that the mapping via $\phi$ is injective and that it preserves the operation. It is clear that the operation is preserved due to the fact that we are transitioning between $k$ and the group of cosets with respect to the ideal $(f(x))$. We need only show that the map is injective. Consider some $a\in k$. If $\phi(a)=0$, then $a\in (f(x))$ because this means that for all $a$, it is true that $a+(f(x))=0$, or that $a=-(f(x))$, namely that $a\in (f(x))$. Suppose that $a\neq 0$. Therefore $a$ must be a unit and cannot be an element of $(f(x))$ as that would contradict the construction of the ideal. Therefore $a=0$. Thus, if $\phi(a)=0$, then $a=0$, satisfying the condition for $\phi$ to be an injection. 

    \indent Since $\phi$ is an isomorphism attached to $k$, we can identify $k$ with $\phi(k)$ instead, meaning that we are now considering $K'$ to be an extension of $k$ using the isomorphism via $\phi$. In other words, $K'$ contains a subfield, namely the aforementioned $\phi(k)$, that is isomorphic to $k$. Therefore we can relabel $K'$ as $K$ containing $k$, as this is the field that we desired to construct.

    \indent To complete the proof, we wish to bring this to context in polynomials. Let $\alpha$ be the coset of $x$ in $K$, namely the coset $\phi(x)=x+(f(x))=\alpha$. Then $0=\phi(f(x))=f(\phi(x))=f(\alpha)$. The second and third equivalence can be seen to be true by considering a simple example. For the sake of simplicity, we consider the monic polynomial $x^{2}\in k[x]$. After evaluating each composition, we see that $\phi(f(x))=f(\phi(x))=x^{2}+(x^{4})$. The equivalence can be generalized by considering the general form of a polynomial $a(x)\in k[x]$ with $\deg(a(x))=n$ and evaluating each composition accordingly.
    
    \indent Therefore $\alpha$ is a root of $f(x)$ in $K$. 
\end{proof}

\noindent In the following, we denote this field $K$ exactly as $k(\alpha)$, where $k$ is an arbitrary field. Let $k[\alpha]$ denote the ring of polynomials in $\alpha$ with coefficients belonging to the field $k$. We have the following. The $\alpha$ described below is exactly the roots of $f(x)$. 

\begin{prop}\label{Proposition: vector space basis for k(alpha) over k}
    The elements $1,\alpha,\alpha^{2},\ldots,\alpha^{n-1}$ form a vector space basis for the finite-dimensional vector space $k(\alpha)$ over $k$, where $k$ is the same arbitrary field and $n$ is the degree of $f(x)$ described in the previous proposition. 
\end{prop}

\noindent We omit the proof as it lies beyond the scope of this paper.

\indent We briefly introduced field extensions and the degree of a field extension in the preliminary section. Proposition \ref{Proposition: vector space basis for k(alpha) over k} shows that if we wish to find a field extension $[k(\alpha):k]=n$ of degree $n$, it is sufficient to construct an irreducible polynomial $f(x)\in k[x]$ with $\deg(f)=n$. In other words, we need only produce such a polynomial to show that a finite field with prime power order exists. 

\indent As we proceed, we will prove a powerful result, due to Gauss, that there exists an irreducible polynomial of every degree in the polynomial ring obtained by adjoining the finite field $\ZZ/p\ZZ=\mathbb{F}_{p}$ with $x$. This polynomial ring can be denoted with $\ZZ/p\ZZ[x]$, but we will use the notation $\mathbb{F}_{p}[x]$ instead. First, notice that in $\mathbb{F}_{p}[x]$ there are finitely many polynomials of any degree, ranging from $0$ to $p-1$. This is obvious because there are a finite number of combinations of elements in $\mathbb{F}_{p}$ that can form a polynomial. 

\indent In the following, we let $F_{d}(x)$ denote the product of all monic irreducible polynomials in $\mathbb{F}_{p}[x]$. The following result ultimately states that the polynomial $x^{p^{n}}-x$ can be factored into a product of monic irreducible polynomials with respect to degrees that are divisors of $n$. 

\begin{prop} \label{Proposition: factorization of x^{p^{n}}-x}
    \begin{equation*}
        x^{p^{n}}-x = \prod_{d|n}F_{d}(x). 
    \end{equation*}
\end{prop}
\begin{proof}
    First, we must prove that the product $F_{d}(x)$ contains only unique monic irreducible polynomials. We do this by supposing that if some arbitrary monic irreducible $f(x)$ divides $x^{p^{n}}-x$, then $f(x)^{2}$ cannot divide $x^{p^{n}}-x$. We show this as follows. Suppose $f(x)^{2}$ does indeed divide $x^{p^{n}}-x$. Then there must exist another monic irreducible $g(x)$ such that $x^{p^{n}}-x = f(x)^{2}g(x)$. Differentiating each side with respect to $x$, we obtain
    \begin{equation*}
        p^{n}x^{p^{n}-1}-1 = 2f(x)f'(x)g(x)+f(x)^{2}g'(x).
    \end{equation*}
\noindent Since $\mathbb{F}_{p}[x]$ has characteristic $p$, we set $p=0$. Therefore 
\begin{equation*}
    -1 = 2f(x)f'(x)g(x)+f(x)^{2}g'(x) = f(x)[2f'(x)+f(x)g'(x)].
\end{equation*}
\noindent This shows that $f(x)|1$, which, when $f(x)$ is a monic irreducible, is impossible. Therefore $f(x)^{2}\nmid x^{p^{n}}-x$. 

\indent We now only need to show that if $f(x)$ is a monic irreducible polynomial with $\deg(f)=d$, then it divides $x^{p^{n}}-x$ if and only $d|n$. Let $K=\ZZ/p\ZZ(\alpha)$ be the field mentioned earlier, where $\alpha$ is a root of $f(x)$ as described in Proposition \ref{Proposition: for a field K there exists alpha in K such that f(alpha)=0}. Since $\deg(f)=d$, we know that $K$ has order $p^{d}$ by Proposition \ref{number of elements in a finite field is a power of a prime}. Therefore all elements of $K$ are roots of $x^{p^{d}}-x$, i.e. satisfy the polynomial equation $x^{p^{d}}-x=0$. 

\indent We first prove the forward direction. Assume, WLOG, that $f(x)|x^{p^{n}}-x$, or that there exists another monic irreducible $g(x)\in \mathbb{F}_{p}[x]$ that divides $x^{p^{n}}-x$. Then we need to show that $\alpha^{p^{n}}=\alpha$ for some arbitrary root $\alpha\in K$. Suppose that $\alpha_{1}=a_{1}\alpha_{1}^{d-1}+a_{2}\alpha_{1}^{d-2}+\cdots+a_{d-1}\alpha_{1}+a_{d}$ is some arbitrary element of $K$. Then, plugging into the equivalence $\alpha^{p^{n}}=\alpha$ for $\alpha$, by Proposition \ref{binomial theorem analogue for finite field with characteristic p}, 
\begin{align*}
    (a_{1}\alpha_{1}^{d-1}+a_{2}\alpha_{1}^{d-2}+\cdots+a_{d-1}\alpha_{1}+a_{d})^{p^{n}} & = a_{1}(\alpha_{1}^{p^{n}})^{d-1} + a_{2}(\alpha_{1}^{p^{n}})^{d-2} + \cdots a_{d-1}(\alpha_{1}^{p^{n}}) + a_{d} \\ 
    & = a_{1}\alpha_{1}^{d-1}+a_{2}\alpha_{1}^{d-2}+\cdots+a_{d-1}\alpha_{1}+a_{d}.
\end{align*}
\noindent Therefore every element of $K$ satisfies the polynomial equation $x^{p^{n}}-x=0$. By the construction of $K$, its elements satisfy the polynomial equation $x^{p^{d}}-x=0$, so it must also be true that $x^{p^{d}}-x|x^{p^{n}}-x$. By Lemma \ref{divisibility of polynomials if and only if l|m}, this implies that $d|n$, thus proving the forward direction.

\indent We now proof the backward direction. Assume that $d|n$. We again have that an arbitrary root $\alpha\in K$ satisfies $\alpha^{p^{d}}=\alpha$. Since $f(x)$ is the monic irreducible with $\alpha$ as a root, we have $f(x)|x^{p^{d}}-x$. By Lemma \ref{divisibility of polynomials if and only if l|m} again, since $d|n$, we have $x^{p^{d}}-x|x^{p^{n}}-x$, and by transitivity $f(x)|x^{p^{n}}-x$, thus proving the proposition. 
\end{proof}

\noindent Now that we have proven that such a factorization of $x^{p^{n}}-x$ exists, we want to prove something about the number of monic irreducibles of a given degree in $\mathbb{F}_{p}[x]$. We let $N_{d}$ denote the number of monic irreducibles of degree $d$. We have the following.

\begin{cor}
    \begin{equation*}
        p^{n} = \sum_{d|n}dN_{d}.
    \end{equation*}
\end{cor}
\begin{proof}
    We equate the degrees of the \textit{LHS} and \textit{RHS} of Proposition \ref{Proposition: factorization of x^{p^{n}}-x}. Since the degree of the \textit{RHS} is the sum of the number of monic irreducibles of degree $d$ multiplied by each divisor of $n$, the result follows. 
\end{proof}

\noindent This gives us the following due to Möbius Inversion. 
\begin{cor}\label{Corollary: formula for number of monic irreducibles of deg n}
    \begin{equation*}
        N_{n} = \frac{1}{n}\sum_{d|n}\mu\bigg(\frac{n}{d}\bigg)p^{d}. 
    \end{equation*}
\end{cor}
\begin{proof}
    We use Theorem \ref{Theorem: Mobius Inversion}. Let the arithmetic function $f(n)=nN_{n}$ and let its summatory function be $F(n)=p^{n}$. By the inversion formula, we can write this as 
    \begin{align*}
        nN_{n} & = \sum_{d|n}\mu\bigg(\frac{n}{d}\bigg)p^{d} \\
        N_{n} & = \frac{1}{n}\sum_{d|n}\mu\bigg(\frac{n}{d}\bigg)p^{d}.
    \end{align*}
\end{proof}

\noindent This gives us the following due to Gauss.

\begin{prop}
    For each integer $n\geq 1$, there exists an irreducible polynomial of degree $n$ in $\mathbb{F}_{p}[x]$. 
\end{prop}
\begin{proof}
    Expanding the sum on the \textit{RHS} of Corollary \ref{Corollary: formula for number of monic irreducibles of deg n} (and excluding the intermediate terms due to the fact that the divisors are arbitrary for an arbitrary non-prime $n$), we obtain 
    \begin{equation*}
        N_{n} = \frac{1}{n}(p^{n}-\cdots+p\mu(n)).
    \end{equation*}
    \noindent Notice that the expression $(p^{n}-\cdots+p\mu(n))$ is never $0$ since the first term is $p^{n}$ and all remaining terms are a prime factor of either $1$ or $-1$ as given in the definition of the Möbius function. This implies that with respect to the degree $n$, there exists at least $1$ irreducible with that degree. 
\end{proof}

\noindent Since we have proven that there is an irreducible of every degree, we have shown that there exists a finite field with $p^{n}$ elements, thereby proving Theorem \ref{existence of a finite field with p^n elements}. 

\indent We have only provided a brief overview of the algebra and number theory required for the study of finite fields, but existing research and current research on finite fields is highly relevant to many areas of ongoing mathematics research. The study of finite fields is also interesting in its own right. In the next section we give a proof of quadratic reciprocity using finite fields. 

\subsection{Another Proof of the Law of Quadratic Reciprocity}

\noindent The following proof of quadratic reciprocity is an expanded version of a proof of quadratic reciprocity using quadratic Gauss sums due to Hausner published in 1961; see \cite{Hausner1961reciprocity}. While this proof relies heavily on quadratic Gauss sums, the fundamental argument of the proof we present relies largely on the existence of finite fields of prime power order to define a modified quadratic Gauss sum, and we take full advantage of properties of finite fields. The proof is as follows.

\begin{proof}[Proof of Theorem \ref{the law of quadratic reciprocity} using finite fields]
    Consider distinct odd primes $p$ and $q$. Obviously, we have $\gcd(p,q)=1$. Therefore, there exists some integer $n$ such that $q^{n}\equiv 1 \pmod{p}$. For example, $n-1$ might satisfy this congruence, which is a special case of Corollary \ref{Fermat's Little Theorem}. Now let $\mathbb{F}$ be some finite field of dimension $n$ over $\mathbb{Z}/q\mathbb{Z}$ (recall from earlier that we are able to do this as we are treating $\mathbb{F}$ as a vector space for which $\mathbb{Z}/q\mathbb{Z}$ is a scalar field). It is well-known that since $q$ is an odd prime, $\mathbb{Z}/q\mathbb{Z}$ is a field. Therefore the multiplicative group $\mathbb{F}^{*}$ is cyclic and has order $|\mathbb{F}|-1=q^{n}-1$. Since $\mathbb{F}^{*}$ is cyclic, it must have a generator. Let $\gamma$ be one such generator, and let $\lambda = \gamma^{(q^{n}-1)/p}$. In other words, $\lambda$ has order $p$, since $p$ is the least integer such that $\lambda^{p}=1$ by Corollary \ref{Fermat's Little Theorem}.

    \indent We now define an analogue to the quadratic Gauss sum. Let 
    \begin{equation*}
        \tau_{a} = \sum_{t=0,a\in\mathbb{Z}}^{p-1}\bigg( \frac{t}{p} \bigg)\lambda^{at}.
    \end{equation*}
    \noindent Similar to quadratic Gauss sums, we will denote the case when $a=1$ as simply $\tau$. Like the proof of quadratic reciprocity using quadratic Gauss sums in section 4.4 of \cite{relmat2022quadrec}, we will need two identities. Namely,
    \begin{enumerate}
        \item \begin{equation*}
            \tau_{a} = \bigg( \frac{a}{p} \bigg)\tau,
        \end{equation*}
        \item \begin{equation*}
            \tau^{2} = (-1)^{\frac{p-1}{2}}\overline{p}.
        \end{equation*}
    \end{enumerate}
    \noindent In (2), we let $\overline{p}$ be the coset of $p$ in $\mathbb{Z}/q\mathbb{Z}$. We first prove (1). Consider the case where $a\equiv 0 \pmod{p}$. This can clearly be shown to be true since $\sum_{t=0}^{p-1}\bigg(\frac{t}{p}\bigg)\cdot 1=0$. The second case is when $a\not\equiv 0\pmod{p}$. We follow the proof procedure used for Proposition \ref{property 1 of quadratic Gauss sums}, and ultimately we can prove the result by expanding the \textit{RHS} of (1) and proceeding with the same approach as in \cite{relmat2022quadrec}. To prove (2) we follow the proof procedure for Proposition \ref{property 2 of quadratic Gauss sums}. We leave the details of the proof to the reader, which can also be found in \cite{relmat2022quadrec}. The proof involves evaluating the particular sum
    \begin{equation*}
    \sum_{a=0}^{p-1}\tau_{a}\tau_{-a}
\end{equation*}
    \noindent in two different ways. 

    \indent Notice that in (2), $\overline{p}$ denotes the coset of $p$ in $\mathbb{Z}/q\mathbb{Z}$. In other words, it is comprised of all elements $p+q\mathbb{Z}$. Note that if some number is a square modulo $q$, then that is equivalent to stating that the coset consists of square elements. As with the proof of quadratic reciprocity using quadratic Gauss sums, we denote $p^{*} = (-1)^{(p-1)/2}p$. Then we can rewrite (2) as $\tau^{2} = \overline{(-1)^{(p-1)/2}p} = \overline{p^{*}}$. Therefore, the coset $p^{*}$ is a square modulo $q$, i.e. $p^{*}$ is a quadratic residue modulo $q$, so $(\frac{p^{*}}{q}) = 1$. Note that this is satisfied if and only if $\tau\in\mathbb{Z}/q\mathbb{Z}$. By Corollary \ref{first corollary to first proposition in section about finite fields}, this biconditional statement is true if and only if 
    \begin{equation*} \label{Equation: eq 1 for quadrec proof using finite fields}\tag{1}
        \tau^{q}=\tau.
    \end{equation*} Notice further that we are able to use Corollary \ref{first corollary to first proposition in section about finite fields} because obviously $\mathbb{Z}/q\mathbb{Z}\subset \mathbb{F}$. If we evaluate $\tau^{q}$, since all intermediate terms reduce to $0$ modulo $q$,
    \begin{equation*}
        \tau^{q} = \bigg( \sum_{t\in\mathbb{F}_{p}}\bigg( \frac{t}{p} \bigg)\lambda^{t} \bigg)^{q} = \sum_{t\in\mathbb{F}_{p}}\bigg( \frac{t}{p} \bigg)\lambda^{qt} = \tau_{q}.
    \end{equation*}
    \noindent Applying \eqref{Equation: eq 1 for quadrec proof using finite fields}, we can see that this is the same as $\tau_{q} = (\frac{q}{p})\tau$. Clearly, the only case where $\tau_{q}=\tau$ is when $(\frac{q}{p})=1$. If we work our way back to the statement of this biconditional, we have $(\frac{p^{*}}{q})=1$ if and only if $(\frac{q}{p})=1$, i.e. 
    \begin{equation*}
        \bigg( \frac{p^{*}}{q} \bigg) = \bigg( \frac{q}{p} \bigg).
    \end{equation*}
    \noindent As with the standard proof of quadratic reciprocity using quadratic Gauss sums, by Theorem \ref{supplement to quadratic reciprocity}, we can write this equivalence as 
    \begin{align*}
        \bigg(\frac{-1}{q}\bigg)^{\frac{p-1}{2}}\bigg(\frac{p}{q}\bigg) & = \bigg(\frac{q}{p}\bigg) \\
        \bigg(\frac{p}{q}\bigg)(-1)^{\frac{p-1}{2}\cdot \frac{q-1}{2}} & = \bigg(\frac{q}{p}\bigg),
    \end{align*}
    \noindent which is equivalent to the statement of Theorem \ref{the law of quadratic reciprocity}.
\end{proof}

\noindent Many proofs of quadratic reciprocity using Gauss sums exist, and this is one such proof that utilizes few facts from algebraic number theory. In the proof given by Hausner, finite fields are referred to instead as Galois fields and are denoted as $GF(q)$ where $q$ denotes order. For our purposes, as proven in Proposition \ref{number of elements in a finite field is a power of a prime}, we are dealing primarily with $GF(p^{k})$ for $k\in\mathbb{Z}^{+}$. Many similar proofs of quadratic reciprocity published during this period utilize similar notation. 

\section{Multiplicative Characters}
\noindent In this section we study multiplicative characters. The main motivation for studying multiplicative characters is to generalize quadratic residue symbols - which we referred to as Legendre symbols - to higher degree residue symbols. While this section provides a general overview of $n$th degree residue symbols, the case when $n=3$ will be the primary consideration of this paper. 

\subsection{Definitions and Some Basic Results}
\noindent One of the most elementary examples of a multiplicative character is the Legendre symbol $(a/p)$, as it can be thought of as a function of the coset of $a$ modulo $p$ a prime. We define a multiplicative character as follows, where we denote the integers modulo $p$ given as $\mathbb{Z}/p\mathbb{Z}$ by $\mathbb{F}_{p}$ for the sake of simplicity.
\begin{defn}[Multiplicative Character]
    We define a multiplicative character on a field $\mathbb{F}_{p}$ with $p$ elements as some mapping $\chi$ from the multiplicative group $\mathbb{F}_{p}^{*}$ to the nonzero complex numbers that satisfies the property that for all $a,b\in \mathbb{F}_{p}^{*}$,
    \begin{equation*}
        \chi(ab)=\chi(a)\chi(b).
    \end{equation*}
\end{defn}
\noindent Once we are more familiar with the basic definitions of multiplicative characters, we will refer to them as ``characters" instead. One concrete notion of the multiplicative character is the trivial multiplicative character, which naturally has properties that map every element of $\mathbb{F}_{p}^{*}$ to the multiplicative identity; namely, for a mapping $\eps$, we have $\eps(a)=1$ for all $a\in \mathbb{F}_{p}^{*}$. To study multiplicative characters in more depth, it is possible to consider the additive identity in the definition of the trivial and nontrivial multiplicative characters. Namely, if we let some character $\chi\neq\eps$, then we can define $\chi(0)=0$, and $\eps(0)=1$. We now prove the following results about multiplicative characters. 
\begin{prop} \label{basic facts about multiplicative characters 1}
    Let $\chi$ be some multiplicative character and $a\in \mathbb{F}_{p}^{*}$. Then
    \begin{enumerate}
        \item $\chi(1)=1$,
        \item $\chi(a)$ is a $(p-1)$st root of unity,
        \item $\chi(a^{-1})=(\chi(a))^{-1}=\overline{\chi(a)}$ (where a bar denotes conjugation).
    \end{enumerate}
\end{prop}
\begin{proof}
    To prove (1), we have from our definition of $\chi$ that $\chi(1)=\chi(1\cdot 1)=\chi(1)\chi(1)$. Since $\chi$ is a map to nonzero complex numbers, the only value for $\chi(1)$ that satisfies this relation is the complex number $1$, so $\chi(1)=1$. 
    
    \indent For (2), some complex number $a$ is a $(p-1)$st root of unity if $a^{p-1}=1$. Thus, following from (1), we have $1=\chi(1)=\chi(a^{p-1})=(\chi(a))^{p-1}$. Therefore $\chi(a)$ is a $(p-1)$st root of unity.

    \indent Finally, for (3), following again from the definition of $\chi$, we have $1=\chi(1)=\chi(a^{-1}a)=\chi(a^{-1})\chi(a)$. Therefore $\chi(a^{-1})=\chi(a)^{-1}$, and also, since $\chi(a)$ is a $(p-1)$st root of unity, meaning that $\chi(a)$ evaluates to $1$, its modulus must be $1$. From the fact that $x\overline{x}=1$ for some $x\in \mathbb{C}$, we have that $\chi(a)^{-1}=\overline{\chi(a)}$. 
\end{proof}
\noindent Many properties of the Legendre symbol can offer good intuition for many of the following results that we develop regarding multiplicative characters. Said properties were developed more precisely in the preliminary section of \cite{relmat2022quadrec}. 
\begin{prop} \label{proof about sum of elements evaluated by a character prop}
    Let $\chi$ be a multiplicative character. If $\chi\neq\eps$, the trivial multiplicative character, then $\sum_{t\in \mathbb{F}_{p}}\chi(t)=0$. Otherwise, the sum is $p$. 
\end{prop}
\begin{proof}
    The last assertion is as follows. Since $t$ runs through all elements of $\mathbb{F}_{p}$, we must have \begin{equation*}
        \sum_{t\in \mathbb{F}_{p}}\chi(t)=\sum_{t\in \mathbb{F}_{p}}\eps(t)=p.
    \end{equation*} 
    \noindent To prove the first assertion, we assume otherwise. Let there exist some $a\in \mathbb{F}_{p}^{*}$ such that $\chi(a)\neq 1$, or $\chi$ does not map $a$ to $1$, hence $\chi$ nontrivial. Let the desired sum be $T=\sum_{t\in \mathbb{F}_{p}}\chi(t)$. Then we may write
    \begin{equation*}
        \chi(a)T=\sum_{t\in \mathbb{F}_{p}}\chi(a)\chi(t)=\sum_{t\in \mathbb{F}_{p}}\chi(at).
    \end{equation*}
    \noindent This equates to $T$ itself as $at$ runs through the exact same number of elements from $\mathbb{F}_p$ as $t$ does, so $\chi(a)T=T$. Then $T(\chi(a)-1)=0$. We stated that necessarily $\chi(a)\neq 1$, so $T=0$, and we are finished.  
\end{proof}
\begin{rem}\label{Remark: characters form a group and special properties}
    An important fact about characters is that they form a group. Necessarily, for two nontrivial characters $\chi$ and $\lambda$, the function $\chi\lambda$ is the map that takes some $a\in \mathbb{F}_{p}^{*}$ to $\chi(a)\lambda(a)$. In order for $\chi\lambda$ to be a character it must be true that the map of the composition is a homomorphism, namely that for all $x,y\in\mathbb{F}_{p}^{*}$, we have $\chi\lambda(xy)=\chi\lambda(x)\chi\lambda(y)$. Another property of the group of characters is that if $\chi$ is some character, then $\chi^{-1}$ is the map that takes some $a\in\mathbb{F}^{*}$ to $\chi(a)^{-1}$, serving as a sort of reverse image mapping. This can be shown to be a character by similarly proving that it is a homomorphism. Finally, it is clear that the identity element of the group is the trivial character $\eps$. 
\end{rem}
\noindent The fact that characters form a group is fundamental to later results. We need the following result. 
\begin{thm} \label{multiplicative group of the integers modulo p is cyclic}
    The multiplicative group of $\mathbb{Z}/p\mathbb{Z}$, namely in our notation, $\mathbb{F}_{p}^{*}$, is cyclic. 
\end{thm}
    \noindent The proof is nearly identical to the proof of Theorem \ref{multiplicative group of finite field is cyclic}, but simpler due to the fact that we are dealing with the integers modulo $p$. 
\noindent We now show that the group of characters also forms a cyclic group of order $p-1$. 
\begin{thm} \label{the group of characters form a cyclic group of order p-1}
    The group of characters form a cyclic group of order $p-1$. Furthermore, for some $a\in\mathbb{F}_{p}^{*}$ with $a\neq 1$, then there exists a character $\chi$ such that $\chi(a)\neq 1$. 
\end{thm}
\begin{proof}
    This theorem asserts that the characters form a cyclic group, so just as before, it is necessary to show that there exists a generator that, when raised to the $(p-1)$st power, yields the identity, namely $\eps$. We first show that the order of the group is $p-1$. By Theorem \ref{multiplicative group of the integers modulo p is cyclic}, $\mathbb{F}_{p}^{*}$ is cyclic. Let some $g\in\mathbb{F}_{p}^{*}$ be a generator. Then every $a\in\mathbb{F}_{p}^{*}$ can be represented as a power of $g$. Then $a=g^{l}$ for some exponent $l$. If $\chi$ is a character on $\mathbb{F}_{p}^{*}$, then it is true that $\chi(a)=\chi(g)^{l}$. What this shows is that the the character $\chi$ is restricted explicitly by the value of $\chi(g)$, the character acting on the generator. By Proposition \ref{basic facts about multiplicative characters 1}, we know that $\chi(g)$ is a $(p-1)$st root of unity. Furthermore, by the definition of prime roots of unity, there are exactly $p-1$ roots. Therefore, the character group has order at most $p-1$.

    \indent Now we proceed to show the existence of a generator on the group of characters to show that it is cyclic. Define some function $\lambda$ as $\lambda(g^{k})=e^{2\pi i(\frac{k}{p-1})}$. The function $\lambda$ is well-defined; this can be verified by considering its properties as a function. Furthermore, $\lambda$ is also a character because it possesses the property that
    \begin{equation*}
        \lambda(g^{k})  = \underbrace{\lambda(g)\cdots \lambda(g)}_{k} = \underbrace{e^{\frac{2\pi i}{p-1}}\cdots e^{\frac{2\pi i}{p-1}}}_{k} = \bigg(e^{\frac{2\pi i}{p-1}}\bigg)^{k} = e^{\frac{2k\pi i}{p-1}},
    \end{equation*}
    \noindent thus satisfying the multiplicative property of the character. If we want to show that the group is cyclic, we need to show that $p-1$ is the smallest integer $n$ such that $\lambda^{n}=\eps$. If $\lambda^{n}=\eps$, then we must have $\lambda^{n}(g)=\eps(g)=1$ for some $g$. Alternatively, we also have $\lambda^{n}(g)=(\lambda(g))^{n}=(e^{\frac{2\pi i}{p-1}})^{n}=e^{\frac{2n\pi i}{p-1}}$. For this to be equivalent to $1$, we must have that $p-1|n$. For some $a$, we have $\lambda^{p-1}(a)=\lambda(a)^{p-1}=\lambda(a^{p-1})=1$, which is only possible if $\lambda^{p-1}=\eps$. Repeating this process, we can show that $\eps,\lambda,\ldots,\lambda^{p-2}$ are distinct, and thus combined with the the fact that the group has a maximum order of $p-1$, we have that the group of characters is cyclic with generator $\lambda$. 

    \indent Finally, to prove the final part, let there be some $a\in\mathbb{F}_{p}^{*}$ with $a\neq 1$. Then $a$ can be represented as a power of the generator $g$, so $a=g^{l}$, with $p-1|l$. Then applying $\lambda$ to $a$ we have $\lambda(a)=\lambda(g^{l})=\lambda(g)^{l}=(e^{\frac{2\pi i}{p-1}})^{l}=\lambda(g)^{l}=e^{\frac{2l\pi i}{p-1}}$. Since $p-1\nmid l$, it must be true that $\lambda(a)\neq 1$, so we are done. 
\end{proof}
\noindent As an analogue to Proposition \ref{proof about sum of elements evaluated by a character prop} and as a result of Theorem \ref{the group of characters form a cyclic group of order p-1}, we can also consider summing over all characters and evaluating each one at a fixed variable from $\mathbb{F}_{p}^{*}$. This gives us the following proposition. 
\begin{prop}
    Let $a\in \mathbb{F}_{p}^{*}$ with $a\neq 1$. Then $\sum_{\chi}\chi(a)=0$ over all characters $\chi$.
\end{prop}
\begin{proof}
    Let us denote the sum above with $S=\sum_{\chi}\chi(a)$. Since we assumed that $a\neq 1$, Theorem \ref{the group of characters form a cyclic group of order p-1} asserts that there must exist some other character $\lambda$ such that $\lambda(a)\neq 1$. Then we have 
    \begin{equation*}
        \lambda(a)S = \lambda(a)\sum_{\chi}\chi(a) = \sum_{\chi}\lambda(a)\chi(a) = \sum_{\chi}\lambda\chi(a)
    \end{equation*}
    \noindent by our discussion in Remark \ref{Remark: characters form a group and special properties}. Note that $\lambda\chi$ runs over an equivalent number of characters as $\chi$ from the group of characters, so we can assert that $\sum_{\chi}\lambda\chi(a)=S$ as we defined earlier. Therefore $\lambda(a)S=S$, so $(\lambda(a)-1)S=0$, which since $\lambda(a)\neq 1$, is only possible if $S=0$. 
\end{proof}

\subsection{Gauss Sums}
\noindent As we have mentioned before, the multiplicative character generalizes the notion of the Legendre symbol. Similarly, it also generalizes the notion of a quadratic Gauss sum to the notion of a general Gauss sum. We define the Gauss sum as follows.
\begin{defn}[Gauss Sum]
    Let $\chi$ be some character on $\mathbb{F}_{p}$ and let $a\in\mathbb{F}_{p}$. Let 
    \begin{equation*}
        g_{a}(\chi)=\sum_{t\in\mathbb{F}_{p}}\chi(t)\zeta_{p}^{at},
    \end{equation*}
    \noindent where $\zeta_{p}=e^{2i\pi/p}$ is a $p$th root of unity. We say that $g_{a}(\chi)$ is a Gauss sum on $\mathbb{F}_{p}$ belonging to the character $\chi$. 
\end{defn}
\noindent With this definition, we can see that a each Gauss sum is over a unique character, so we are dealing with sums of a singular character evaluated at all values of $\mathbb{F}_p$ equipped with an additional parametrization of $t$. We examine some basic properties of the Gauss sum.
\begin{lemma} \label{lemma with some important properties of the Gauss sum and its evaluation}
The following are true. 
    \begin{enumerate}
        \item If $a\neq 0$ and $\chi\neq \eps$, then $g_{a}(\chi)=\overline{\chi(a)}g_{1}(\chi)$. 
        \item If $a\neq 0$ and $\chi = \eps$ then $g_{a}(\eps)=0$.
        \item If $a=0$ and $\chi\neq \eps$, then $g_{a}(\chi)=0$.
        \item If $a=0$ and $\chi=\eps$, then $g_{a}(\chi)=p$. 
    \end{enumerate}
\end{lemma}
\begin{proof}
    \noindent Let us begin by proving (1). Let $a\neq 0$ and $\chi\neq \eps$. Then 
    \begin{equation*}
        \chi(a)g_{a}(\chi) = \chi(a)\sum_{t\in\mathbb{F}_{p}}\chi(t)\zeta_{p}^{at} = \sum_{t\in\mathbb{F}_{p}}\chi(a)\chi(t)\zeta_{p}^{at} = \sum_{t\in\mathbb{F}_{p}}\chi(at)\zeta_{p}^{at} = g_{1}(\chi).
    \end{equation*}
    \noindent Then $\chi(a)g_{a}(\chi) = g_{1}(\chi)$, so $g_{a}=g_{1}(\chi)\chi(a)^{-1} = \chi(a^{-1})g_{1}(\chi) = \overline{\chi(a)}g_{1}(\chi)$.

    \indent We now prove (2). Let $a\neq 0$ but $\chi=\eps$. Since $\eps$ maps all $a\in\mathbb{F}_{p}$ to $1$, we have 
    \begin{equation*}
        g_{a}(\eps) = \sum_{t\in\mathbb{F}_{p}}\eps(t)\zeta_{p}^{at} = \sum_{t\in\mathbb{F}_{p}}\zeta_{p}^{at}.
    \end{equation*}
    \noindent Recall that $\mathbb{F}_{p}$ is the integers modulo $p$, so $t$ runs through all residue class representatives. Namely, it goes from $t=0$ to $p-1$. Therefore $\sum_{t\in\mathbb{F}_{p}}\zeta_{p}^{at} = \sum_{t=0}^{p-1}\zeta_{p}^{at}$. Since $a\neq 0$, we consider two cases: (a) when $a\equiv 0 \pmod{p}$ and (b) when $a\not\equiv 0\pmod{p}$. Considering (a), if $a\equiv 0 \pmod{p}$, then for some $k\in\mathbb{Z}$, we have $\zeta_{p}^{a} = (e^{2i\pi/p})^{kp} = e^{2ki\pi}=1$ for all values of $k$. Then $\sum_{t=0}^{p-1}(\zeta_{p}^{a})^{t}=1+\cdots+1=p$. We now consider (b). If $a\not\equiv 0\pmod{p}$, then we can evaluate the sum as a finite geometric series. Then
    \begin{equation*}
        \sum_{t=0}^{p-1}\zeta_{p}^{at} = \sum_{t=1}^{p}\zeta_{p}^{at} = \frac{1(1-\zeta_{p}^{ap})}{1-\zeta_{p}^{a}} = \frac{\zeta_{p}^{ap}-1}{\zeta_{p}^{a}-1}.
    \end{equation*}
    \noindent We know that $\zeta_{p}^{ap}=1$ for all $p$ prime, so $\frac{\zeta_{p}^{ap}-1}{\zeta_{p}^{a}-1}=0/(\zeta_{p}^{a}-1)=0$.

    \indent To prove (3), let $a=0$ and $\chi\neq \eps$. Then $g_{0}(\chi)=\sum_{t\in\mathbb{F}_{p}}\chi(t)=0$ by Proposition \ref{proof about sum of elements evaluated by a character prop}.

    \indent To prove (4), let both $a=0$ and $\chi=\eps$. Then $g_{0}(\eps)=\sum_{t\in\mathbb{F}_{p}}\eps(t) = \underbrace{1+\cdots+1}_{p}=p$. 
\end{proof}
\noindent In our proof of (2), we split the evaluation of the sum into two cases with dependence on the value of $a$. This result can be summarized as follows.
\begin{lemma} \label{result about zeta from special property prop about Gauss sums}
    \begin{equation*}
        \sum_{t=0}^{p-1}\zeta_{p}^{at} = \left\{ \begin{array}{ll}
            p, & a\equiv 0\pmod{p}, \\
            0, & a\not\equiv 0\pmod{p}.
        \end{array}
        \right.
    \end{equation*}
\end{lemma}

\noindent Suppose we wish to determine the absolute value of the Gauss sum if the character is nontrivial. In prior literature it is possible to compute the value of the quadratic Gauss sum as well as its sign. Though we do not humor the intricacies of that proof as they lie beyond the scope of the paper, we compute the value of the general Gauss sum in Lemma \ref{Lemma: value of the general Gauss sum}. The following simple result is useful in proving the lemma and following results. 
\begin{cor}[Corollary to Lemma \ref{result about zeta from special property prop about Gauss sums}]\label{corollary to Lemma 3.7}
    \begin{equation*}
     p^{-1}\sum_{t=0}^{p-1}\zeta_{p}^{t(x-y)}  = \delta(x,y).
\end{equation*}
\end{cor}

\begin{proof}
    The proof follows by considering each case and evaluating accordingly.
\end{proof}

\subsection{Jacobi Sums}
\noindent The theory of Jacobi sums extends far beyond what we will discuss here, specifically in regard to solving Diophantine equations, but basic properties of the Jacobi sum will be useful later. We define a Jacobi sum as follows.
\begin{defn}[Jacobi Sum]
    Let $\chi$ and $\lambda$ be two characters on $\mathbb{F}_{p}$. Then we define a Jacobi sum over $\chi$ and $\lambda$ to be
    \begin{equation*}
        J(\chi,\lambda) = \sum_{\scriptsize\begin{aligned}a+b&=1 \\[-4pt] a,b &\in \mathbb{F}_{p}\end{aligned}}\chi(a)\lambda(b),
    \end{equation*}
    where $a,b\in \mathbb{F}_{p}$.
\end{defn}
\noindent The following theorem relates Jacobi sums to Gauss sums. 
\begin{thm} \label{basic properties of the jacobi sum and relation to Gauss sum}
    Let $\chi$ and $\lambda$ be characters such that neither is the trivial character $\eps$. Then
    \begin{enumerate}
        \item $J(\eps,\eps)=p$,
        \item $J(\eps,\chi)=0$,
        \item $J(\chi,\chi^{-1})=-\chi(-1)$,
        \item If the composition $\chi\lambda\neq \eps$, then 
        \begin{equation*}
            J(\chi,\lambda) = \frac{g(\chi)g(\lambda)}{g(\chi\lambda)}.
        \end{equation*}
    \end{enumerate}
\end{thm}
\begin{proof}
    \noindent Note that $x,y\in\mathbb{F}_{p}$ in the following.
    
    \indent (1) follows from the fact that for all $\alpha$ we have $\eps(\alpha)=1$, so that $J(\eps,\eps) = \sum_{a+b=1}\chi(a)\lambda(b) = \sum_{a+b=1}1 = p$ since $\mathbb{F}_{p}$ has order $p$. 

    \indent For (2), we have $J(\eps,\chi) = \sum_{a+b=1}\eps(a)\chi(b) = \sum_{a+b=1}\chi(b)$, which equates to $0$ as an extension of the result in Proposition \ref{proof about sum of elements evaluated by a character prop}.

    \indent For (3), we have 
    \begin{equation*}
        J(\chi,\chi^{-1}) = \sum_{a+b=1}\chi(a)\chi^{-1}(b) = \sum_{\scriptsize\begin{aligned}a+b&=1, \\[-4pt] b &\neq 0 \end{aligned}}\chi(ab^{-1}) = \sum_{a\neq 1}\chi\bigg( \frac{a}{1-a} \bigg).
    \end{equation*}
    \noindent Let $q=\sum_{a\neq 1}\chi(\frac{a}{1-a})$. With the constraint that $c\neq 1$, we can express $a=\frac{c}{1+c}$. As the value of $a$ runs over the entire field $\mathbb{F}_{p}$, with the exception that $a\neq 1$, simultaneously $c$ also varies over $\mathbb{F}_{p}$, with the exception that $c\neq 1$. Therefore, by Proposition \ref{basic facts about multiplicative characters 1} we can take the sum
    \begin{equation*}
        J(\chi,\chi^{-1}) = \sum_{c\neq -1}\chi(c) = \sum_{c}\chi(c)-\sum_{c=-1}\chi(c) = -\chi(-1). 
    \end{equation*}
    \noindent For (4), we have
    \begin{align*}
        \label{eq:1} g(\chi\lambda) = g(\chi)g(\lambda) & =\bigg( \sum_{x}\chi(x)\zeta^{x} \bigg)\bigg( \sum_{y}\lambda(x)\zeta^{y} \bigg) \\ & = \sum_{x,y}\chi(x)\lambda(y)\zeta^{x+y} \\ & =  \sum_{t\in\mathbb{F}_{p}}\bigg(\sum_{x+y=t}\chi(x)\lambda(y)\bigg)\zeta^{t}.  
    \end{align*}
    \noindent We consider two cases for the value of $t$. If $t=0$, then, choosing to sum over $x$ and by the fact that the composition $\chi\lambda\neq \eps$, clearly 
    \begin{equation*}
        \sum_{x+y=0}\chi(x)\lambda(y) = \sum_{x}\chi(x)\lambda(-x) = \sum_{x}\lambda(-1)\chi(x)\lambda(x) = \lambda(-1)\sum_{x}\chi\lambda(x) = 0
    \end{equation*}
    \noindent by Proposition \ref{proof about sum of elements evaluated by a character prop}. In the case that $t\neq 0$, define two new elements $x'$ and $y'$ as $x=tx'$ and $y=ty'$. Then, if we have $x+y=t$, then substituting we have $tx'+ty'=t$, so that $x'+y'=1$. Therefore 
    \begin{align*}
        \sum_{x+y=t}\chi(x)\lambda(y) & = \sum_{x'+y'=1}\chi(tx')\lambda(ty') = \sum_{x'+y'=1}\chi(t)\lambda(t)\chi(x')\lambda(y')  = \sum_{x'+y'=1}\chi\lambda(t)\chi(x')\lambda(y') \\ & = \chi\lambda(t)J(\chi,\lambda).
    \end{align*}
    \noindent If we substitute this into our evaluation of $g(\chi)g(\lambda)$, then we have 
    \begin{equation*}
        g(\chi)g(\lambda) = \sum_{t\in\mathbb{F}_{p}}\chi\lambda(t)J(\chi,\lambda)\zeta^{t} = J(\chi,\lambda)\sum_{t\in\mathbb{F}_{p}}\chi\lambda(t)\zeta^{t} = J(\chi,\lambda)g(\chi\lambda). 
    \end{equation*}
    \noindent Setting this equal to $g(\chi)g(\lambda)$ and dividing by $g(\chi\lambda)$, we have 
    \begin{equation*}
        J(\chi,\lambda) = \frac{g(\chi)g(\lambda)}{g(\chi\lambda)}.
    \end{equation*}
\end{proof}

\noindent Before we can proceed, we require several technical lemmas. We have the following.
\begin{lemma}\label{Lemma: value of the general Gauss sum}
    If $\chi\neq \varepsilon$ is a nontrivial character, then $|g(\chi)|^{2}=p$. 
\end{lemma}
\begin{proof}
    The proof is very similar to proofs regarding the quadratic Gauss sum discussed in \cite{relmat2022quadrec}. We want to evaluate the sum 
    \begin{equation*}
        \sum_{a\in\mathbb{F}_{p}}g_{a}(\chi)\overline{g_{a}(\chi)}
    \end{equation*}
    \noindent in two different ways. We first want to evaluate the argument of the sum. Assume that $a\neq 0$. By (1) of Lemma \ref{lemma with some important properties of the Gauss sum and its evaluation}, we can write 
    \begin{equation*}
        \overline{g_{a}(\chi)} = \overline{\chi(a^{-1})g(\chi)} = \chi(a)\overline{g(\chi)}.
    \end{equation*}
    \noindent Taking the conjugate, we also have $g_{a}(\chi) = \chi(a^{-1})g(\chi)$. Multiplying,
    \begin{equation*}
    \chi(a)\overline{g(\chi)}\chi(a^{-1})g(\chi) = \overline{g(\chi)}g(\chi) = |g(\chi)|^{2}.
    \end{equation*}
    \noindent Since $\sum_{a\in\mathbb{F}_{p}}$ sums over all elements of the finite field $\mathbb{F}_{p}$ except $a=0$, we consider this quantity $p-1$ times, so 
    \begin{equation*}
        \sum_{a\in\mathbb{F}_{p}}g_{a}(\chi)\overline{g_{a}(\chi)} = (p-1)|g(\chi)|^{2}.
    \end{equation*}
    \noindent Similarly, considering two parameters $x$ and $y$ and writing the argument of the sum as a double sum, we have
    \begin{equation*}
        g_{a}(\chi)\overline{g_{a}(\chi)} = \sum_{x\in\mathbb{F}_{p}}\sum_{y\in\mathbb{F}_{p}}\chi(x)\overline{\chi(y)}\zeta^{ax-ay}.
    \end{equation*}
    \noindent Summing over all elements of $\mathbb{F}_{p}$ and applying Corollary \ref{corollary to Lemma 3.7} we have
    \begin{align*}
        \sum_{a\in\mathbb{F}_{p}}\sum_{x\in\mathbb{F}_{p}}\sum_{y\in\mathbb{F}_{p}}\chi(x)\overline{\chi(y)}\zeta^{ax-ay} & = pp^{-1}\sum_{a\in\mathbb{F}_{p}}\sum_{x\in\mathbb{F}_{p}}\sum_{y\in\mathbb{F}_{p}}\chi(x)\overline{\chi(y)}\zeta^{ax-ay}p \\
        & = p\sum_{x\in\mathbb{F}_{p}}\sum_{y\in\mathbb{F}_{p}}\chi(x)\overline{\chi(y)}\zeta^{ax-ay}\delta(x,y)
    \end{align*}
    \noindent where $\delta(x,y)$ denotes the Kronecker delta. If $x\not\equiv y\pmod{p}$ then the double sum will equate to $0$, so we consider when $x\equiv y\pmod{p}$. If this is true, then the argument of the sum will run over exactly $p-1$ elements, so that 
    \begin{equation*}
        p\sum_{x\in\mathbb{F}_{p}}\sum_{y\in\mathbb{F}_{p}}\chi(x)\overline{\chi(y)}\zeta^{ax-ay}\delta(x,y) = p(p-1).
    \end{equation*}
    \noindent Equating our two evaluations, we have 
    \begin{align*}
        (p-1)|g(\chi)|^{2} & = p(p-1) \\
        |g(\chi)|^{2} & = p,
    \end{align*}
    \noindent so we are done. 
\end{proof}
\noindent It is important to notice that the same result holds when $g(\overline{\chi})$ is considered instead, i.e. $|g(\overline{\chi})|^{2}=p$. This is evident in the following. 
\begin{cor}\label{Corollary to Lemma about value of general Gauss sum}
    \begin{equation*}
        g(\chi)g(\overline{\chi}) = \chi(-1)p.
    \end{equation*}
\end{cor}
\begin{proof}
    Note first that $\chi(-1)=\overline{\chi(-1)}$ since both values equate to $\pm 1$. Taking the conjugate of the equivalence given in Lemma \ref{Lemma: value of the general Gauss sum}, we have
    \begin{align*}
        \overline{\overline{g(\chi)}}& = \overline{\chi(-1)g(\overline{\chi})} \\ 
        g(\chi) & = \chi(-1)\overline{g(\overline{\chi})} \\ 
        g(\chi)g(\overline{\chi}) & = \chi(-1)\overline{g(\overline{\chi})}g(\overline{\chi}) = \chi(-1)|g(\overline{\chi})|^{2} = \chi(-1)p,
    \end{align*}
    \noindent thereby proving the result.
\end{proof}

\noindent The following result is general for characters, but will be useful when considering the relation between Gauss sums and Jacobi sums.  
\begin{thm} \label{characters of degree n>2 expression in terms of Gauss and Jacobi sum}
    It is well known that there exist an infinite number of primes of the form $p\equiv 1 \pmod{n}$ since all primes are odd. As such, let $\chi$ be a character of order $n>2$. Then
    \begin{equation*}
        g(\chi)^{n} = \chi(-1)p(\chi,\chi)J(\chi,\chi^{2})\cdots J(\chi,\chi^{n-2}).
    \end{equation*}
\end{thm}
\begin{proof}
    We can express (4) of Theorem \ref{basic properties of the jacobi sum and relation to Gauss sum} as $J(\chi,\chi)g(\chi\cdot\chi) = g(\chi)g(\chi)$. This gives $g(\chi)^{2} = J(\chi,\chi)g(\chi^{2})$. In the $n=3$ case, multiply both sides by $g(\chi)$ and we have 
    \begin{equation*}
        g(\chi)^{3} = J(\chi,\chi)g(\chi)g(\chi^{2}) = J(\chi,\chi)J(\chi,\chi\cdot\chi)g(\chi^{3}) = J(\chi,\chi)J(\chi,\chi^{2})g(\chi^{3}).
    \end{equation*}
    \noindent We can continue multiplying the \textit{LHS} and \textit{RHS} by $g(\chi)$, so eventually, we have the $(n-1)$th case, so that 
    \begin{equation*}
        g(\chi)^{n-1} = J(\chi,\chi)J(\chi,\chi^{2})\cdots J(\chi,\chi^{n-2})J(\chi,\chi^{n-1}).
    \end{equation*}
    \noindent Notice, however, by the fact that characters form a cyclic group, that $g(\chi)^{n-1}=g(\chi)^{n}g(\chi)^{-1} = g(\chi)^{-1}=g(\overline{\chi})$. Therefore we have $g(\chi)g(\chi^{n-1}) = g(\chi)g(\overline{\chi}) = \chi(-1)p$ by Corollary \ref{Corollary to Lemma about value of general Gauss sum}. Multiplying both sides of the equation above by $g(\chi)$, we have
    \begin{equation*}
        g(\chi)^{n} = \chi(-1)p(\chi,\chi)J(\chi,\chi^{2})\cdots J(\chi,\chi^{n-2}),
    \end{equation*}
    \noindent so we are done. 
\end{proof}
\noindent What follows from this is a corollary concerning the relationship between the cubic Gauss sum and the Jacobi sum. 
\begin{cor} \label{corollary for relation between cubic Gauss sum and jacobi sum}
    Let $\chi$ be the cubic character. Then
    \begin{equation*}
        g(\chi)^{3} = pJ(\chi,\chi). 
    \end{equation*}
\end{cor}
\begin{proof}
    This is a special case of Theorem \ref{characters of degree n>2 expression in terms of Gauss and Jacobi sum}. Take $n=3$. Then, since $-1$ is clearly a cube, namely $\chi(-1)=\chi((-1)^{3})=1$, we have
    \begin{equation*}
        g(\chi)^{3} = \chi(-1)pJ(\chi,\chi) = pJ(\chi,\chi).
    \end{equation*}
\end{proof}
\noindent These results will be of utmost importance later when we seek to prove cubic reciprocity. 

\section{Cubic Reciprocity}
\noindent Whereas quadratic reciprocity resides in $\ZZ$, cubic reciprocity resides in $\ZZ[\omega]$, the Eisenstein integers. In order to state and then prove cubic reciprocity, it is necessary to take a step back and examine characteristics of the ring $\mathbb{Z}[\omega]$. We will begin our examination by looking at the prime elements and units of $\ZZ[\omega]$. 

\subsection{Units and Primes in $\mathbb{Z}[\omega]$}
\noindent We say that some $\omega$ is a cube root of unity, and that the set $\mathbb{Z}[\omega]$ contains complex number elements of the form $\alpha=a+b\omega$ for $a,b\in\mathbb{Z}$. It is in fact true that $\mathbb{Z}[\omega]$ is a ring. In our case, we consider when $\omega=-1/2+i\sqrt{3}/2$. 

\indent We say that the \textit{norm} of $\alpha$, written as $N\alpha$, is the product of $\alpha$ and its conjugate, namely $N\alpha=\alpha\overline{\alpha}=a^{2}-ab+b^{2}$. We may now consider the prime elements and units of $\mathbb{Z}[\omega]$. Note that this is possible explicitly because there inherently exists a notion of unique factorization in $\mathbb{Z}[\omega]$, and therefore we are able to consider the supposed ``building blocks" of the ring. By means of notational convention, we denote $\mathbb{Z}[\omega]=D$. 

\begin{prop} \label{classification of units in Z[omega]}
    Some $\alpha\in D$ is a unit of $D$ if and only if $N\alpha=1$. Furthermore, the units of $D$ are $1,-1,\omega,-\omega,\omega^{2},-\omega^{2}$.
\end{prop}
\begin{proof}
    We prove the backward direction first. If $N\alpha=1$, then by the definition of the norm, $\alpha\overline{\alpha}=1$, which implies that $\alpha$ must be a unit. 

    \indent To prove the reverse direction, suppose that $\alpha$ is a unit. By the definition of the unit, this must mean there exists some $\beta\in D$ such that $\alpha\beta=1$. Therefore $N\alpha N\beta=1$. However, since $N\alpha$ and $N\beta$ must be integers, in order for their product to be $1$, they must both be $1$. Therefore $N\alpha=1$. 

    \indent We now determine the units of $D$. Suppose that $\alpha=a+b\omega\in D$ is a unit. Then by definition of the norm we have $1=a^{2}-ab+b^{2}$. We may rewrite this as $4=4a^{2}-4ab+4b^{2}$, or $4=4a^{2}-4ab+b^{2}+3b^{2}$ so $4=(2a-b)^{2}+3b^{2}$. If we observe this Diophantine equation, we can see that there are only two sets of possible solutions. We write them as follows.
    \begin{enumerate}
        \item $2a-b=\pm 1, b=\pm 1$,
        \item $2a-b=\pm 2, b=0$.
    \end{enumerate}
    \noindent We need to solve all 6 possible pairs of equations for $a$ and $b$. The first is $2a-b=1$ and $b=1$. Then $a=b=1$. The second is $2a-b=-1$ and $b=-1$. Then $a=b=-1$. The third is $2a-b=1$ and $b=-1$. Then $a=0$ and $b=-1$. The fourth is $2a-b=-1$ and $b=1$. Then $a=-1$ and $b=1$. The fifth is $2a-b=2$ and $b=0$. Then $a=1$ and $b=0$. The final is $2a-b=-2$ and $b=0$. Then $a=-1$ and $b=0$. Plugging these into the expression for $\alpha$, we have the units $1+\omega,-1-\omega,-\omega,\omega,1,-1$. An identity asserts that $\omega^{2}+\omega+1=0$, so rewriting, we can express the first two units as $-\omega^{2}$ and $\omega^{2}$ respectively. 
\end{proof}
\noindent Now that we have determined the units, we begin to observe characteristics of prime elements in $\mathbb{Z}[\omega]$.
\begin{rem}
     Note that primes in $\mathbb{Z}$ are not necessarily prime in $D$. For example, consider the prime $7$. We can express it as $7=(3+\omega)(2-\omega)=6-3\omega+2\omega-\omega^{2}=6-\omega-\omega^{2}=7$. Therefore, in order to distinguish between integer primes and primes in $D$, we refer to primes in $D$ as primes and primes in $\mathbb{Z}$ as \textit{rational primes}.
\end{rem}

\begin{prop}
    Let $\pi$ be a prime in $D$. Then there is some rational prime $p$ such that $N\pi=p$ or $p^{2}$. If $N\pi=p$ then $\pi$ is not associate to a rational prime, and if $N\pi=p^{2}$ then $\pi$ is associate to a rational prime. 
\end{prop}
\begin{proof}
    By the definition of the norm, we must have that $N\pi=\pi\overline{\pi}=n>1$. Clearly, by the fundamental theorem of arithmetic, $n$ is a product of rational primes. Therefore $\pi|p$ for some rational prime $p$. Since $\pi|p$, there exists some $\gamma\in D$ such that $p=\pi\gamma$. Then we may write $N\pi N\gamma=N\pi\gamma=Np=p(p)=p^{2}$. For this equality to be true we must have either $N\pi=p^{2}$ with $N\gamma=1$ or that $N\pi=N\gamma=p$. We consider the first case. If $N\pi=p^{2}$ and $N\gamma=1$, then $\gamma$ must be a unit of $D$, so $\pi$ is associate to $p$. In the second case, suppose that $\pi$ is associate to some other rational prime $q\in\mathbb{Z}$, with $u$ a unit. Then we must have $\pi=uq$. Then $N\pi=Nuq=NuNq=1(q)(q)=q^{2}$. Clearly, a rational prime cannot be a square of another rational prime, so $\pi$ must not be associate to a rational prime. Thus we are done. 
\end{proof}
\begin{prop} \label{some prime pi in D such that its norm is p, then pi is prime}
    If there is some $\pi\in D$ such that $N\pi=p$ where $p$ is a rational prime, then $\pi$ is a prime in $D$. 
\end{prop}
\begin{proof}
    Assume that $\pi$ is not prime in $D$. Then we can express $\pi$ as a product of primes $\rho,\gamma\in D$, such that for $N\rho,N\gamma>1$, we have $p=N\pi=N\rho\gamma=N\rho N\gamma$. However, since $p$ is itself prime in $\mathbb{Z}$, this argument is not possible. Therefore $\pi$ is prime in $D$. 
\end{proof}
\noindent Now that we have shown several properties of primes and units in $D$, we might be interested in classifying its primes. 
\begin{thm} \label{classification of primes in Z[omega]}
    Let $p$ and $q$ be rational primes. 
    \begin{enumerate}
        \item If $q\equiv 2 \pmod{3}$ then $q$ is prime in $D$. 
        \item If $p\equiv 1 \pmod{3}$, then $p=N\pi=\pi\overline{\pi}$, where $\pi$ is a prime in $D$. 
        \item $3=-\omega^{2}(1-\omega)^{2}$, and $1-\omega$ is prime in $D$. 
    \end{enumerate}
\end{thm}
\begin{proof}
    \noindent We begin by proving (1). Suppose that $p$ is not a rational prime. Then we can write $p=\pi\gamma$ for $\pi,\gamma\in D$ and $N\pi,N\gamma>1$. Then taking the norm of both sides, we have $Np=N\pi N\gamma$ so $p^{2}=N\pi N\gamma$. Therefore we can write $N\pi=p$. Let $\pi\in D$ be of the from $\pi=a+b\omega$. Then we may write $N\pi=a^{2}-ab+b^{2}=p$. Using the same factorization in Proposition \ref{classification of units in Z[omega]}, we have $4p=(2a-b)^{2}+3b^{2}$. Reducing modulo 3, we have $p\equiv (2a-b)^{2} \pmod{3}$. If $3\nmid p$, then it must be true that $p\equiv 1\pmod{3}$ because $1$ is the only integer such that it is a nonzero square modulo 3. In other words, $1$ is a quadratic residue modulo 3. If we again look at $a^{2}-ab+b^{2}$ and substitute all pairs $(a,b)$ where $a,b\in\ZZ/3\ZZ$, we can see that it is impossible for $a^{2}-ab+b^{2}\equiv 2\pmod{3}$. Note that no prime is congruent to $0$ modulo $3$. Therefore there must exist some rational prime $q\equiv 2\pmod{3}$ that is prime in $D$. 

    \indent We now prove (2). Suppose that $p\equiv 1 \pmod{3}$. Using Theorem \ref{the law of quadratic reciprocity}, we have
    \begin{align*}
        \bigg( \frac{-3}{p} \bigg) & = \bigg( \frac{-1}{p} \bigg) \bigg( \frac{3}{p} \bigg) = (-1)^{\frac{p-1}{2}}\bigg( \frac{p}{3} \bigg)(-1)^{\frac{p-1}{2}\cdot \frac{3-1}{2}} \\ & = (-1)^{\frac{p-1}{2}+\frac{p-1}{2}}\bigg( \frac{p}{3} \bigg) = \bigg( \frac{p}{3} \bigg) = \bigg( \frac{1}{3} \bigg) = 1. 
    \end{align*}
    \noindent Therefore, $-3$ is a quadratic residue modulo 3, meaning that there exists some $a\in\mathbb{Z}$ such that $a^{2}\equiv -3 \pmod{3}$. We may rewrite this as $pb=a^{2}+3$ for some $b\in\mathbb{Z}$. Factorizing, we have that $p$ divides $pb=(a-\sqrt{-3})(a+\sqrt{-3})=(a+1+2\omega)(a-1-2\omega)$. Since $D$ is a UFD, by the class inclusions mentioned in Remark \ref{Remark: class inclusions}, $D$ is also an integral domain. Therefore an analogue of Euclid's lemma applies, so $p$ must divide one of the factors. Assume that $p$ is a prime. Since $p\neq 2$, it cannot divide the first factor. Furthermore, $2/p$ is rational, and so it cannot divide the second factor. Therefore $p$ is not prime and can be expressed as $p=\pi\gamma$ for $\pi,\gamma\in D$ nonunits (as this guarantees that one is not associate to the other). Therefore, taking the norm of both sides, we have $p^{2}=N\pi\gamma=N\pi N\gamma$, so that $p=N\pi = \pi\overline{\pi}$. 

    \indent We now prove (3). Note that we can factorize $x^{3}-1=(x-1)(x-\omega)(x-\omega^{2})$. Therefore since $x^{3}-1=(x-1)(x^{2}+x+1)$, we have that $x^{2}+x+1=(x-\omega)(x-\omega^{2})$. Letting $x=1$, we have $3=(1-\omega)(1-\omega^{2})=(1-\omega)(1-\omega)(1+\omega)=(1+\omega)(1-\omega)^{2}$. Recalling that $1+\omega=-\omega^{2}$, we have $3=-\omega^{2}(1-\omega)^{2}$. Taking the norm of both sides, we obtain
    \begin{align*}
        3(3) & = N(-\omega^{2})N(1-\omega)^{2} \\
        9 & = N(\omega^{3})N(1-\omega)^{2} \\
        3 & = N(1-\omega),
    \end{align*}
    \noindent so we are done.
\end{proof}

\noindent Now that we are equipped with information about prime and unitary elements of $\ZZ[\omega]$, we are prepared to introduce the second most important theorem in this paper. 

\subsection{The Residue Class Ring $\mathbb{Z}[\omega]/\pi \mathbb{Z}[\omega]$ Is a Finite Field for $\pi$ Prime}
\noindent In this subsection, we prove the most important connections between finite fields and reciprocity and outline why we needed to survey finite fields in so much depth. Much like in section 4.1 of \cite{relmat2022quadrec}, it is useful to think about notions of congruence in $D$ as well. Let $\alpha,\beta,\gamma\in D$ with $\gamma$ nonzero and a nonunit. Then we say that $\alpha\equiv\beta\pmod{p}$ if $\gamma|\alpha-\beta$. Note that this definition is very similar to the definition of congruence in $\mathbb{Z}$. Each residue class modulo $\gamma$ may be combined in a quotient ring of the form $D/\gamma D$. To see how this is analogous to the conventional construction of a quotient ring, notice that $D/\gamma D$ runs through moduli $0,1,2,\ldots,\gamma-1$, with each residue class forming a set of congruent numbers in $D$, much like how $\mathbb{Z}/n\mathbb{Z}$ runs through moduli $0,1,2,\ldots,n-1$. We refer to such a quotient ring as a \textit{residue class ring modulo $\gamma$}. We now prove an important property of residue class rings. 
\begin{thm} \label{Theorem: Residue Class Ring D/piD has Npi elements (D=Eisenstein Integers)}
    Let some $\pi\in D$ be a prime. Then the residue class ring $D/\pi D$ is a finite field with $N\pi$ elements. 
\end{thm}
\begin{proof}
    \noindent In order to prove that $D/\pi D$ is a finite field, we need to first prove that it is a field, and then show that it has a finite number of elements; in this case, exactly $N\pi$. The second part of the proof requires that we consider all possible cases in Theorem \ref{classification of primes in Z[omega]}. 

    \indent We begin by showing that $D/\pi D$ is a field. Clearly, properties such as commutativity, transitivity, etc. natural to fields are present in $D/\pi D$. We need only show the existence of a unit in $D/\pi D$. Let there exist some $\alpha\in D$ with the property that $\alpha\not\equiv 0 \pmod{\pi}$. Since $D/\pi D$ is a commutative ring, it is an integral domain. Therefore for some elements $\beta,\gamma\in D$ it is possible to write $1=\beta\alpha+\gamma\pi$. Reducing both sides modulo $\pi$, we have $\alpha\beta \equiv 1 \pmod{\pi}$. If we rewrite this congruence as an equality, we have that $[\alpha\beta]=[\alpha][\beta]=1$, which is the requirement for $[\alpha]$ to be a unit in $D/\pi D$. Therefore $D/\pi D$ is a field. 

    \indent We now prove that $D/\pi D$ has $N\pi$ elements by considering all possible cases of Theorem \ref{classification of primes in Z[omega]}. We begin by supposing that $q\equiv 2 \pmod{3}$, where $q\in\mathbb{Z}$ is a rational prime. In order to show that there are $N\pi$ elements, we need to show that there is a complete set of coset, or residue class, representatives that has cardinality $N\pi$. Therefore, we need to show that the set $\{ a+b\omega | 0\leq a<q \wedge 0\leq b<q \}$ is a complete set of coset representatives for $D/qD$ with cardinality $Nq=q^{2}$. Suppose that there is some $\mu=m+n\omega\in D$ for $m,n\in\mathbb{Z}$. Then by the division algorithm, we can express $m=qs+a$ and $n=qt+b$, where $s,t,a,b\in\mathbb{Z}$ with the constraint that $a,b\in[0,q)$. By our definition of $\mu$, it is clear that $\mu$ belongs to its own residue class modulo $q$, namely $\mu\equiv a+b\omega \pmod{q}$. We want to show that every coset representative $\mu$ can be constructed in this form, and is unique. Suppose that for $0\leq a,b,a',b'<q$ that $a+b\omega\equiv a'+b'\omega \pmod{q}$. Rearranging, we have that 
    \begin{align*}
        a-a'+b\omega-b'\omega & \equiv 0\pmod{q} \\
        (a-a') + (b-b')\omega & \equiv 0\pmod{q} \\
        \frac{a-a'}{q}+\frac{b-b'}{q}\omega & \in D.
    \end{align*}
    \noindent By the definition of elements of $D$, this means $(a-a')/q$ and $(b-b')/q$ are both integers, but since they are rational it is only possible if $a=a'$ and $b=b'$, thus proving uniqueness. 

    \indent Now we show that this is true when $p\equiv 1 \pmod{3}$ is a rational prime, and $p=\pi\overline{\pi}=N\pi$. Much like the first part of the proof, we want to show that a set is a complete set of coset representatives, but alternatively in the form $\{ 0,1,2,\ldots,p-1 \}$ with cardinality $p=N\pi$. Begin by letting $\pi=a+b\omega$ be a prime. By the definition of the norm, $p=a^{2}-ab+b^{2}$, and it is obvious that $p\nmid b$. Let there exist $\mu=m+n\omega\in D$. Then there must exist some $c\in\mathbb{Z}$ such that $cb\equiv n \pmod{p}$. Then we have $\mu-c\pi\equiv m+n\omega-c(a+b\omega) \equiv m+n\omega-ac-bc\omega \equiv m-ac+(n-bc)\omega \equiv m-ac \pmod{p}$. Therefore it is also true that $\mu\equiv m-ca \pmod{p}$ by taking both sides modulo $\pi$. This shows that every element of $D$ is congruent to a rational integer modulo $\pi$. Now we need to show modulo $\pi$, these elements correspond to one of $\{ 0,1,2,\ldots,p-1 \}$. Let some $l\in \mathbb{Z}$. Then by the division algorithm we can write $l=sp+r$ for $s,r\in \mathbb{Z}$ and $0\leq r<p$. Therefore $l\equiv r\pmod{p}$, and in fact $l\equiv r\pmod{\pi}$. Since we showed earlier that every element of $D$ is congruent to a rational integer modulo $\pi$, this argument demonstrates that all elements of $D$ are congruent to exactly one of $\{ 0,1,2,\ldots,p-1 \}$ modulo $p$ instead. Now we need only prove that these coset representatives are unique. Suppose that for $r,r'\in\mathbb{Z}$ and $0\leq r,r'<p$, there is a congruence $r\equiv r'\pmod{\pi}$. Then $r-r'=\pi\gamma$ for some arbitrary $\gamma\in D$. Then taking the norm, we have $(r-r')^{2}=N\pi\gamma=pN\gamma$. Then $p|(r-r')^{2}$, and so $p|(r-r')$. This implies that $r\equiv r'\pmod{p}$. Since we initially stated that $r$ and $r'$ are least residues modulo $p$, they are not only equivalent modulo $p$, but further satisfy $r=r'$, thus proving uniqueness. 

    \indent The final case is when an element of $D$ has a norm of 3. Proposition \ref{some prime pi in D such that its norm is p, then pi is prime} guarantees that $1-\omega$ is prime in $D$ because $3$ is prime in $\ZZ$. In other words, since part (3) of Theorem \ref{classification of primes in Z[omega]} asserts that $1-\omega$ is prime in $D$, the residue class ring $D/(1-\omega)D$ contains exactly $N(1-\omega)=3$ elements. Since $\pi=1-\omega$, we will be proving this modulo $(1-\omega)$. To see what these cosets look like, we must determine what the elements of $D/(1-\omega)D$ are. Notice that since this is a residue class ring, we are taking the elements from the ideal $(1-\omega)D$ and combining them with the elements of $D$ to form the set of coset representatives $\{ r+(1-\omega)D | r\in D \}$. The three coset representatives are $0,1$ and $2$, so the respective cosets are $0+D/(1-\omega)D,1+D/(1-\omega)D,$ and $2+D/(1-\omega)D$.
\end{proof}
\noindent The significance of the above result is that it allows us to consider elements of the ring $D/\pi D$ for $\pi$ prime in such a way to be a finite field, which is a fundamental fact when studying cubic reciprocity considering many of the results that we derived in section 1. 

\subsection{Statement of Cubic Reciprocity}
\noindent Now that we have formed the connection between finite fields and the residue class ring $D/\pi D$ for $\pi$ prime, we can begin familiarizing ourselves with characteristics of the finite field. Since $D/\pi D$ is a finite field with order $N\pi$, its multiplicative group $(D/\pi D)^{*}$ has order $N\pi-1$. Theorem \ref{multiplicative group of finite field is cyclic} asserts that $(D/\pi D)^{*}$ must be cyclic, so with it we have a useful analogue in $D/\pi D$ to Fermat's Little Theorem, namely
\begin{thm}[Analogue to Fermat's Little Theorem in $D/\pi D$] \label{Analogue to Fermat's Little Theorem in D/pi D}
    If $\pi$ is a prime and $\pi \nmid \alpha$, then
    \begin{equation*}
        \alpha^{N\pi-1}\equiv 1 \pmod{\pi}.
    \end{equation*}
\end{thm}
\noindent The proof of this follows similarly to Corollary \ref{Fermat's Little Theorem}, where instead we consider the residue classes modulo $\pi$ and use the fact that $(D/\pi D)^{*}$ is cyclic. 

\indent In order to consider higher reciprocity, we must consider cases where $N\pi\neq 3$, namely, when $N\pi>3$. Specifically, if $N\pi\neq 3$, then the residue classes formed by $1,\omega$ and $\omega^{2}$ would necessarily be distinct, unlike if $N\pi=3$. We show this as follows. 

\indent Suppose that $N\pi\neq 3$ in all following cases. Then suppose that $\omega$ and $1$ belong to the same residue class, i.e. $\omega\equiv 1\pmod{\pi}$. Then $\pi | 1-\omega$. However, we showed before that $1-\omega$ is a prime element of $D/\pi D$, so it must be true that $\pi$ is associate to $1-\omega$, i.e. there exists some unit $u$ such that $\pi=u(1-\omega)$. Taking the norm of both sides, we have $N\pi=NuN(1-\omega)=N(1-\omega)=3$, but we assumed that $N\pi\neq 3$, so this is a contradiction. Therefore $\omega$ and $1$ belong to distinct residue classes modulo $\pi$. We can repeat this for $1$ and $\omega^{2}$, and finally for $\omega$ and $\omega^{2}$. For the first one, we assume that $\omega^{2}\equiv 1\pmod{\pi}$. Then $\pi | 1-\omega^{2}$, so for some unit $u$, we have $\pi = u(1-\omega^{2})$. Notice that $-\omega^{2}=-\omega$, so $\pi = u(1-\omega)$. Taking the norm of both sides, we have $N\pi=NuN(1-\omega)=N(1-\omega)=3$, but this is a contradiction, so $\omega^{2}\not\equiv 1 \pmod{\pi}$. The same procedure can be done to show the distinctness of the remaining two residue classes. 

\indent Taking these three distinct residue class representatives, we have a cyclic group of order $3$, namely $\{ 1,\omega,\omega^{2} \}$. By Theorem \ref{Lagrange's Theorem}, this cyclic group divides the order of the group for which it is a subgroup, namely $|(D/\pi D)^{*}|=N\pi -1$. Therefore $3|N\pi -1$. Alternatively, if $\pi=q$ is some rational prime, then taking the norm of both sides, we have $N\pi =q^{2}$. Then $q^{2}\equiv 1 \equiv N\pi \pmod{3}$. Proposition \ref{some prime pi in D such that its norm is p, then pi is prime} asserts the existence of some other prime $p$ such that $N\pi = p$, so $q^{2}\equiv p \equiv 1\pmod{3}$, which is the same as $3|p-1$. This leads us to the following result.

\begin{prop} \label{Proposition: There exists a unique integer m=0,1,2 analogue to FLT}
    Let $\pi$ be a prime with $N\pi\neq 3$ and $\pi \nmid \alpha$. Then there must exist a unique integer $m=0,1,2$ such that 
    \begin{equation*}
        \alpha^{\frac{N\pi-1}{3}}\equiv \omega^{m} \pmod{\pi}.
    \end{equation*}
\end{prop}
\begin{proof}
    Rearranging the statement of Theorem \ref{Analogue to Fermat's Little Theorem in D/pi D}, we have $\alpha^{N\pi-1}-1\equiv 0\pmod{\pi}$, so $\pi|\alpha^{N\pi-1}-1$. Factoring the \textit{LHS}, we have 
    \begin{equation*}
        \alpha^{N\pi-1}-1 = (\alpha^{\frac{N\pi-1}{3}}-1)(\alpha^{\frac{N\pi-1}{3}}-\omega)(\alpha^{\frac{N\pi-1}{3}}-\omega^{2})
    \end{equation*}
    \noindent so $\pi|(\alpha^{\frac{N\pi-1}{3}}-1)(\alpha^{\frac{N\pi-1}{3}}-\omega)(\alpha^{\frac{N\pi-1}{3}}-\omega^{2})$. As we discussed prior to this proof, it must be true that $3|N\pi-1$ for each element of the cyclic group $\{ 1,\omega,\omega^{2} \}$. If $\pi$ divided more than one factor, then the first and intermediate terms would no longer satisfy the property that $3|N\pi -1$. Therefore, $\pi$ must divide exactly one of the factors. Therefore, considering each factor, we have $\pi|(\alpha^{\frac{N\pi-1}{3}}-1)$, $\pi|(\alpha^{\frac{N\pi-1}{3}}-\omega)$, and $\pi|(\alpha^{\frac{N\pi-1}{3}}-\omega^{2})$. Naturally, it follows that for distinct integer values $m$ running from $0$ to $2$, we have $\alpha^{\frac{N\pi-1}{3}}\equiv \omega^{m} \pmod{\pi}$. 
\end{proof}
\noindent We now proceed to define the cubic residue character. Note the following.
\begin{rem}
While the vertical notation $(a/p)$ is preferable for the Legendre symbol and there does exist a vertical cubic residue symbol $(\alpha/\pi)_{3}$, we instead use a shorthand as we will be writing it many times. Therefore, we denote the cubic residue character with $\chi_{\pi}(\alpha)$ to represent the cubic character of $\alpha$ modulo $\pi$. 
\end{rem}
\begin{defn}[Cubic residue character]
    Let $N\pi\neq 3$. We say that the cubic residue character of $\alpha$ modulo $\pi$ is defined as
    \begin{enumerate}
        \item $\chi_{\pi}(\alpha)=0$ if $\pi|\alpha$, 
        \item $\alpha^{\frac{N\pi-1}{3}}\equiv \chi_{\pi}(\alpha) \pmod{\pi}$ where 
        \begin{equation*}
            \chi_{\pi}(\alpha) = 
            \left\{
        \begin{array}{ll}
            1 & \text{if $\alpha$ is a cubic residue},\\
            \omega \ \text{or} \ \omega^{2} & \text{otherwise}.
        \end{array}
    \right.
        \end{equation*}
    \end{enumerate}
\end{defn}
\noindent Recall that the Legendre symbol outputted solutions to the equation $x^{2}-1=0$, so naturally, the cubic residue character outputs solutions to the equation $x^{3}-1=0$, roots of which are cube roots of unity. Note that both symbols can also output $0$. We first prove an important property that the congruence of two cubic characters modulo $\pi$ implies their equality, and follow by proving some other properties of the cubic residue character. 
\begin{lemma} \label{when two numbers are equivalent when they are congruent mod pi lemma}
    Let $\pi\in D$ be a prime. Then suppose that there exist some $a,b\in\{ 0,1,\omega,\omega^{2} \}$ with the property that $a\equiv b \pmod{\pi}$. Then $a=b$. 
\end{lemma}
\begin{proof}
    Notice that if $a\equiv b\pmod{\pi}$, then by the definition of congruence we can write $a-b=\pi\gamma$, where $\gamma\in D$. For $a=b$, the \textit{RHS} must reduce to $0$ modulo $\pi$. 

    \indent We begin by taking the differences of each possible pair in the set $\{ 0,1,\omega,\omega^{2} \}$. We obtain $-1,1,-\omega,\omega,-\omega^{2},\omega^{2},1-\omega,\omega-1,1-\omega^{2},\omega^{2}-1,\omega-\omega^{2},\omega^{2}-\omega$. If we apply the identity $1+\omega+\omega^{2}=0$ we can reduce these differences to $-1,1,-\omega,\omega,-1-\omega,1+\omega,1-\omega,\omega-1,2+\omega,-2-\omega,2\omega+1,-2\omega-1$. We want to show that these differences are always a multiple of some prime $\pi$. 

    \indent To do this, it is sufficient to show that each of these differences is either prime or irreducible. If an element of $D$ is a unit, then it is irreducible. Proposition \ref{classification of units in Z[omega]} asserts that the units of $D$ are $1,-1,\omega,-\omega,\omega^{2},-\omega^{2}$, so when congruent modulo $\pi$ these elements will also be equivalent. Part (3) of Theorem \ref{classification of primes in Z[omega]} asserts that $1-\omega$ is prime in $D$, so by the same argument, $-1-\omega$ is also prime in $D$. Furthermore, by applying the identity $1+\omega+\omega^{2}=0$, notice that $-1-\omega=\omega^{2}$ and $1+\omega=-\omega^{2}$, which are both units. Finally, applying Proposition \ref{some prime pi in D such that its norm is p, then pi is prime}, we need only check that the norm of the 4 differences is prime in $\ZZ$. We have $N(2+\omega)=N(-2-\omega)=2^{2}-2(1)+1=3$, which is prime in $\ZZ$. Similarly, $N(1+2\omega)=N(-1-2\omega)=1^{2}-2(1)+2^{2}=3$, which is also prime in $\ZZ$. 

    \indent Therefore if each element is congruent to every other element modulo $\pi$, they must be equal. 
\end{proof}

\begin{prop} \label{properties of the cubic residue character}
    The following are true.
    \begin{enumerate}
        \item $\chi_{\pi}(\alpha)=1$ if and only if the congruence $x^{3}\equiv \alpha \pmod{\pi}$ is solvable, namely, if $\alpha$ is a cubic residue modulo $\pi$, or $\pi\nmid\alpha$.
        \item \begin{equation*}\alpha^{\frac{N\pi-1}{3}}\equiv \chi_{\pi}(\alpha)\pmod{\pi}. \end{equation*}
        \item \begin{equation*}
            \chi_{\pi}(\alpha\beta) = \chi_{\pi}(\alpha)\chi_{\pi}(\beta).
        \end{equation*}
        \item If $\alpha\equiv\beta\pmod{\pi}$, then \begin{equation*}\chi_{\pi}(\alpha)=\chi_{\pi}(\beta). \end{equation*}
    \end{enumerate}
\end{prop}
\begin{proof}
    \noindent To prove (1), consider Theorem \ref{last result about F*}. Let $\mathbb{F}=D/\pi D$, so that $\mathbb{F}^{*}=(D/\pi D)^{*}$. Furthermore, let $q=N\pi$ and $n=3$ as we are dealing with a cubic. Therefore the congruence $x^{3}\equiv \alpha \pmod{\pi}$ is solvable if and only if $\alpha^{(N\pi-1)/\gcd(3,N\pi -1)}\equiv 1\pmod{\pi}$. This is solvable because $N\pi -1$ is always divisible by the possible values of $\gcd(3,N\pi -1)=1$ or $3$, and $1$ is always a cubic residue modulo $\pi$. Furthermore, there is either one solution or three solutions depending on the value of $\gcd(3,N\pi-1)$. This is the same idea as what is described in Remark \ref{Remark: nth power residues and how many solutions to x^n=alpha} but applied to $D/\pi D$ instead. 

    \indent (2) follows from the definition of the cubic residue character. 

    \indent For (3) and (4) we use Lemma \ref{when two numbers are equivalent when they are congruent mod pi lemma}. To show (3), we have 
    \begin{equation*}
        \chi_{\pi}(\alpha\beta) \equiv (\alpha\beta)^{\frac{N\pi-1}{3}} \equiv \alpha^{\frac{N\pi-1}{3}}\beta^{\frac{N\pi-1}{3}}\equiv \chi_{\pi}(\alpha)\chi_{\pi}(\beta) \pmod{\pi}.
    \end{equation*}
    \noindent By Lemma \ref{when two numbers are equivalent when they are congruent mod pi lemma}, we have $ \chi_{\pi}(\alpha\beta)=\chi_{\pi}(\alpha)\chi_{\pi}(\beta)$.

    \indent To show (4), notice that if $\alpha\equiv\beta\pmod{\pi}$, then
    \begin{equation*}
        \chi_{\pi}(\alpha) \equiv \alpha^{\frac{N\pi-1}{3}} \equiv \beta^{\frac{N\pi-1}{3}} \equiv \chi_{\pi}(\beta)\pmod{\pi}.
    \end{equation*}
    \noindent By Lemma \ref{when two numbers are equivalent when they are congruent mod pi lemma}, we then have $\chi_{\pi}(\alpha)=\chi_{\pi}(\beta)$.
\end{proof}
\noindent We will now study cubic characters and their function, as the proof of cubic reciprocity requires the use of cubic Gauss sums. 
\begin{prop} \label{properties of the cubic character modulo pi}
    The following are true. 
    \begin{enumerate}
        \item $\overline{\chi_{\pi}(\alpha)}=\chi_{\pi}(\alpha)^{2}=\chi_{\pi}(\alpha^{2})$.
        \item $\overline{\chi_{\pi}(\alpha)}=\chi_{\overline{\pi}}(\overline{\alpha})$.
    \end{enumerate}
\end{prop}
\begin{proof}
    \noindent By definition, $\chi_{\pi}(\alpha)\in\{ 0, 1,\omega,\omega^{2} \}$. Squaring each one, we have $0,1,\omega^{2}=\omega$, and $\omega^{4}=\omega^{2}$, all of which are equivalent to their conjugate.

    \indent To prove the second property, recall by Proposition \ref{properties of the cubic residue character} that $\alpha^{(N\pi-1)/3}\equiv \chi_{\pi}(\alpha)\pmod{\pi}$. Conjugating both sides, we have
    \begin{equation*}
        \overline{\alpha}^{(N\overline{\pi}-1)/3}\equiv \overline{\chi_{\pi}(\alpha)}\pmod{\overline{\pi}}.
    \end{equation*}
    \noindent Notice that $\overline{\alpha}^{(N\overline{\pi}-1)/3} \equiv \chi_{\overline{\pi}}(\overline{\alpha}) \pmod{\overline{\pi}}$, but $N\overline{\pi}=N\overline{\overline{\pi}}=N\pi$, so this is just $\chi_{\overline{\pi}}(\overline{\alpha}) \equiv \overline{\chi_{\pi}(\alpha)} \pmod{\overline{\pi}}$. By Lemma \ref{when two numbers are equivalent when they are congruent mod pi lemma}, we thus have $\chi_{\overline{\pi}}(\overline{\alpha}) = \overline{\chi_{\pi}(\alpha)}$. 
\end{proof}
\noindent From this we have the following corollary. 
\begin{cor} \label{Corollary: cor to props of cub char mod pi}
\noindent The following are true for a rational integer $q$.
    \begin{enumerate}
        \item $\chi_{q}(\overline{\alpha})=\chi_{q}(\alpha^{2})$.
        \item If $n$ is a rational integer coprime to $q$, then $\chi_{q}(n)=1$. 
    \end{enumerate}
\end{cor}
\begin{proof}
    Since $q$ is also a rational integer, it is obvious that $\overline{q}=q$. Therefore, by Proposition \ref{properties of the cubic character modulo pi}, we have $\chi_{q}(\overline{\alpha}) = \chi_{\overline{q}}(\overline{\alpha}) = \overline{\chi_{q}(\alpha)} = \chi_{q}(\alpha)^{2} = \chi_{q}(\alpha^{2})$. 

    \indent To prove (2), notice similarly that $\overline{n}=n$. Therefore following the same procedure as (1) we have $\chi_{q}(n) = \chi_{\overline{q}}(\overline{n}) = \overline{\chi_{q}(n)} = \chi_{q}(n)^{2}$. It is impossible for $\chi_{q}(n)$ to be $0$ as $n\nmid q$, so it must be that $\chi_{q}(n)=1$. In this way, this corollary asserts that $n$ is a cubic residue modulo $q$ if both are rational integers. 
\end{proof}
\begin{rem} \label{Remark: trivial case of cubic reciprocity}
As a special case of the law of cubic reciprocity, consider two primes $q_{1}\neq q_{2}$ such that $q_{1}\equiv q_{2}\equiv 2 \pmod{3}$. Then $\chi_{q_{1}}(q_{2}) = \chi_{q_{2}}(q_{1})$. This is a special case of cubic reciprocity where both the modulus and argument are rational integers. 
\end{rem}
\noindent In order to state the general case, we need to extend this result to all prime elements in $D$. 
\begin{defn}[Primary]
    Let $\pi\in D$ be prime. Then $\pi$ is \textit{primary} if $\pi\equiv 2\pmod{3}$. 
\end{defn}
\noindent We know that $\pi$ is either rational or not rational. In the rational case, the previous discussion in Remark \ref{Remark: trivial case of cubic reciprocity} applies. If $\pi$ is not rational, then we consider when $\pi=a+b\omega$, namely when $a\equiv 2\pmod{3}$ and $b\equiv 0\pmod{3}$. Naturally, there exist 6 possible associates for every element $\pi$ of $D$, as the 6 units of $D$ act as multiplicative identities. However, it is necessary to derive a result that removes the ambiguity concerning which associate can be used for each element of $D$. 
\begin{lemma} \label{lemma for if Npi=1(mod 3) then one associate of pi is primary}
    Let $N\pi=p\equiv 1\pmod{3}$. Exactly one associate of $\pi$ is primary.
\end{lemma}
\begin{proof}
    \noindent Express $\pi=a+b\omega$ for $a,b\in\mathbb{Z}$. Then the associates of $\pi$ are given as (1) $\pi$, since there is some unit $u$ such that $\pi=u\pi$, (2) $\omega\pi$, since $\omega$ is a unit and $\pi=u(\omega\pi)$ for some unit $u$, (3) $\omega^{2}\pi$, since $\omega^{2}$ is also a unit and $\pi=u(\omega^{2}\pi)$ for some unit $u$, (4) $-\pi$, for the same reason as (1), (5) $-\omega\pi$ for the same reason as (2), and (6) $-\omega^{2}\pi$ for the same reason as (3). 

    \indent With all of these associates for $\pi$, we can now express each in terms of $a$ and $b$. Namely, in order, we have
    \begin{enumerate}
        \item $a+b\omega$,
        \item $\omega(a+b\omega)=a\omega+b\omega^{2} = a\omega + b(-1-\omega) = -b + (a-b)\omega$,
        \item $\omega^{2}(a+b\omega) = a\omega^{2}+b\omega^{3} = a(-1-\omega)+b = (b-a)-a\omega$, 
        \item $-a-b\omega$,
        \item $-\omega(a+b\omega) = -(-b+(a-b)\omega) = b+(b-a)\omega$,
        \item $-\omega^{2}(a+b\omega) = -((b-a)-a\omega) = (a-b)+a\omega$.
    \end{enumerate}
    \noindent Among these we must determine the primary associate. Recall that $N\pi=p=a^{2}-ab+b^{2}$. In this expression, only one of $a$ and $b$ is divisible by $3$ since $\pi$ is primary. Therefore, we proceed by assuming, WLOG, that $a\not\equiv 0\pmod{3}$. Then we further assume that $a\equiv 2\pmod{3}$. With these assumptions, we have that $p=a^{2}-ab+b^{2}\implies p\equiv 1\equiv 2^{2}-2b+b^{2} \pmod{3} \implies -3\equiv b(b-2)\equiv 0 \pmod{3}$. We now consider two cases for $b$. If $3|b$, then we have $a\equiv 2\pmod{3}$ and $b\equiv 0\pmod{3}$, such that $\pi\equiv 2\pmod{3}$ and so $a+b\omega$ is primary. If $b\equiv 2 \pmod{3}$, then we must also have $a\equiv 2\pmod{3}$, so $\pi\equiv 2+2\omega \equiv b+(b-a)\omega \pmod{3}$. Therefore $b+(b-a)\omega$ is primary. 

    \indent Now all that remains is to prove uniqueness. Let $a+b\omega$ be primary. This only occurs when $b\equiv 0\pmod{3}$, so looking at (2), it is clear that $-b+(a-b)\omega$ cannot be primary. The same is true for (3). For (4), notice that $b\equiv 0\pmod{3}$ implies that $\pi\equiv -a\equiv -2\equiv 1 \pmod{3}$, so $-a-b\omega$ is not primary. For (5), the argument is identical to (2). For (6), the argument is identical to (3). 
\end{proof}
\noindent Now we are equipped to state cubic reciprocity.
\begin{thm}[The Law of Cubic Reciprocity] \label{The Law of Cubic Reciprocity}
    Let $\pi_{1}$ and $\pi_{2}$ be primary. Furthermore, let $N\pi_{1},N\pi_{2}\neq 3$ with $N\pi_{1}\neq N\pi_{2}$. Then 
    \begin{equation*}
        \chi_{\pi_{1}}(\pi_{2}) = \chi_{\pi_{2}}(\pi_{1}).
    \end{equation*}
    \noindent In words, if $\pi_{1}$ and $\pi_{2}$ are primary with different norms not equal to $3$, then $\pi_{1}$ is a cubic residue modulo $\pi_{2}$ if $\pi_{2}$ is a cubic residue modulo $\pi_{1}$, and $\pi_{1}$ is a cubic nonresidue modulo $\pi_{2}$ if $\pi_{2}$ is a cubic nonresidue modulo $\pi_{1}$. 
\end{thm}
    \noindent This theorem requires that we consider 3 scenarios. The first scenario requires that we consider whether each $\pi$ is rational or not. Namely, when both $\pi_{1}$ and $\pi_{2}$ are rational, when one of $\pi_{1}$ and $\pi_{2}$ is rational and the other is complex (i.e. in $D$), and when both $\pi_{1}$ and $\pi_{2}$ are complex. In Remark \ref{Remark: trivial case of cubic reciprocity}, we showed that the first case is trivial. We will also consider special cases of cubic reciprocity, namely when the input of the cubic character is either a unit, which we consider in the first supplement, or a special prime, which we consider in the second supplement. The first supplement is easily provable, but the second is a little more difficult. 

\subsection{Supplements to Cubic Reciprocity}

    \indent We first consider cubic reciprocity when the inputs are the units, namely $1,\omega,$ and $\omega^{2}$ and their negatives. Clearly, $(-1)^{3}=-1$, so $-1$ is always a cubic residue modulo any $\pi$ prime, i.e. $\chi_{\pi}(-1)=1$. By (2) of Proposition \ref{properties of the cubic residue character}, we have that $\chi_{\pi}(\omega) = \omega^{\frac{N\pi-1}{3}}$. Therefore, the cubic character of units can be stated as follows. 

\begin{thm}[First supplement to the Law of Cubic Reciprocity] \label{First supplement to the Law of Cubic Reciprocity}
    Let $\omega$ be a cube root of unity. Then
    \begin{equation*}
        \chi_{\pi}(\omega) = \omega^{\frac{N\pi-1}{3}} = 
        \left\{
        \begin{array}{lll}
            1 & \text{if} \ N\pi\equiv 1 \pmod{9}, \\
            \omega & \text{if} \ N\pi\equiv 4 \pmod{9}, \\
            \omega^{2} & \text{if} \ N\pi\equiv 7 \pmod{9}.
        \end{array}
        \right.
    \end{equation*}
\end{thm}
\begin{proof}
    This is not difficult to show by considering each case. 
\end{proof}
\noindent Notice that in the identity $1+\omega+\omega^{2}=0$, we can write $\omega^{2} = -1-\omega$. Therefore the final case requires us to consider the cubic residue character $\chi_{\pi}(1-\omega)$, where $1-\omega$ is also a prime in $D$ as shown in Theorem \ref{classification of primes in Z[omega]}. The proof is divided into two separate cases, specifically when the modulus is a rational prime $\pi=q$ and when the modulus is a non-rational prime $\pi$. The first case is easily considered, but the second case requires us to consider more about primary elements in $D$. 
\begin{thm}[Second supplement to the Law of Cubic Reciprocity] \label{Second supplement to the Law of Cubic Reciprocity}
    Let $N\pi\neq 3$. If $\pi=q$ is rational, then write $q=3m-1$. If $\pi$ is primary with $\pi\in D$ then write $\pi=a+b\omega$ and take $a=3m-1$. Then
    \begin{equation*}
        \chi_{\pi}(1-\omega) = \omega^{2m}.
    \end{equation*}
\end{thm}
\begin{proof}[Proof of Theorem \ref{Second supplement to the Law of Cubic Reciprocity} when $\pi=q$ is rational]
    This supplement requires us to consider two different cases. The first case is when $\pi=q$ is a rational prime, and the second is when $\pi=a+b\omega$ is non-rational. We consider the case where $\pi=q$ is a rational prime in this proof.

    \indent First notice that $(1-\omega)^{2} = 1-2\omega+\omega^{2} = -1-\omega+1-2\omega = -3\omega$. Therefore by definition of a character, $\chi_{q}((1-\omega)^{2}) = \chi_{q}(-3\omega) = \chi_{q}(-3)\chi_{q}(\omega)$. By Corollary \ref{Corollary: cor to props of cub char mod pi}, since we know that $\pi=q$ is a rational prime, and since $\gcd(-3,q)=1$, we can write $\chi_{q}(\overline{-3}) = \chi_{q}(-3) = 1$. Taking $\pi=q$ to be a rational prime in Theorem \ref{First supplement to the Law of Cubic Reciprocity}, and since $N\pi=Nq = q^{2}$, we can write $\chi_{q}(\omega) = \omega^{(N\pi-1)/3} = \omega^{(q^{2}-1)/3}$. If we square both sides and notice that $\chi_{q}$ is inherently a cube root of unity, then by Proposition \ref{properties of the cubic character modulo pi} we have
    \begin{align*}
        (\chi_{q}(1-\omega)^{2})^{2} = \chi_{q}(1-\omega)^{4} = (\chi_{q}(1-\omega))^{3}\chi_{q}(1-\omega) & = \chi_{q}(1-\omega) \\ & = \omega^{2\cdot \frac{q^{2}-1}{3}}.
    \end{align*}
    \noindent We now evaluate the \textit{RHS}. Let $q=3m-1$ for some $m\in\ZZ$. Then $q^{2}-1 = (3m-1)^{2}-1 = 9m^{2}-6m$. Therefore the exponent of the \textit{RHS} is $2/3(9m^{2}-6m) = 6m^{2}-4m$. Reducing this modulo $3$ because we are dealing with all possible powers of $\omega$ - namely, the cube roots of unity - we have $6m^{2}-4m \equiv -4m \equiv 2m \pmod{3}$. Substituting, we have $\chi_{q}(1-\omega)=\omega^{2m}$, which is our desired result. 
\end{proof}
\noindent The second case, where $\pi$ is a complex prime, requires that we investigate some facts about primary elements in $D$. 
We begin with the following lemma. 
\begin{lemma} \label{Second supplement to cubic reciprocity: Lemma 1}
    Let $\alpha$ and $\beta$ be two primary elements in $D$. Then $-\alpha\beta$ is primary.
\end{lemma}
\begin{proof}
    Let $\alpha = a+b\omega$ and $\beta= c+d\omega$. By definition of primary, we have $\alpha\equiv 2\pmod{3}$ and $\beta\equiv 2\pmod{3}$. For some $s,t\in\mathbb{Z}$, let $\alpha=3s+2$ and $\beta=3t+2$. Taking the product, we have
    \begin{align*}
        -\alpha\beta = -(3s+2)(3t+2) = -(9st+6s+6t+4) & \equiv -4 \pmod{3} \\ & \equiv 2\pmod{3}.
    \end{align*}
    \noindent Therefore $-\alpha\beta$ is primary. 
\end{proof}
\noindent This can in fact be extended to a product of any number primary elements. This gives us a sort of primary factorization for primary elements in $D$, as shown below. 
\begin{cor} \label{Second supplement to cubic reciprocity: Corollary 1}
    Let $\gamma_{1},\gamma_{2},\ldots,\gamma_{n}\in D$ be primary. Then $(-1)^{n-1}\gamma_{1}\gamma_{2}\cdots \gamma_{n}$ is also primary. 
\end{cor}
\begin{proof}
    We prove this with induction. When $n=2$, we have $(-1)^{1}\gamma_{1}\gamma_{2}$, which is primary by Lemma \ref{Second supplement to cubic reciprocity: Lemma 1}. Assume that this holds for some $n=k$. Then $(-1)^{k-1}\gamma_{1}\gamma_{2}\cdots \gamma_{k}$ is primary. We want to show that $(-1)^{k}\gamma_{1}\gamma_{2}\cdots\gamma_{k}\gamma_{k+1}$ is primary. Notice that 
    \begin{equation*}
        (-1)^{k}\gamma_{1}\gamma_{2}\cdots\gamma_{k}\gamma_{k+1} = (-1)(-1)^{k-1}\gamma_{1}\gamma_{2}\cdots\gamma_{k}(\gamma_{k+1}) = -(\gamma_{k+1})((-1)^{k-1}\gamma_{1}\gamma_{2}\cdots\gamma_{k}).
    \end{equation*}
    \noindent We know that $\gamma_{k+1}$ is primary, so by Lemma \ref{Second supplement to cubic reciprocity: Lemma 1}, the product is also primary.
\end{proof}
\noindent This primary factorization allows us to consider the set of primaries as a UFD. In other words, if $\gamma\in D$ is some primary element, then we can write $\gamma=(-1)^{k-1}\gamma_{2}\gamma_{2}\cdots\gamma_{k}$, where, as in $\mathbb{Z}$, the $\gamma_{i}$s need not be distinct primary prime elements.

\begin{proof}[Proof of Theorem \ref{Second supplement to the Law of Cubic Reciprocity} when $\pi$ is a non-rational prime]
    We can now prove the second case of Theorem \ref{Second supplement to the Law of Cubic Reciprocity}. Let $\pi = a+b\omega$ be a primary non-rational prime. This is only possible when $a\equiv 2 \equiv -1 \pmod{3}$ and $b\equiv 0\pmod{3}$. For $m,n\in\mathbb{Z}$, let $a=3m-1$ and $b=3n$. By definition, $a$ is primary. If $a$ is non-prime, then Corollary \ref{Second supplement to cubic reciprocity: Corollary 1} asserts that for some sequence of primary primes $a_{i}$, we can write $a=(-1)^{n-1}a_{1}a_{2}\cdots a_{n}$. WLOG, we assume that $a$ is a primary rational prime because we can choose any such $a_{i}$ to be our $a$. By extension, we can say that $a+b$ is also a primary rational prime. Furthermore, since $a$ and $b$ are nonzero and we assumed $\pi$ to be complex, it is also true that $\gcd(a,a+b)=\gcd(b,a+b)=\gcd(a,a+bw)$. 
    
    \indent Before we proceed, we note the following important identities. 
    \begin{equation*}\label{Equation: Second supp to cub rec: eq 1}\tag{1}
        \frac{Na-1}{3}=\frac{(3m-1)^{2}-1}{3} \equiv 2m\pmod{3}. 
    \end{equation*}
    \noindent Furthermore, notice that 
    \begin{equation*}\label{Equation: Second supp to cub rec: eq 2}\tag{2}
        a+b\omega\equiv b\omega \pmod{a}.
    \end{equation*}
    \noindent We also note that
    \begin{align*}
        a+b\omega & \equiv 0\pmod{\pi} \\
        a-a\omega+a\omega+b\omega & \equiv 0\pmod{\pi} \\
        a-a\omega & \equiv -(a+b)\omega \pmod{\pi} \label{Equation: Second supp to cub rec: eq 3}\tag{3}. 
    \end{align*}
    \noindent Now, recall that we defined $N\pi=p=a^{2}-ab+b^{2}$, so
    \begin{align*}
        p & = a^{2}-ab+b^{2} \\ &
        = (3m-1)^{2}-(3m-1)(3n)+(3n)^{2} \\ 
        \frac{p-1}{3} & = 3m^{2}-2m-3mn+3n^{2}+n \\ &
        \equiv -2m+n\pmod{3} \label{Equation: Second supp to cub rec: eq 4}\tag{4}.
    \end{align*}
    \noindent Finally, we have the following. 
    \begin{align*}
        a+b & \equiv 0\pmod{a+b} \\
        a+b\omega & \equiv b\omega -b\pmod{a+b} \\
        a+b\omega&\equiv -b(1-\omega) \pmod{a+b}\label{Equation: Second supp to cub rec: eq 5}\tag{5}. 
    \end{align*}
    \noindent Using these results, we can compute the following. We use $\pi=a+b\omega$. Since $a$ is a rational primary, by Remark \ref{Remark: trivial case of cubic reciprocity}, we can write

\begin{align*}
    \chi_{a+b\omega}(1-\omega) & = \chi_{a}(b)\chi_{a+b\omega}(1-\omega) \\ & 
    = \chi_{a}(b\omega^{3})\chi_{a+b\omega}(1-\omega) \\ &
    = \chi_{a}(\omega^{2})\chi_{a}(b\omega)\chi_{a+b\omega}(1-\omega).
\end{align*}

\noindent By \eqref{Equation: Second supp to cub rec: eq 2}, we can expand $\chi_{a}(\omega)^{2}\chi_{a}(a+b\omega)\chi_{a+b\omega}(1-\omega) = \omega^{\frac{2(Na-1)}{3}}\chi_{a+b\omega}(a)\chi_{a+b\omega}(1-\omega)$.

\indent Combining \eqref{Equation: Second supp to cub rec: eq 1} and \eqref{Equation: Second supp to cub rec: eq 3}  and then simplifying, we have

\begin{align*}
    \omega^{2m}\chi_{a+b\omega}(a(1-\omega)) & = \omega^{2m}\chi_{a+b\omega}(-(a+b)\omega) \\ & 
    = \omega^{2m}\chi_{a+b\omega}(-1)\chi_{a+b\omega}(\omega)\chi_{a+b\omega}(a+b) \\ &
    = \omega^{2m}(1)\omega^{\frac{N\pi-1}{3}}\chi_{a+b\omega}(a+b).
\end{align*}
\noindent Applying properties of the cubic character, Theorem \ref{The Law of Cubic Reciprocity}, \eqref{Equation: Second supp to cub rec: eq 4}, and \eqref{Equation: Second supp to cub rec: eq 5}, we have

    \begin{equation*}
        \omega^{2m-2m+n}\chi_{a+b}(a+b\omega)  = \omega^{n}\chi_{a+b}(-b(1-\omega)) = \omega^{n}\chi_{a+b}(1-\omega).
    \end{equation*}
    \noindent It is not difficult to verify the following by evaluating each part individually and simplifying.
    \begin{equation*}\label{Equation: Second supp to cub rec: eq 6}\tag{6}
        \frac{2(N(a+b)-1)}{3} \equiv 2(m+n)\pmod{3}. 
    \end{equation*}
    \noindent Leading toward the final result and recalling that $(1-\omega)^{2}=-3\omega$, we have
    \begin{align*}
        \chi_{a+b\omega}(1-\omega) = \omega^{n}\chi_{a+b}(1-\omega) = \omega^{n}\chi_{a+b}(1-\omega)^{4} & = \omega^{n}(\chi_{a+b}(1-\omega)^{2})^{2} \\ &
        = \omega^{n}\chi_{a+b}(-3\omega)^{2} \\ &
        = \omega^{n}(1)^{2}(1)^{2}\chi_{a+b}(\omega)^{2} \\ & 
        = \omega^{n}(\omega^{\frac{N(a+b)-1}{3}})^{2}.
    \end{align*}
    \noindent By \eqref{Equation: Second supp to cub rec: eq 6}, this is equivalent to writing $\omega^{n}\omega^{2(m+n)} = \omega^{2m+3n} = \omega^{2m}\omega^{3n}=\omega^{2m}$, which is what we wanted to prove.
\end{proof}

\noindent In section 3.6, we will also introduce a special supplement of cubic reciprocity, namely the cubic character of 2 modulo $\pi$. We now proceed to prove cubic reciprocity. 

\subsection{Proof of Cubic Reciprocity}
\noindent Before continuing with the proof, we need to make some preliminary statements regarding $D/\pi D$. We let $\pi\in D$ be a prime such that $N\pi=p\equiv 1 \pmod{3}$. We showed earlier that $D/\pi D$ is a finite field with characteristic $p$, so naturally, it contains the field $\mathbb{Z}/p\mathbb{Z}$ as well. Both fields have $p$ elements. Therefore, it is useful to define an isomorphism between $D/\pi D$ and $\mathbb{Z}/p\mathbb{Z}$. Namely, we have a bijection, where we map residue classes from $\mathbb{Z}/p\mathbb{Z}$ to their complex counterparts in $D/\pi D$. In this way, we are mapping the coset of some residue class in $\mathbb{Z}/p\mathbb{Z}$ to some other coset in $D/\pi D$. This isomorphism allows us to extend the cubic character $\chi_{\pi}$ to not only $D/\pi D$, but $\mathbb{Z}/p\mathbb{Z}$ as well. This means that the properties of the cubic character in $D\pi D$ are also valid in $\mathbb{Z}/p\mathbb{Z}$, allowing us to consider cubic Gauss sums $g_{a}(\chi_{\pi})$ as well as cubic Jacobi sums $J(\chi_{\pi},\chi_{\pi})$ on $\mathbb{Z}/p\mathbb{Z}$. This realization is ultimately what allows us to prove cubic reciprocity, and explains why we investigated $D/\pi D$ so thoroughly. 

\indent Moving on, we need to prove some important properties about the Jacobi sum that relate directly to cubic reciprocity. We first have the following. 
\begin{lemma} \label{exercise result for congruence of sum of powers up to k from 1 to p-1}
    \begin{equation*}
    1^{k}+2^{k}+3^{k}+\cdots +(p-1)^{k} = \sum_{l=1}^{p-1}l^{k} \equiv \left\{
        \begin{array}{ll}
            0 & \text{if $p-1\not\equiv 0 \pmod{k}$} \\
            -1 & \text{if $p-1\equiv 0\pmod{k}$}.
        \end{array}
    \right.
    \end{equation*}
\end{lemma}
\begin{proof}
    The proof can follow by considering $[g]$ to be a primitive root of $(\mathbb{Z}/p\mathbb{Z})^{*}$, and noticing that it is identical to considering the complete set of representatives $\{ [0], [1],[2],\ldots,[p-1] \}$, we can evaluate the sum over the entire finite field and evaluate each congruence depending on whether $p-1|k$ or $p- 1\nmid k$. 
\end{proof}
\begin{prop} \label{proposition stating that the cubic jacobi sum of cubic character equals pi}
    Let $\pi$ be primary. Then 
    \begin{equation*}
        J(\chi_{\pi},\chi_{\pi}) = \pi.
    \end{equation*}
\end{prop}
\begin{proof}
    Let there exist some other primary number $\pi'$ such that $J(\chi_{\pi},\chi_{\pi}) = \pi'$. By the definition of $\pi$ we must have $N\pi = \pi\overline{\pi} = p = \pi'\overline{\pi}'$, since the norm of every primary element $\pi$ is $p$. Therefore either $\pi|\pi'$ or $\pi|\overline{\pi}'$. However, since all prime elements are primary, each prime is coprime to every other prime, implying that we must have either $\pi = \pi'$ or $\pi = \overline{\pi}'$. We need to show that the second equation is not possible in order to show that $\pi$ is unique and that there is only one such primary element. 

    \indent We begin by writing out the cubic Jacobi sum. Then for some $x$ that runs over $\mathbb{Z}/p\mathbb{Z}$, we have
    \begin{equation*}
        J(\chi_{\pi},\chi_{\pi}) = \sum_{x}\chi_{\pi}(x)\chi_{\pi}(1-x).
    \end{equation*}
    \noindent By Proposition \ref{properties of the cubic character modulo pi}, we can rewrite each character so that 
    \begin{equation*}
        \sum_{x}\chi_{\pi}(x)\chi_{\pi}(1-x) \equiv \sum_{x}x^{\frac{p-1}{3}}(1-x)^{\frac{p-1}{3}} \pmod{\pi}. 
    \end{equation*}
    \noindent Notice that the degree of this polynomial is $\deg(f(x))=2((p-1)/3)<p-1$. Obviously, this means that $p-1\nmid \deg(f(x))$. Analogously, by Lemma \ref{exercise result for congruence of sum of powers up to k from 1 to p-1}, and since $x$ runs over all elements of $\mathbb{Z}/p\mathbb{Z}$, we can assert that 
    \begin{equation*}
        \sum_{x}x^{\frac{p-1}{3}}(1-x)^{\frac{p-1}{3}} \equiv 0 \pmod{p}. 
    \end{equation*}
    \noindent Equivalence modulo $p$ can be extended to equivalence modulo $\pi$. Therefore we can make the assertion that 
    \begin{equation*}
        \sum_{x}\chi_{\pi}(x)\chi_{\pi}(1-x) \equiv \sum_{x}x^{\frac{p-1}{3}}(1-x)^{\frac{p-1}{3}} \equiv J(\chi_{\pi},\chi_{\pi}) =\pi' \equiv 0 \pmod{\pi}. 
    \end{equation*}
    \noindent By definition of congruence, $\pi|\pi'$, which is impossible unless $\pi=\pi'$ since $\pi$ and $\pi'$ are both primary. 
\end{proof}
\noindent A simple corollary follows by substituting this result into Corollary \ref{corollary for relation between cubic Gauss sum and jacobi sum} as follows. 
\begin{cor} \label{corollary to proposition relating jacobi sums to Gauss sums}
    \begin{equation*}
        g(\chi_{\pi})^{3} = p\pi.
    \end{equation*}
\end{cor}

\begin{proof}
    By Corollary \ref{corollary for relation between cubic Gauss sum and jacobi sum}, we know that $g(\chi)^{3} = pJ(\chi,\chi)$. Take the character to be the cubic character. Then by Proposition \ref{proposition stating that the cubic jacobi sum of cubic character equals pi}, we have $g(\chi_{\pi})^{3} = p\pi$. 
\end{proof}
\noindent We need a final fact about the Jacobi sum. The significance of the following result is that it allows us to make the assertion that the Jacobi sum $J(\chi,\chi)$ is a primary prime in $D$. In fact, since $J(\chi,\chi)$ is indeed primary, we have $J(\chi,\chi)\overline{J(\chi,\chi)}=p$, i.e. $J(\chi,\chi)$ has norm $p$. First, we must utilize a fact about algebraic integers. Let $\Omega$ denote the set of algebraic integers. 

\begin{lemma} \label{Lemma about prime power with algebraic integers}
    Let $\omega_{1},\omega_{2}\in\Omega$ and $p\in\mathbb{Z}$ be prime. Then
    \begin{equation*}
        (\omega_{1}+\omega_{2})^{p}\equiv \omega_{1}^{p}+\omega_{2}^{p} \pmod{p}.
    \end{equation*}
\end{lemma}
\begin{proof}
    The proof follows by expanding the \textit{LHS} with the Binomial Theorem and noticing that all intermediate terms reduce to $0$ modulo $p$.
\end{proof}

\noindent In the following, assume that $\pi\in D$ is primary. 
\begin{prop} \label{Proposition about jacobi sum being a primary prime}
    If $J(\chi,\chi)=a+b\omega$ where $\omega\in\mathbb{Z}[\omega]$, then $a\equiv -1 \pmod{3}$ and $b\equiv 0 \pmod{3}$. 
\end{prop}
\begin{proof}
    It is well known that the algebraic integers form a ring. Working with congruences in $\Omega$ and using Lemma \ref{Lemma about prime power with algebraic integers} as well as the fact that $\chi(t)$ is a cubic character, we have 
    \begin{equation*}
        g(\chi)^{3} = \bigg( \sum_{t\in\mathbb{F}_{p}}\chi(t)\zeta_{3}^{t} \bigg)^{3} \equiv \sum_{t\in\mathbb{F}_{p}}\chi(t)^{3}\zeta_{3}^{3t} \pmod{3}.
    \end{equation*}
    \noindent Clearly $\chi(0)=0$, and since $\chi(t)$ is a cubic character, $\chi(t)^{3}=1$ for nonzero $t$. Therefore 
    \begin{equation*}
        \sum_{t\in\mathbb{F}_{p}}\chi(t)^{3}\zeta_{3}^{3t} = \sum_{t\neq 0}\zeta_{3}^{3t} = \sum_{t\neq 0}e^{2\pi it} = -1.
    \end{equation*}
    \noindent By Corollary \ref{corollary to proposition relating jacobi sums to Gauss sums}, we have
    \begin{equation*}
        g(\chi)^{3} = pJ(\chi,\chi) \equiv a+b\omega \equiv -1 \pmod{3}. 
    \end{equation*}
    \noindent Alternatively, consider the conjugate $\overline{\chi}$. By Proposition \ref{properties of the cubic character modulo pi}, we have that $\overline{g(\chi)} = g(\overline{\chi})$. Again applying Corollary \ref{corollary to proposition relating jacobi sums to Gauss sums}, we have
    \begin{equation*}
        g(\overline{\chi})^{3} = pJ(\overline{\chi},\overline{\chi}) \equiv a+b\overline{\omega} \equiv -1\pmod{3}. 
    \end{equation*}
    \noindent Equating these equations modulo $3$, we have
    \begin{align*}
        (a+b\omega)-(a+b\overline{\omega}) & \equiv -1-(-1) \pmod{3} \\ 
        b(\omega-\overline{\omega}) & \equiv 0 \pmod{3} \\
        b\bigg( \frac{-1+i\sqrt{3}}{2} - \frac{-1-i\sqrt{3}}{2} \bigg) & \equiv 0 \pmod{3} \\
        b\bigg( \frac{2i\sqrt{3}}{2} \bigg) = bi\sqrt{3} & \equiv 0 \pmod{3} \\
        -3b^{2} & \equiv 0\pmod{9}.
    \end{align*}
    \noindent This implies that $3|b$, i.e. $b\equiv 0 \pmod{3}$. Therefore $a+b\omega \equiv a \equiv -1 \equiv 2\pmod{3}$. 
\end{proof}

\noindent With all of these results, we are now sufficiently equipped to prove cubic reciprocity.
\begin{proof}[Proof of The Law of Cubic Reciprocity (Theorem \ref{The Law of Cubic Reciprocity})]
    As indicated before, this is a proof by cases. In order to prove the full law, we need to consider when one of $\pi_{1}$ and $\pi_{2}$ is complex and the other is rational, and when both are complex. 
    
    \indent We first prove the case when one is complex and the other is rational. In other words, we need to show that if $\pi_{1}=q\equiv 2$ and $\pi_{2}=\pi$ a primary with $N\pi=p$, then $\chi_{\pi}(q) = \chi_{q}(\pi)$. Begin by taking the expression in Corollary \ref{corollary to proposition relating jacobi sums to Gauss sums} and raising the \textit{LHS} and \textit{RHS} to a power of $(q^{2}-1)/3$. This results in the equivalence $g(\chi_{\pi})^{((q^{2}-1)/3)\cdot 3} = g(\chi_{\pi})^{q^{2}-1} = (p\pi)^{(q^{2}-1)/3}$. Take this equality modulo $q$, so by properties of the cubic character, we have $g(\chi_{\pi})^{q^{2}-1}\equiv \chi_{q}(p\pi)\pmod{q}$. Notice that the \textit{RHS} of the congruence can be expressed as $\chi_{q}(p\pi) = \chi_{q}(p)\chi_{q}(\pi)$. Clearly, since $p$ and $q$ are coprime, by Corollary \ref{Corollary: cor to props of cub char mod pi}, $\chi_{q}(p)=1$. Therefore we can rewrite the congruence as 
    \begin{align*}
        g(\chi_{\pi})^{q^{2}-1} & \equiv \chi_{q}(p)\chi_{q}(\pi)\pmod{q} \\
        g(\chi_{\pi})^{q^{2}} & \equiv \chi_{q}(\pi)g(\chi_{\pi}) \pmod{q}.
    \end{align*}
    \noindent Now we can examine the \textit{LHS}. If we expand the Gauss sum of the \textit{LHS}, we have
    \begin{equation*}
        g(\chi_{\pi})^{q^{2}} = \bigg( \sum_{t\in\mathbb{F}_{p}}\chi_{\pi}(t)\zeta^{t} \bigg)^{q^{2}}.
    \end{equation*}
    \noindent The intermediate terms of this sum will have some form of $q$ as a factor, so if we take this modulo $q$, we are left with
    \begin{equation*}
        g(\chi_{\pi})^{q^{2}} \equiv \sum_{t\in\mathbb{F}_{p}}\chi_{\pi}(t)^{q^{2}}\zeta^{q^{2}t} \pmod{q}. 
    \end{equation*}
    \noindent Notice that $q\equiv 2\pmod{3}\implies q^{2}\equiv 1\pmod{3}$. Also, since $\chi_{\pi}(t)$ is a cube root of 1, meaning that is is a cube root of unity, we can express the \textit{RHS} as a Gauss sum so that we can simplify to $g(\chi_{\pi})^{q^{2}} \equiv g_{q^{2}}(\chi_{\pi})\pmod{q}$. By Lemma \ref{lemma with some important properties of the Gauss sum and its evaluation} and since $\overline{q^{2}}=q^{-2}$ and $\chi_{\pi}(q^{3})=1$, the \textit{RHS} in fact becomes $g_{q^{2}}(\chi_{\pi}) = \chi_{\pi}(q^{-2})g(\chi_{\pi}) = \chi_{\pi}(q^{-2})\chi_{\pi}(q^{3})g(\chi_{\pi}) = \chi_{\pi}(q)g(\chi_{\pi})$. Thus, setting the two equations equal to each other, we have the expression
    \begin{equation*}
        \chi_{\pi}(q)g(\chi_{\pi}) \equiv \chi_{q}(\pi)g(\chi_{\pi}) \pmod{\pi}.
    \end{equation*}
    \noindent Notice that $g(\chi_{\pi})g(\overline{\chi_{\pi}})=\chi_{\pi}(-1)p=p$ by Corollary \ref{Corollary to Lemma about value of general Gauss sum}, so multiplying both sides by $g(\overline{\chi_{\pi}})$ we have 
    \begin{align*}
        \chi_{\pi}(q)g(\chi_{\pi}) & \equiv \chi_{q}(\pi)g(\chi_{\pi}) \pmod{q} \\
        \chi_{\pi}(q)p & \equiv \chi_{q}(\pi)p \pmod{q} \\
        \chi_{\pi}(q) & \equiv \chi_{q}(\pi) \pmod{q}.
    \end{align*}
    \noindent By Lemma \ref{when two numbers are equivalent when they are congruent mod pi lemma}, this means that $\chi_{\pi}(q) = \chi_{q}(\pi)$. 

    \indent Now we need to show that this is true when both $\pi_{1}$ and $\pi_{2}$ are non-rational. In this case, we have that the norm of both must be congruent to 1 modulo 3, i.e. $N\pi_{1}=p_{1}\equiv 1\pmod{3}$ and $N\pi_{2} = p_{2}\equiv 1 \pmod{3}$. (Note in some way that this is a result of Theorem \ref{classification of primes in Z[omega]}) Let some $\gamma_{1}=\overline{\pi_{1}}$ and some $\gamma_{2}=\overline{\pi_{2}}$. By Lemma \ref{lemma for if Npi=1(mod 3) then one associate of pi is primary}, we know that exactly one associate of each $p_{1}$ or $p_{2}$ is primary, so we call them $\gamma_{1}$ and $\gamma_{2}$. Then $p_{1}=\pi_{1}\gamma_{1}$ and $p_{2}=\pi_{2}\gamma_{2}$. We approach this problem in a similar way to when one prime is rational and the other is complex. Take the expression in Corollary \ref{corollary to proposition relating jacobi sums to Gauss sums}, and write it as $g(\chi_{\gamma_{1}})^{3} = p_{1}\gamma_{1}$. Raising the \textit{LHS} and \textit{RHS} to a power of $(N\pi_{2}-1)/3$ or $(p_{2}-1)/3$ and taking the expression modulo $\pi_{2}$, we obtain 
    \begin{align*}
        (g(\chi_{\gamma_{1}})^{3})^{\frac{p_{2}-1}{3}} = g(\chi_{\gamma_{1}})^{\pi_{2}-1} &
        = (p_{1}\gamma_{1})^{\frac{p_{2}-1}{3}} \equiv \chi_{\pi_{2}}(p_{1}\gamma_{1}) \pmod{\pi_{2}} \\
        g(\chi_{\gamma_{1}})^{p_{2}}g(\chi_{\gamma_{1}})^{-1} & \equiv \chi_{\pi_{2}}(p_{1}\gamma_{1}) \pmod{\pi_{2}} \\ g(\chi_{\gamma_{1}})^{p_{2}} & \equiv g(\chi_{\gamma_{1}})\chi_{\pi_{2}}(p_{1}\gamma_{1}) \pmod{\pi_{2}}.
    \end{align*}
    \noindent We can simplify the \textit{LHS}. Notice that 
    \begin{equation*}
        g(\chi_{\gamma_{1}})^{p_{2}} = \bigg( \sum_{t\in\mathbb{F}_{p}}\chi_{\gamma_{1}}(t)\zeta^{t} \bigg)^{p_{2}} \equiv \sum_{t\in\mathbb{F}_{p}}\chi_{\gamma_{1}}(t)^{p_{2}}\zeta^{p_{2}t} \pmod{\pi_{2}} 
    \end{equation*}
    since the intermediate terms are all congruent to 0 modulo $\pi_{2}$, and disappear after reduction. Notice that this is also a Gauss sum, namely $g_{p_{2}}(\chi_{\gamma_{1}})$, so $g(\chi_{\gamma_{1}})^{p_{2}}\equiv g_{p_{2}}(\chi_{\gamma_{1}})\pmod{\pi_{2}}$. Notice that by Lemma \ref{lemma with some important properties of the Gauss sum and its evaluation} and Corollary \ref{Corollary: cor to props of cub char mod pi}, and since $\overline{p_{2}}=p_{2}$, the \textit{RHS} of this congruence can be equivalently written as $g_{p_{2}}(\chi_{\gamma_{1}}) = \chi_{\gamma_{1}}(p_{2}^{2})g(\chi_{\gamma_{1}})$. If we equate this to the expression derived above, then we have 
    \begin{equation*}
        \chi_{\gamma_{1}}(p_{2}^{2})g(\chi_{\gamma_{1}})  \equiv g(\chi_{\gamma_{1}})\chi_{\pi_{2}}(p_{1}\gamma_{1}) \pmod{\pi_{2}}
    \end{equation*}
    Recall again that $g(\chi_{1})\overline{g(\chi_{\gamma_{1}})}=p$, so multiplying both sides by $\overline{g(\chi_{\gamma_{1}})}$, and by Lemma \ref{when two numbers are equivalent when they are congruent mod pi lemma} , we have
    \begin{align*}
        \chi_{\gamma_{1}}(p_{2}^{2})p & \equiv p\chi_{\pi_{2}}(p_{1}\gamma_{1}) \pmod{\pi_{2}}  \nonumber\\
        \chi_{\gamma_{1}}(p_{2}^{2}) & \equiv \chi_{\pi_{2}}(p_{1}\gamma_{1}) \pmod{\pi_{2}} \nonumber\\ 
        \chi_{\gamma_{1}}(p_{2}^{2}) & = \chi_{\pi_{2}}(p_{1}\gamma_{1}).\label{Equation: eq1 for cr proof}\tag{1}
    \end{align*}
    \noindent We now seek to evaluate the the same thing, but instead using the relation $g(\chi_{\pi_{2}})^{3}=p_{2}\gamma_{2}$. Raising the \textit{LHS} and \textit{RHS} to a power of $(N\pi_{1}-1)/3$ or $(p_{1}-1)/3$ and taking the expression modulo $\pi_{1}$, we obtain
    \begin{align*}
        (g(\chi_{\pi_{2}})^{3})^{\frac{p_{1}-1}{3}} & = (p_{2}\pi_{2})^{\frac{p_{1}-1}{3}} \\ 
        g(\chi_{\pi_{2}})^{p_{1}-1} & \equiv \chi_{\pi_{1}}(p_{2}\pi_{2}) \pmod{\pi_{1}} \\
        g(\chi_{\pi_{2}})^{p_{1}} & \equiv \chi_{\pi_{1}}(p_{2}\pi_{2})g(\chi_{\pi_{2}}) \pmod{\pi_{1}}. 
    \end{align*}
    \noindent Evaluating the \textit{LHS} of this congruence, notice that by using the same facts about Gauss sums as before we have 
    \begin{equation*}
        g(\chi_{\pi_{2}})^{p_{1}} = \bigg( \sum_{t\in\mathbb{F}_{p}}\chi_{\pi_{2}}(t)\zeta^{t} \bigg)^{p_{1}} \equiv \sum_{t\in\mathbb{F}_{p}}\chi_{\pi_{2}}(t)^{p_{1}}\zeta^{p_{1}t} \pmod{\pi_{1}}.
    \end{equation*}
    \noindent This is also a Gauss sum, so we can write $g(\chi_{\pi_{2}})^{p_{1}} = g_{p_{1}}(\chi_{\pi_{2}}) = \chi_{\pi_{2}}(p_{1}^{2})g(\chi_{\pi_{2}})$ since $\overline{p_{1}}=p_{1}$. Equating this to the equation derived above, we now have 
    \begin{align*}
        \chi_{\pi_{2}}(p_{1}^{2})g(\chi_{\pi_{2}}) & \equiv \chi_{\pi_{1}}(p_{2}\pi_{2})g(\chi_{\pi_{2}}) \pmod{\pi_{1}} \\
        \chi_{\pi_{2}}(p_{1}^{2})p & \equiv \chi_{\pi_{1}}(p_{2}\pi_{2})p \pmod{\pi_{1}} \\
        \chi_{\pi_{2}}(p_{1}^{2}) & \equiv \chi_{\pi_{1}}(p_{2}\pi_{2}) \pmod{\pi_{1}}. \label{Equation: eq2 for cr proof}\tag{2}
    \end{align*}
    \noindent We have evaluated the cases for both $\pi_{1}$ and $\pi_{2}$, but now we are interested in relating them. We want to evaluate $\chi_{\gamma_{1}}(p_{2}^{2})$. Notice that by (1) of Corollary \ref{Corollary: cor to props of cub char mod pi} again, we can rewrite this as $\chi_{\gamma_{1}}(p_{2}^{2}) = (\chi_{\gamma_{1}}(p_{2}))^{2} = \overline{\chi_{\gamma_{1}}(p_{2})}$. Since $\overline{\gamma_{1}}=\overline{\overline{\pi_{1}}}=\pi_{1}$, and since $\overline{p_{2}}=p_{2}$, we have $\overline{\chi_{\gamma_{1}}(p_{2})} = \chi_{\pi_{1}}(p_{2})$, so 
    \begin{equation*}
    \chi_{\gamma_{1}}(p_{2}^{2}) = \chi_{\pi_{1}}(p_{2}). \label{Equation: eq3 for cr proof}\tag{3}
    \end{equation*}
    \noindent Now we are able to finish the proof. 

    \indent We compute the following. We have
    \begin{equation*}
        \chi_{\pi_{1}}(\pi_{2})\chi_{\pi_{2}}(p_{1}\gamma_{1}) = \chi_{\pi_{1}}(\pi_{2})\chi_{\gamma_{1}}(p_{2}^{2}), \label{Equation: eq4 for cr proof}\tag{4}
    \end{equation*}
    \noindent which follows by substituting \eqref{Equation: eq1 for cr proof}. Using \eqref{Equation: eq3 for cr proof}, we have that 
    \begin{equation*}
        \chi_{\pi_{1}}(\pi_{2})\chi_{\gamma_{1}}(p_{2}^{2}) = \chi_{\pi_{1}}(\pi_{2}\chi_{\pi_{1}}(p_{2}) = \chi_{\pi_{1}}(p_{2}\pi_{2}).
    \end{equation*}
   \noindent By \eqref{Equation: eq2 for cr proof}, we can write this as
   \begin{align*}
       \chi_{\pi_{2}}(p_{1}^{2}) & = \chi_{\pi_{2}}(p_{1}\pi_{1}\gamma_{1}) \\ 
       & = \chi_{\pi_{2}}(\pi_{1})\chi_{\pi_{2}}(p_{1}\gamma_{1}). \label{Equation: eq5 for cr proof}\tag{5}
   \end{align*}
   \noindent Equating the \textit{LHS} of \eqref{Equation: eq4 for cr proof} to \eqref{Equation: eq5 for cr proof} and dividing both sides by $\chi_{\pi_{2}}(p_{1}\gamma_{1})$, we have
   \begin{align*}
       \chi_{\pi_{1}}(\pi_{2})\chi_{\pi_{2}}(p_{1}\gamma_{1}) & = \chi_{\pi_{2}}(\pi_{1})\chi_{\pi_{2}}(p_{1}\gamma_{1}) \\
       \chi_{\pi_{1}}(\pi_{2}) & = \chi_{\pi_{2}}(\pi_{1}),
   \end{align*}
   \noindent which is the statement of cubic reciprocity.
\end{proof}

\subsection{The Cubic Character of $2$}
\noindent Now that we have proven cubic reciprocity, we might be interested in investigating what special values might be cubic residues. The special case that we will consider in particular is the even prime $2$. We will not prove the final result as it uses facts about Jacobi sums that lie beyond the scope of this paper, but a detailed proof may be found in Chapter 10 of \cite{Rousseau2012reciprocity}. To begin, we have the following result about special rational primes. 

\begin{prop}\label{Proposition about how every int is a cub res if q=2 mod 3}
    If $q\equiv 2\pmod{3}$ is a rational prime, then every integer is a cubic residue modulo $q$. 
\end{prop}
\begin{proof}
    We begin by assuming that $q\equiv 2\pmod{3}$ is a rational prime. Since $q$ is rational, we can work in the integers modulo $q$, namely $\mathbb{Z}/q\mathbb{Z}$. We define a group homomorphism $\phi:(\mathbb{Z}/q\mathbb{Z})^{*} \rightarrow (\mathbb{Z}/q\mathbb{Z})^{*}$ with the mapping $\phi(k)=k^{3}$ for some $k\in (\mathbb{Z}/q\mathbb{Z})^{*}$. By Theorem \ref{First Isomorphism Theorem}, 
    \begin{equation*}
        (\mathbb{Z}/q\mathbb{Z})^{*}/\text{Ker}(\phi)\approx \text{Im}(\phi).
    \end{equation*}
    \noindent We determine the kernel of $\phi$. Clearly, it is only possible for some $k\in \text{Ker}(\phi)$ if $k^{3}=1$, i.e. $\phi$ maps $k$ to the identity of $(\mathbb{Z}/q\mathbb{Z})^{*}$, which is $1$. However, Theorem \ref{multiplicative group of the integers modulo p is cyclic} asserts that the multiplicative group $(\mathbb{Z}/q\mathbb{Z})^{*}$ is cyclic with order $q-1$. But, $3\nmid q-1$, so naturally the relation $k^{3}=1$ is possible if and only if $k=1$, as it maps to the identity. Thus $\text{Ker}(\phi)$ is trivial, in that it only contains one element, so that 
    \begin{equation*}
        |\text{Im}(\phi)|=|(\mathbb{Z}/q\mathbb{Z})^{*}/\text{Ker}(\phi)|=|(\mathbb{Z}/q\mathbb{Z})^{*}|/1 = |(\mathbb{Z}/q\mathbb{Z})^{*}|.
    \end{equation*}
    \noindent This satisfies the condition for $\phi$ to be surjective, so due to the mapping $\phi$ defined earlier, every element of $(\mathbb{Z}/q\mathbb{Z})^{*}$ is a perfect cube, i.e. every integer is a cubic residue modulo $q$.
\end{proof}

\noindent This takes care of one case. This result directly implies that $2$ is always a cubic residue modulo a rational prime $q\equiv 2\pmod{3}$. The following result gives special conditions for the cubic character of $2$ given a prime modulus. 

\begin{prop} \label{Proposition about pi=1 mod 2 for x^3=2 mod pi to be solvable}
    The cubic congruence $x^{3}\equiv 2\pmod{\pi}$ where $\pi\in D$ is prime is solvable if and only if $\pi\equiv 1\pmod{2}$, i.e. if and only if $a\equiv 1\pmod{2}$ and $b\equiv 0\pmod{2}$ in $\pi=a+b\omega$. 
\end{prop}
\begin{proof}
    If $\pi=q$ is some primary rational prime, then Proposition \ref{Proposition about how every int is a cub res if q=2 mod 3} asserts that every integer is a perfect cube, so every integer is a cubic residue modulo $q$ a primary rational prime. Therefore we need to prove that this result is true for a complex primary prime $\pi$. 

    \indent Let $\pi=a+b\omega$ be a primary prime. By Theorem \ref{The Law of Cubic Reciprocity}, we have that $\chi_{\pi}(2) = \chi_{2}(\pi)$. We evaluate the \textit{RHS} of this equality. We have
    \begin{equation*}
        \pi^{\frac{N(2)-1}{3}} = \pi^{(4-1)/3} = \pi \equiv \chi_{2}(\pi)\pmod{2}. 
    \end{equation*}
    \noindent By Proposition \ref{properties of the cubic residue character}, $\chi_{2}(\pi)=1$, i.e. $\pi$ is a cubic residue modulo $2$, if and only if the congruence $x^{3}\equiv \pi\pmod{2}$ is solvable. This congruence is solvable, however, if and only if $\pi\equiv 1\pmod{2}$, because $\pi$ would no longer be prime if $\pi\equiv 0\pmod{2}$. Therefore $\chi_{2}(\pi)=1$ if and only if $\pi\equiv 1\pmod{2}$. Similarly, we have that $\chi_{\pi}(2)=1$ if and only if $\pi\equiv 1\pmod{2}$. In either case, we thus have that the congruence $x^{3}\equiv 2\pmod{\pi}$ is solvable if and only if $\pi\equiv 1\pmod{2}$. 
\end{proof}
\noindent With this result, we can prove a condition for the solvability of $x^{3}\equiv 2\pmod{p}$, which allows us to characterize the cubic character of $2$. Proposition \ref{Proposition about pi=1 mod 2 for x^3=2 mod pi to be solvable} asserts that the modulus must be a rational primary prime such that $a\equiv 1\pmod{2}$ and $b\equiv 0\pmod{2}$. We use this fact concerning the modulus in the proof of the following due to Gauss. 
\begin{thm}
    Let $p\equiv 1\pmod{3}$. Then the congruence $x^{3}\equiv 2\pmod{p}$ is solvable if and only if there exist $C,D\in\mathbb{Z}$ such that $p=C^{2}+27D^{2}$. 
\end{thm}
\begin{proof}
    The proof may be found in Chapter 10 of \cite{Rousseau2012reciprocity}. 
\end{proof}

\section{A Brief Survey of Biquadratic Reciprocity}
\noindent While Carl Friedrich Gauss had published 8 proofs of quadratic reciprocity by his death, he stated, without proof, cubic and biquadratic reciprocity. Though Gauss did not provide proofs, he stated that their proofs likely involved Gauss sums, a new technique for proving higher reciprocity laws. The proof of cubic reciprocity given in Section 3 is exactly using cubic Gauss sums. In this rather brief section we aim to introduce the fundamentals of the Gaussian integers $\mathbb{Z}[i]$ and outline the biquadratic reciprocity law. 
\subsection{The Statement of Biquadratic Reciprocity}
\noindent As explained in the introduction, there are many connections between cubic and biquadratic reciprocity in that they both utilize finite fields. Throughout the rest of this section we let $D=\ZZ[i]$ denote the Gaussian integers. If some $\alpha\in D$ then $(\alpha)=\alpha D$ is the principal ideal generated by $\alpha$. This is useful in defining the residue class ring later. It is well known that $D$ is a Euclidean domain, i.e. there exists a Euclidean algorithm over $D$. Therefore, an analogue of Euclid's lemma applies such that if $\pi\in D$ is some irreducible element and $\alpha,\beta\in D$, then $\pi|\alpha\beta$ implies that $\pi|\alpha$ or $\pi|\beta$. 

\indent There is also a norm function over $D$ so that $N\alpha=\alpha\overline{\alpha}$. Some $\alpha\in D$ is a unit if and only if $N\alpha=1$. Suppose that $\alpha=a+bi$ is a unit. Then $\alpha|1$, so for some $\beta\in D$ we have $\alpha\beta=1$. Taking norms, we have $N\alpha\beta=N(1)=1$. Now suppose that $N\alpha=1$. Then by the definition of the norm $a^{2}+b^{2}=1$. This is possible only if either $a$ or $b$ is $0$, in which case the units of $D$ are $\pm 1$ and $\pm i$. 

\indent Let $\pi\in D$ be an irreducible. 
\begin{thm}
    The residue class ring $D/\pi D$ is a finite field with $N\pi$ elements. 
\end{thm}
\begin{proof}
    The proof of this fundamental result is very similar to Theorem \ref{Theorem: Residue Class Ring D/piD has Npi elements (D=Eisenstein Integers)}. In this proof we use more facts about irreducibles in $D$. 
\end{proof}
\noindent A natural corollary as an analogue to Fermat's Little Theorem easily follows.
\begin{cor}
    If $\pi\nmid \alpha$ then $\alpha^{N\pi-1}\equiv 1\pmod{\pi}$. 
\end{cor}
\noindent The following is very similar to Proposition \ref{Proposition: There exists a unique integer m=0,1,2 analogue to FLT}, and the proof is ultimately identical with consideration for $\ZZ[i]$ instead. 
\begin{prop} \label{Proposition: There exists a unique integer m=0,1,2,3 analogue to FLT biquadratic}
    If $\pi\nmid\alpha$ and the ideal $(\pi)\neq (1+i)$ (the significance of this is due to the fact that $1+i$ is irreducible in $D$), then there exists a unique integer $j=0,1,2,3$ such that 
    \begin{equation*}
        \alpha^{\frac{N\pi-1}{4}}\equiv i^{j}\pmod{\pi}.
    \end{equation*}
\end{prop}
\noindent As such, the biquadratic character is defined as follows. 
\begin{defn}
    If $\pi$ is irreducible and $N\pi\neq 2$, for $\pi\nmid \alpha$ the biquadratic character of $\alpha$ is defined as $\chi_{\pi}(\alpha)=i^{j}$ where the value of $j$ is determined exactly by Proposition \ref{Proposition: There exists a unique integer m=0,1,2,3 analogue to FLT biquadratic}. If $\pi|\alpha$ then we define $\chi_{\pi}(\alpha)=0$. 
\end{defn}

\noindent With this definition we can state biquadratic reciprocity. 

\begin{thm}[The Law of Biquadratic Reciprocity]
    Let $\pi,\lambda\in D$ be relatively prime primary elements. Then
    \begin{equation*}
        \chi_{\pi}(\lambda)\chi_{\lambda}(\pi) = (-1)^{\frac{N\lambda - 1}{4}\cdot \frac{N\pi - 1}{4}}.
    \end{equation*}
\end{thm}

\noindent While the statement of biquadratic reciprocity is more complex than cubic reciprocity, the mechanisms in its proof are fundamentally identical to that of cubic reciprocity. However, to full prove this result, it is necessary to prove consistently more results regarding Jacobi sums than we have in this paper. Full coverage of biquadratic reciprocity and its details can be found in chapter 9 of \cite{ireland1990classical}.

\subsection{Higher Reciprocity}
\noindent As mentioned in the introduction, Eisenstein reciprocity was the first generalized reciprocity law. However, before obtaining generalized reciprocity, many attempts were made to generalize reciprocity beyond cubic and biquadratic reciprocity. Gauss (who was the first to state potential results for higher reciprocity), Jacobi, and Eisenstein made multiple attempts to no avail. In 1839, Jacobi stated - without proof - special cases of higher reciprocity for 5th, 8th, and 12th degrees, but he was unable to consider more general laws; see \cite{Jacobi1839}. The reason for the difficulty in generalizing reciprocity laws beyond 3rd and 4th degrees is that most sets of complex integers adjoined with increasing roots of unity fail to contain an analogue to the Euclidean Algorithm and do not form a UFD. The cases we have considered in this paper both rely on the fact that $\ZZ[\omega]$ and $\ZZ[i]$ have the aforementioned properties. The existence of these properties allows us to consider their residue class rings and proceed from there.  

\newpage

\section*{Acknowledgements}
    \noindent The author would like to thank Dr. Tamar Avineri at the North Carolina School of Science and Mathematics for being a wonderfully supportive reviewer for this paper throughout the revision process; the author would also like to thank Dr. Tamar Avineri for insightful and useful conversations regarding mathematics and number theory. The author would further like to thank Dr. Simon Rubinstein-Salzedo for helpful insights and suggestions on numerous proofs throughout this paper. Finally, the author would like to thank Dr. Frank Thorne at the University of South Carolina for agreeing to endorse him on the arXiv.

\bibliographystyle{alpha}
\bibliography{biblio.bib}

\end{document}